\documentclass[12pt,a4paper,reqno]{amsart}
\usepackage{cite}
\usepackage[alphabetic]{amsrefs}

\usepackage{amsmath,amssymb,amsfonts,amsthm,epsfig,graphicx}
\usepackage{graphicx,color}
\usepackage{hyperref}
\usepackage{url}

\numberwithin{equation}{section}
\numberwithin{figure}{section}

\renewcommand{\leq}{\leqslant}
\renewcommand{\le}{\leqslant}

\renewcommand{\ge}{\geqslant}
\renewcommand{\epsilon}{\varepsilon}
\newcommand{\e}{\varepsilon}
\newcommand{\betxxalf}{{{\alpha}}}
\newcommand{\alphxx}{{{\alpha}}}
\newcommand{\alphyy}{{\tilde{\alpha}}}

\newtheorem{theorem}{Theorem}[section]
\newtheorem{lemma}[theorem]{Lemma}
\newtheorem{corollary}[theorem]{Corollary}
\newtheorem{proposition}[theorem]{Proposition}
\newtheorem*{theorem*}{Theorem}
\newtheorem*{corollary*}{Corollary}

\def\R{\mathbb R}
\def\N{\mathbb N}

\allowbreak
\allowdisplaybreaks

\begin{document}

\title[Nonlocal minimal graphs in the plane
are generically sticky]{Nonlocal minimal graphs in the plane
\\ are generically sticky}

\author[Serena Dipierro,
Ovidiu Savin, and Enrico Valdinoci]{Serena Dipierro${}^{(1)}$\and
Ovidiu Savin${}^{(2)}$
\and
Enrico Valdinoci${}^{(1)}$}

\maketitle

{\scriptsize \begin{center} (1) -- Department of Mathematics and Statistics\\
University of Western Australia\\ 35 Stirling Highway, WA6009 Crawley (Australia)\\
\end{center}
\scriptsize \begin{center} (2) --
Department of Mathematics\\
Columbia University\\
2990 Broadway, NY 10027
New York (USA)
\end{center}
\bigskip

\begin{center}
E-mail addresses:
{\tt serena.dipierro@uwa.edu.au},
{\tt savin@math.columbia.edu},
{\tt enrico.valdinoci@uwa.edu.au}
\end{center}
}
\bigskip\bigskip
\begin{abstract}
We prove that nonlocal minimal graphs in the plane
exhibit generically stickiness effects and boundary discontinuities.
More precisely, we show that if
a nonlocal minimal graph in a slab is continuous up to the boundary,
then arbitrarily small perturbations of the far-away data 
produce boundary discontinuities.

Hence, either a nonlocal minimal graph is discontinuous at the boundary, or a small perturbation of the prescribed conditions produces
boundary discontinuities.

The proof relies on a sliding method combined with a fine boundary regularity
analysis,
based on a discontinuity/smoothness alternative.
Namely, we establish that nonlocal minimal graphs are either discontinuous
at the boundary or their derivative is H\"older continuous up to the
boundary. In this spirit, we prove that the boundary
regularity of nonlocal minimal graphs in the plane ``jumps'' from
discontinuous to~$C^{1,\gamma}$, with no intermediate possibilities
allowed.

In particular, we deduce that the nonlocal curvature equation
is always satisfied up to the boundary.

As an interesting byproduct of our analysis, one obtains
a detailed understanding of the ``switch'' between
the regime of continuous (and hence differentiable) nonlocal minimal
graphs to that of discontinuous (and hence with differentiable inverse)
ones.
\end{abstract}

\section{Introduction}

Nonlocal minimal surfaces have been introduced in~\cite{MR2675483}
with the aim of comprising long range effects in a variational problem resembling
the one arising from the minimization of the classical perimeter functional
and of modeling concrete problems in which
remote interactions play a decisive role.

Specifically, given~$s\in(0,1)$, one considers the long range interaction
of two disjoint (measurable) sets~$F$, $G\subseteq\R^n$ given by
\begin{equation}\label{FeG} {\mathcal{I}}_s(F,G):=\iint_{F\times G}\frac{dx\,dy}{|x-y|^{n+s}}.\end{equation}
For a bounded reference domain~$\Omega$ with Lipschitz boundary,
one also defines the $s$-perimeter of a set~$E\subseteq\R^n$ in~$\Omega$ as
$$ {\rm Per}_s(E;\Omega):=
{\mathcal{I}}_s(E\cap\Omega,E^c\cap\Omega)+
{\mathcal{I}}_s(E\cap\Omega,E^c\cap\Omega^c)+
{\mathcal{I}}_s(E\cap\Omega^c,E^c\cap\Omega),$$
where the superscript ``$c$'' denotes the complementary set in~$\R^n$.

One says that~$E$ is~$s$-minimal in~$\Omega$ if~${\rm Per}_s(E;\Omega)<+\infty$ and
\begin{equation}\label{MINIMIZZ} {\rm Per}_s(E;\Omega)\le {\rm Per}_s(E';\Omega)\end{equation}
for every~$E'\subseteq\R^n$ such that~$E'\cap\Omega^c=E\cap\Omega^c$.
\medskip

An intense research activity was recently focused on $s$-minimal sets.
Among the many topics covered, such a research took into account:
\begin{itemize}
\item {\em Asymptotics:} as $s\nearrow1$, the $s$-perimeter recovers
the classical perimeter, see~\cites{MR1945278, MR1942130, MR2782803, MR2765717, MR3161386}, and as~$s\searrow0$ the problem is related to Lebesgue
measure, see~\cites{MR1940355, MR3007726}.
\item {\em Interior regularity:} $s$-minimal sets
have $C^\infty$ boundary when~$n=2$, and also when~$n\le7$ provided that~$s\in(1-\e_0,1)$
for a sufficiently small~$\e_0\in(0,1)$, see~\cites{MR3090533, MR3107529, MR3331523}.
One can also consider regularity problems for stable, rather than minimal, objects, see~\cites{BV, CCC}.
\item {\em Bernstein property:} $s$-minimal sets with
a graphical structure are necessarily half-spaces
in dimension~$n\le3$, and also in dimension~$n\le8$
provided that~$s\in(1-\e_0,1)$
for a sufficiently small~$\e_0\in(0,1)$, see~\cites{MR3680376, FAR}.
More generally, one can prove similar results
assuming only that
some partial derivatives of the graph are bounded from either above or below,
see~\cite{FMOS}.
\item {\em Isoperimetric problems:} minimizing the $s$-perimeter under suitable volume
conditions naturally leads to a number of fractional isoperimetric problems, see~\cites{MR2469027, MR2799577, MR3322379, MR3412379}.
\item {\em Growth at infinity:} if an $s$-minimal set is the graph of a function~$u$,
then the gradient of~$u$ is bounded in the interior of a ball by a power of its oscillation,
see~\cite{COZCAB}. 
{F}rom this and~\cite{MR3331523}, one obtains
also the~$C^\infty$-regularity of~$u$.
\item {\em Connection to the fractional Allen-Cahn equation:} minimizers
of long range phase
transition models and of fractional Allen-Cahn energy functionals
approach, at a large scale, nonlocal minimal surfaces~\cites{MR2948285, MR3133422, MR3900821}.
This fact makes geometric techniques available in the study
of the symmetry properties of the solutions of the fractional Allen-Cahn
equation and for nonlinear boundary reaction equations,
in the spirit of a classical conjecture by Ennio De Giorgi, see~\cites{MR2177165, MR2498561, MR2644786, MR3148114, MR3280032, MR3610941, DEGIO, MR3812860, SAV2, FSERRA, GUI}.
\item {\em Connection with spin models in statistical mechanics:} 
ground states for long-range Ising models and nonlocal minimal surfaces
approximate each other in a suitable $\Gamma$-convergence setting, see~\cite{MR3652519}.
\item {\em Free boundary problems:} Several new free boundary problems have
been studied by taking into account the nonlocal perimeter as interfacial energy, see~\cites{MR3390089, MR3427047, MR3678490, MR3712006}.
\item {\em Surfaces of constant nonlocal mean curvature:} minimizers
of the fractional perimeter satisfy a suitable Euler-Lagrange equation, which can be written in the form
$$ \int_{\R^n} \frac{\chi_{E^c}(y)-\chi_E(y)}{|x-y|^{n+s}}\,dy=0,\qquad x\in(\partial E)\cap\Omega.$$
In analogy with the classical case, the left hand side of this equation can be considered
as a nonlocal mean curvature (or simply a nonlocal curvature when~$n=2$).
It is natural to consider curves, surfaces and hypersurfaces
with vanishing, or prescribed, nonlocal mean curvature, see e.g.~\cite{MR3485130, MR3770173, MR3744919, MR3836150, MR3881478} for
a number of results in this direction.
\item {\em Geometric flows:} In analogy with the classical case, one can also consider
geometric evolution equations of nonlocal type, such as the evolution of a hypersurface
with normal velocity equal to the nonlocal mean curvature,
see e.g.~\cites{MR2487027, MR2564467, MR3401008, MR3713894, MR3778164, SAEZ}.
This type of geometric motions also appears as a limit of discrete
heat flows and can be seen as a toy model for the evolution
of cellular automata, with potential applications in population dynamics. 
\item {\em Nonlocal capillarity problems:} Nonlocal interactions as the ones in~\eqref{FeG}
have been also exploited to model capillarity phenomena
in which the shape of the droplets are influenced by long range
interactions, see~\cites{MR3717439, MR3707346}:
in particular, in this context, one can describe the contact
angle between the droplet and the container in terms
of a suitable nonlocal Young's Law.
\end{itemize}
Furthermore, nonlocal perimeters and related fractional operators
have been studied also from the numerical point of view,
also due to their flexibility in image reconstruction theory, see e.g.~\cites{MR3413590, MR3893441, NOCHETTO}.

Hence, in general, the
nonlocal perimeter functional provides
a burgeoning topic of research which is
experiencing an intense
activity in many directions, involving mathematicians with different backgrounds
and combining different methodologies coming from geometric analysis, partial differential equations,
geometric measure theory, calculus of variations and functional analysis. 

Since now,
all the problems covered in this framework, the methods exploited
and the results obtained have shown {\em
significant differences} with respect to
their classical analogues, and the new features provided by the nonlocal aspect of
the problem turned out to play a very major role.

We also refer to~\cites{MR3588123, MR3824212} for recent surveys on nonlocal minimal
surfaces and related topics.\medskip

Going back to the minimization problem in~\eqref{MINIMIZZ},
for unbounded domains~$\Omega$, the $s$-perimeter of many ``interesting sets''
can become unbounded. Nevertheless, one can make sense of
the minimization procedure by saying that~$E$ is
locally~$s$-minimal in a (possibly unbounded) domain~$\Omega$
if~$E$ is $s$-minimal in every bounded and Lipschitz domain~$\Omega'\Subset\Omega$
(see Section~1.3 in~\cite{MR3827804} for additional details
on these minimality notions). 

In particular, in this way, one
can take into account the local minimization problem on cylindrical domains
of the form~$B\times\R$, where~$B\subset\R^{n-1}$
is a bounded set with smooth boundary. This setting naturally comprises
the one of ``graphs'', i.e. sets which happen to be the subgraph of a certain function.
The graphical setting was studied in detail in~\cite{MR3516886},
establishing that a locally~$s$-minimal
set which is graphical outside a cylinder is
necessarily graphical over the entire space.

In this setting, the locally $s$-minimal set
is described by a uniformly continuous graph inside the cylinder
which can exhibit boundary discontinuities of jump type, that is
the boundary datum is not necessarily attained continuously
-- even for smooth and compactly supported data, as shown by an example in~\cite{MR3596708}.

In jargon, $s$-minimal sets with graphical properties are called $s$-minimal graphs,
and the boundary discontinuity phenomenon is known with the name of ``stickiness''
(meaning that the interior $s$-minimal surface has to stick at the walls of the cylinder
to attain its exterior datum). 
\medskip

Of course, this stickiness phenomenon
is a {\em purely nonlocal feature}, since classical minimal surfaces leave convex domain
in a transversal fashion, and it seems to be a {\em very distinctive phenomenon
for nonlocal minimal surfaces that is not shared by other problems
of fractional type} (e.g., solutions of fractional Laplace equations
do not exhibit jumps at the boundary).
\medskip

The main goal of this article is to show that this stickiness phenomenon
and the corresponding boundary discontinuity for nonlocal minimal
graphs in the plane, as introduced in~\cites{MR3516886, MR3596708},
is indeed a ``generic'' feature. More precisely, we show that
either a nonlocal minimal
graph in the plane is boundary discontinuous, or there is an arbitrarily small
perturbation of its exterior data which produces a boundary discontinuous
nonlocal minimal graph. In this sense, boundary continuity is an ``unstable''
property of nonlocal minimal graphs, since it is destroyed by arbitrarily small perturbations,
and the stickiness phenomenon holds true ``essentially'' for all prescribed exterior data.
Our precise result is the following:

\begin{theorem}[Genericity of the stickiness phenomenon]\label{GENER}
Let~$\alphxx\in(s,1)$ and~$\Omega:=(0,1)\times\R$.
Let~$v\in C^{1,\alphxx}(\R)$.
Let~$\varphi\in C^{1,\alphxx}(\R,[0,+\infty))$
be not identically zero,
with~$\varphi=0$ in~$(-d,1+d)$, for some~$d>0$.

Let~$u:\R\times[0,+\infty)\to\R$ be defined, for all~$t\ge0$, by
$$ u(x_1,t):=v(x_1)+t\varphi(x_1)\qquad{\mbox{
if~$x_1\in\R\setminus(0,1)$,}}$$
and, for~$x_1\in(0,1)$, by requiring that
the subgraph
\begin{equation}\label{etga}
E_t:=\{ x=(x_1,x_2)\in\R^2 {\mbox{ s.t. }} x_2<u(x_1,t)\}\end{equation}
is locally $s$-minimal in~$\Omega$.

Assume that
\begin{equation}\label{TB0}
\lim_{{x_1\searrow0}} u(x_1,0)=v(0).\end{equation}
Then, for any~$t>0$,
\begin{equation}\label{TB}
\limsup_{{x_1\searrow0}} u(x_1,t)>v(0).\end{equation}
\end{theorem}

Concerning the definition of $u(x_1,t)$ in Theorem~\ref{GENER},
we recall that any $s$-minimal set given by exterior data with a graphical structure
has also a graphical structure, thanks to~\cite{MR3516886},
therefore the set~$E_t$ in~\eqref{etga} is indeed a subgraph.\medskip

The proof of Theorem~\ref{GENER} relies on an auxiliary boundary regularity
result, that we now state. This result seems to be very interesting in itself, since
it rules out ``intermediate'' pathologies in the regularity theory.
Namely: 
\begin{itemize}
\item On the one hand,
nonlocal minimal graphs {\em can well be discontinuous
at the boundary} (as the example in~\cite{MR3596708}).
\item On the other hand, we prove that
{\em if nonlocal minimal graphs happen to be continuous
at the boundary, then they are necessarily differentiable},
and the derivative is H\"older continuous.\end{itemize}
In addition, the H\"older exponent can be explicitly determined,
and it will be {\em sufficiently large} for concrete applications.

In particular, we establish that {\em no halfway boundary regularity}
is possible for nonlocal minimal graphs: they cannot be merely continuous,
or H\"older, or Lipschitz, since the absence of boundary jumps is sufficient
to differentiability and H\"older regularity of the derivative.
\medskip

The precise statement that we prove is the following:

\begin{theorem}[Continuity implies differentiability]\label{BR}
Let~$\betxxalf\in(s,1)$, $u:\R\to\R$, with~$u\in C^{1,\betxxalf}([-h,0])$
for some~$h\in(0,1)$, and
$$ E:=\{ (x_1,x_2)\in\R^2 {\mbox{ s.t. }} x_2<u(x_1)\}.$$
Assume that~$E$ is locally $s$-minimal in~$(0,1)\times\R$.
Suppose also that
\begin{equation}\label{GIAJd}\lim_{x_1\searrow0}u(x_1)=
\lim_{x_1\nearrow0}u(x_1).\end{equation}
Then, $u\in C^{1,\gamma}([-h,1/2])$, with
$$\gamma:=\min\left\{\betxxalf,\,\frac{1+s}2\right\}.$$
\end{theorem}

We think that Theorem~\ref{BR} possesses some surprising features.
First of all, at a first glance,
the boundary regularity of Theorem~\ref{BR} seems ``too good
to be true'' when compared with the case
of nonlocal linear equations, in which solutions are in general not better
than H\"older continuous at the boundary. In this sense,
the regularity obtained in Theorem~\ref{BR} arises from the combination
of two distinctive properties, namely an improvement of flatness method,
which is specific for nonlocal minimal surfaces, and that we perform here
at boundary points, and the higher order boundary regularity for linear
equations. The first of these two ingredients provides a differentiability
result, and only after this the second ingredient comes into play.
Namely, we will exploit the linear theory only after having determined
that the linear equation is a ``very good
approximation'' of the nonlocal curvature equation ``with respect to its own
tangent direction''. In this way, having already established a differentiability result
by the improvement of flatness method, the linear theory necessarily forces
the first term in the boundary expansion to vanish, and hence
the second term of the expansion becomes
representative of the boundary regularity
(this justifies the exponent~$1+\frac{1+s}{2}$ in Theorem~\ref{BR},
since the linearized equation in this case is related to the fractional Laplacian
of order~$\frac{1+s}{2}$, and the exponent~$1+\frac{1+s}{2}$
is precisely the one arising from the second term in the boundary expansion).\medskip

Another very intriguing feature given by
Theorem~\ref{BR} is that
the switch 
between ``non-sticky'' and ``sticky'' nonlocal minimal graphs
is continuous in~$L^\infty$ but not in~$C^1$. That is,
if one considers a nonlocal minimal graph which is continuous
up to the boundary, then it is in fact $C^{1,\frac{1+s}2}$-regular
by Theorem~\ref{BR}, and, according to Theorem~\ref{GENER}, a small perturbation of the exterior data
makes this $C^{1,\frac{1+s}2}$-graph switch to a discontinuous graph. If the perturbation is small,
the two graphs are close to each other, nevertheless their boundary derivative is very different,
since in the unperturbed case the graph detaches in a $C^{1,\frac{1+s}2}$-way
from any prescribed tangent direction, while in the 
perturbed case the graph detaches in a $C^{1,\frac{1+s}2}$-way
from the vertical direction (in particular, any small perturbation of
the exterior data makes the boundary derivative
pass suddenly from a given, possibly zero, value to infinity).\medskip

\begin{figure}[h]
\centering
\includegraphics[width=7.5 cm]{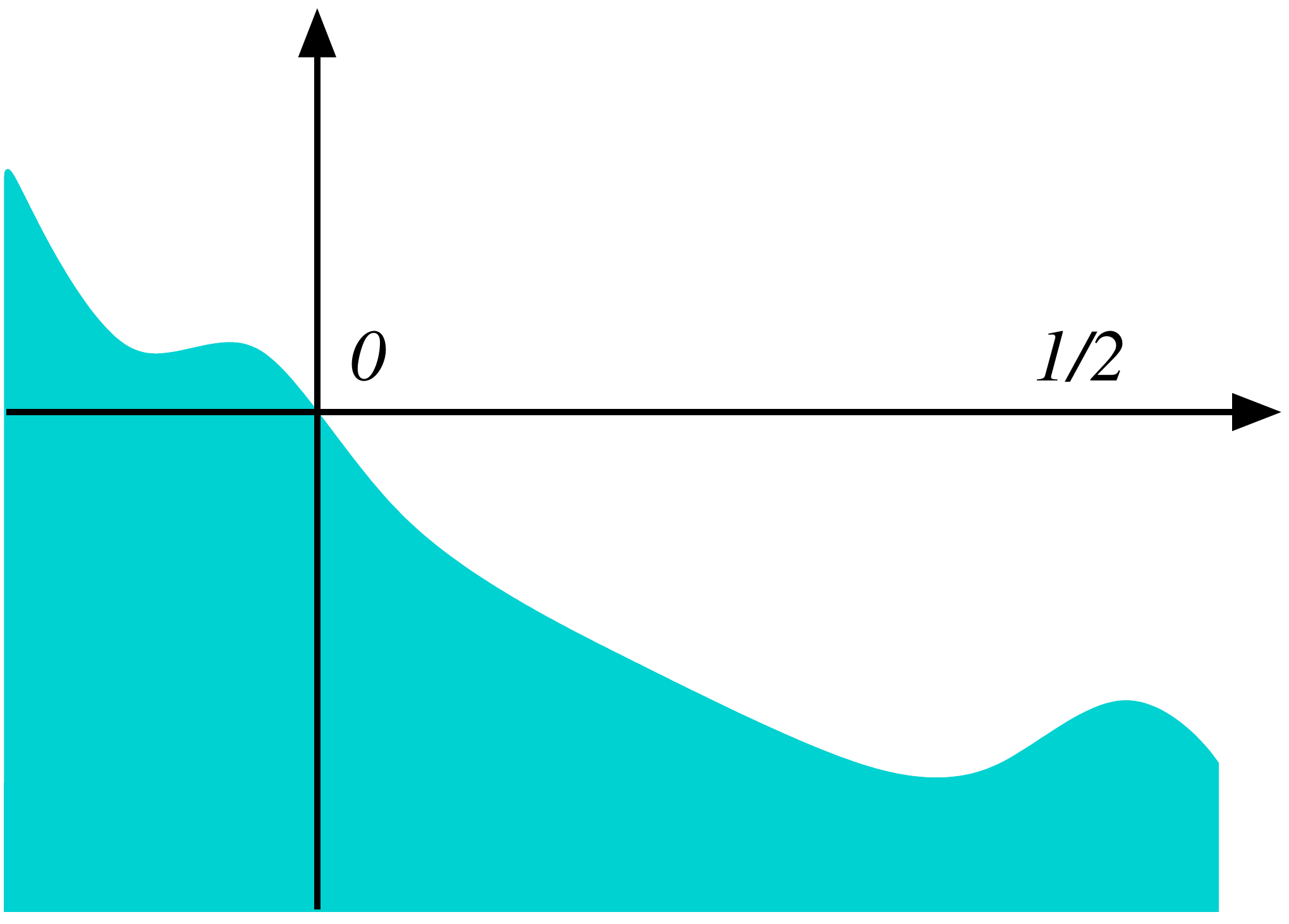}$\qquad$
\includegraphics[width=7.5 cm]{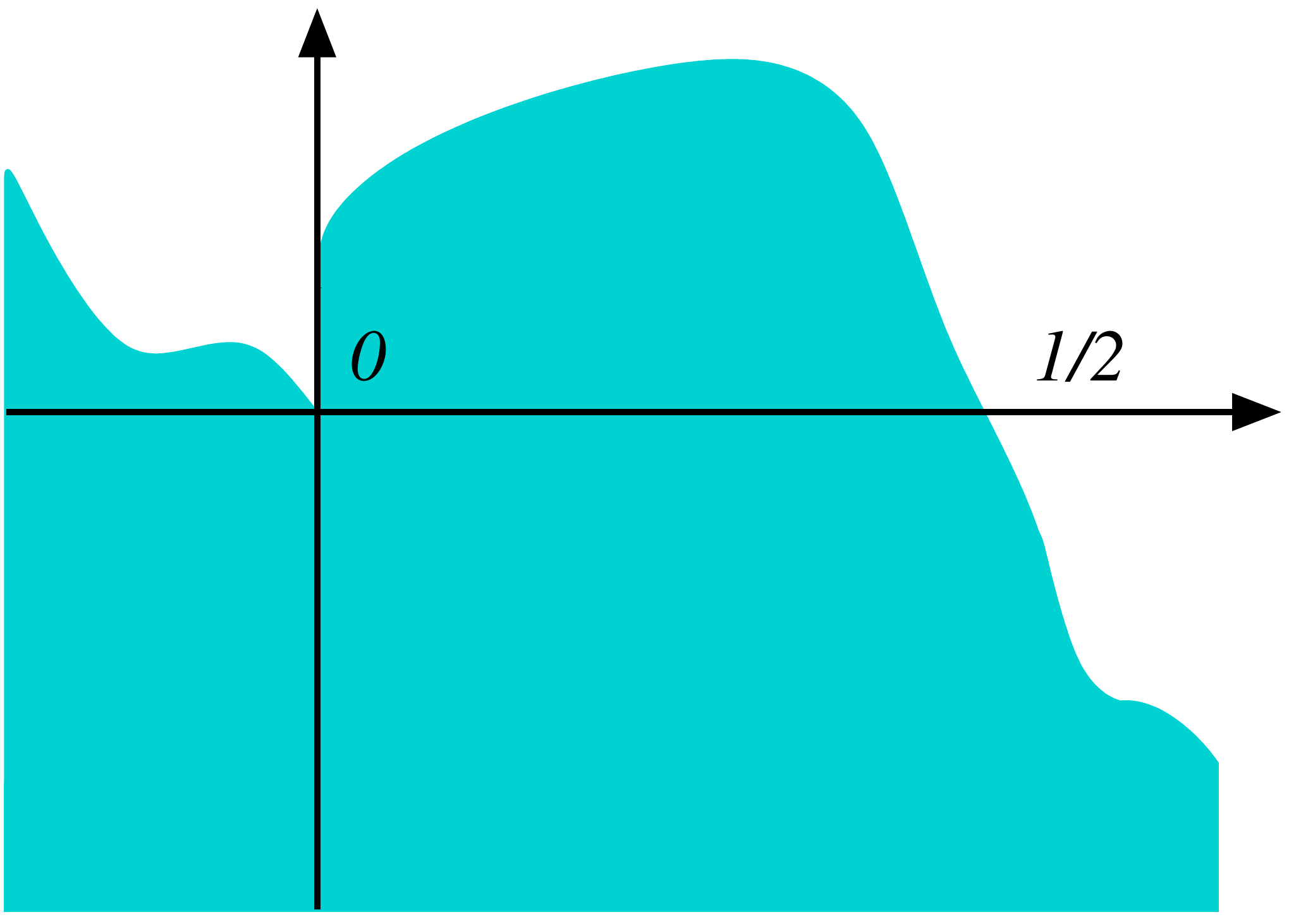}
\caption{\footnotesize\it The alternative boundary behavior in Corollary~\ref{CURV}:
(i) on the left, differentiability up to the boundary, (ii) on the right, stickiness and  boundary
discontinuity. No intermediate option.}
\label{FIGSCAL-ALTER}
\end{figure}

More specifically, an interesting consequence of Theorem~\ref{BR} and of the results in~\cite{MR3532394}
is that nonlocal minimal graphs in the slab are $C^{1,\frac{1+s}2}$-curves
from the interior
up to the boundary, and the following alternative holds: either they are
global $C^{1,\frac{1+s}2}$-graphs, or they exhibit boundary discontinuities
and the graphicality direction becomes horizontal near the boundary
(i.e., in this case, it is the inverse of the solution~$u$, denoted
by~$u^{-1}$, that is significant for the boundary regularity
of the curve describing the nonlocal minimal graph in the slab). 
This alternative is described in Figure~\ref{FIGSCAL-ALTER}.
The precise result is the following:

\begin{corollary}[Geometric regularity up to the boundary]\label{CURV}
Let~$u:\R\to\R$, with~$u\in C^{1,\frac{1+s}2}([-h,0])$
for some~$h\in(0,1)$, and
$$ E:=\{ (x_1,x_2)\in\R^2 {\mbox{ s.t. }} x_2<u(x_1)\}.$$
Assume that~$E$ is locally $s$-minimal in~$(0,1)\times\R$.

Then, $\overline{(\partial E)\cap ((0,1)\times\R)}$
is a~$C^{1,\frac{1+s}2}$-curve.

Moreover, the following alternative holds:
\begin{itemize}
\item[(i)] either
$$ \lim_{x_1\searrow0}u(x_1)=\lim_{x_1\nearrow0}u(x_1)$$
and~$u\in C^{1,\frac{1+s}2}\left(\left[-h,\frac12\right]\right)$,
\item[(ii)] or
$$ \ell:=\lim_{x_1\searrow0}u(x_1)\ne\lim_{x_1\nearrow0}u(x_1)$$
and there exists~$\mu>0$ such that~$u^{-1}\in
C^{1,\frac{1+s}2}([\ell-\mu,\ell+\mu])$.
\end{itemize}
\end{corollary}

Interestingly, Corollary~\ref{CURV} states that the geometric regularity
of nonlocal minimal graphs is the same as that for obstacle-type problems, see~\cite{MR3532394},
since the curve $\overline{(\partial E)\cap ((0,1)\times\R)}$
always detaches from its tangent direction in a global
$C^{1,\frac{1+s}2}$-way. This is in principle absolutely not obvious
and, to prove it, one cannot use directly the obstacle-type results
by considering the tangent line as an obstacle for the curve itself, since
one does not know that such a curve always lies at one side of its tangent line,
and there are no general results in the literature dealing with convexity
properties of nonlocal minimal graphs.\medskip

Another important consequence of Theorem~\ref{BR} is that planar
nonlocal minimal graphs satisfy the Euler-Lagrange equation
(i.e. the vanishing nonlocal curvature equation) not only at interior boundary
points, but also at all the boundary points that are accessible as limits
of interior boundary points. Interestingly, this statement holds true
independently on the stickiness phenomenon and it is thus valid
in any situation, making it a cornerstone towards the proof of
Theorem~\ref{GENER}. The precise result that we have is the following:

\begin{theorem}[Pointwise nonlocal curvature equation up to the boundary]\label{ELBOU}
Let~$\betxxalf\in(s,1)$, $u:\R\to\R$, with~$u\in C^{1,\betxxalf}([-h,0])$
for some~$h\in(0,1)$, and
$$ E:=\{ (x_1,x_2)\in\R^2 {\mbox{ s.t. }} x_2<u(x_1)\}.$$
Assume that~$E$ is locally $s$-minimal in~$(0,1)\times\R$.
Then
\begin{equation}\label{USEF} \int_{\R^2} \frac{\chi_{\R^2\setminus E}(y)-\chi_{E}(y)}{|x-y|^{2+s}}\,dy=0,\end{equation}
for every~$x=(x_1,x_2)\in\overline{(\partial E)\cap((0,1)\times\R)}$.
\end{theorem}

The strategy to prove Theorem~\ref{GENER} is to prove first Theorem~\ref{BR},
which will in turn lead to Theorem~\ref{ELBOU}. Then, the proof of
Theorem~\ref{GENER} will exploit Theorem~\ref{ELBOU} and a sliding method
based on maximum principle.

The proof of Theorem~\ref{BR} presents an intermediate step (namely, the forthcoming
Theorem~\ref{PROVV:Areg}) which is already a regularity statement
in which continuous solutions are proved to be differentiable.
Nevertheless, this preliminary result is not sufficient to
obtain Theorem~\ref{ELBOU}, and hence Theorem~\ref{GENER},
since the H\"older exponent of Theorem~\ref{PROVV:Areg} is too small
to allow the Euler-Lagrange equation to pass to the limit.
For this, it will be important to enhance the H\"older exponent,
and, as a matter of fact, we conjecture that the H\"older exponent
obtained in Theorem~\ref{BR} is optimal.\medskip

For boundary properties of nonlocal minimal graphs in~$\R^3$,
we refer to~\cite{CRdsafgCR}.\medskip

The rest of this paper is organized as follows.
Section~\ref{SE:1} is devoted to a first and second blow-up analysis at the boundary.
Differently than the approach in the interior, in our case the monotonicity
formula is not available, hence some specific arguments are needed
to replace the classical blow-up methods in our framework,
and this is the reason for which we provided full details of the proofs (deferred to the appendix,
not to interrupt the main line of reasoning).

In Section~\ref{SE:2}, we consider some sliding methods to ``clean''
boundary points for nonlocal minimal graphs which possess trivial blow-up limits.
This will be a pivotal step towards the alternative, provided in
Section~\ref{SALE}, according to which
the stickiness phenomenon and the regularity of the blow-up limit
are the only (mutually excluding) possible alternatives.

Sections~\ref{HAR}, \ref{HAR2}, \ref{HAR3}, \ref{HAR4} and~\ref{HAR5}
address the proof of
Theorem~\ref{BR}. In a sense, the scheme of this proof is classical,
since it relies on a Harnack Inequality (Section~\ref{HAR}),
which will allow us to classify the limits of the vertical rescalings of the
solution (Section~\ref{HAR2}). Nevertheless, differently from the existing
literature, our arguments need to take into account the boundary effects,
which, given the stickiness phenomenon, appear to always be non-negligible 
for what regularity concerns. In addition, in our setting,
one has to introduce a new barrier (Section~\ref{HAR3})
that is capable to detect -- and, in fact, exclude -- merely Lipschitz
singularities at the boundary. This turns out to be a crucial step
in our analysis, leading to a boundary regularity
alternative in which corners directly lead to discontinuities, while, viceversa,
continuity directly leads to differentiability.

The differentiability result is then obtained
via a boundary improvement of flatness, which is specifically
designed for non-sticky points (Section~\ref{HAR4}), and which
will lead us to an enhanced H\"older exponent for the derivative of the solution
and to the completion of Theorem~\ref{BR} (Section~\ref{HAR5}).

Theorems~\ref{ELBOU} and~\ref{GENER} are then proved in Sections~\ref{SDRA6}
and~\ref{SDRA7}, respectively.

The paper ends with some technical appendices which collect some ancillary results
and proofs that are postponed not to break the flow of ideas. More specifically,
Appendix~\ref{TECH} contains the technical proofs
of the statements given in Section~\ref{SE:1} concerning the boundary
blow-up limits, Appendix~\ref{LIN} collects
some auxiliary results
from the linear theory of fractional equations, Appendix~\ref{T5PAO}
recalls some density estimates (extending the interior ones
up to the boundary) and the related uniform
convergence results,
and Appendix~\ref{TED} contains
the proof of a technical statement needed
for improving the H\"older exponent of our main regularity results.

\section{Boundary blow-up analysis}\label{SE:1}

In this section we discuss the blow-up methods for nonlocal minimal graphs.
For concreteness, we stick here to the two-dimensional case,
but the $n$-dimensional analysis in this section would remain completely
unaltered. The proofs of this section are rather technical,
and therefore they have been all {\em deferred to Appendix~\ref{TECH}}.

We consider~$\Omega:=(0,1)\times\R$
and a subgraph~$E$ which is locally $s$-minimal in~$\Omega$.

We assume that the boundary of~$E$
meets the boundary of~$\Omega$ at the origin in a $C^{1,\alpha}$
fashion from outside~$\Omega$,
namely there exists~$r_0>0$ such that~$E\cap \{x_1\in(-r_0,0)\}$
is the subgraph of a $C^{1,\alpha}$ function~$v$, namely
$$ E\cap \{x_1\in(-r_0,0)\}=\{ x=(x_1,x_2)\in\R^2 {\mbox{ s.t. }} x_1\in(-r_0,0)
{\mbox{ and }}x_2<v(x_1)\},$$
with~$v(0)=0$.

We consider a first blow-up sequence, defined, for all~$k\in\N$ with~$k\ge1$,
\begin{equation}\label{Ekdef} E_k:= kE=\{ kx,\;\,x\in E\}.\end{equation}
Differently than the previous literature, the blow-up
sequence in~\eqref{Ekdef}
is centered at boundary points rather than in the interior
(this makes some classical tools such as monotonicity formula
and density estimates not available in this context).
We have the following first blow-up result.

\begin{lemma}\label{primo blow up}
There exists~$E_0\subset\R^2$ such that,
up to a subsequence, we have that
\begin{equation}\label{BL1}
{\mbox{$\chi_{E_k}\to \chi_{E_0}$ in~$L^1_{\rm loc}(\R^2)$.}}\end{equation}
In addition, 
\begin{equation}\label{BL2}
{\mbox{$E_0$ is locally $s$-minimal in~$\{x_1>0\}$}}\end{equation} and
\begin{equation}\label{BL3}
E_0\cap\{x_1<0\}= \left\{ x_2< v'(0)\,x_1\right\}\cap\{x_1<0\}.
\end{equation}
\end{lemma}

Given the setting in Lemma~\ref{primo blow up}, and without
the availability of monotonicity formulas, it is also convenient
to consider a second blow-up sequence, defined, for all~$k\in\N$ with~$k\ge1$,
\begin{equation}\label{EZEROk} E_{0k}:= kE_0=\{ kx,\;\,x\in E_0\}.\end{equation}
The second blow-up procedure has the advantage
that the datum outside~$\{x_1>0\}$ is a cone, which
makes it possible to establish that the limit is a cone (also in~$\{x_1>0\}$),
as next result points out:

\begin{lemma}\label{secondo blow up}
There exists~$E_{00}\subset\R^2$ such that,
up to a subsequence, we have that
\begin{equation}\label{BL1-0}
{\mbox{$\chi_{E_{0k}}\to \chi_{E_{00}}$ in~$L^1_{\rm loc}(\R^2)$.}}\end{equation}
In addition, 
\begin{equation}\label{BL2-0}
{\mbox{$E_{00}$ is locally $s$-minimal in~$\{x_1>0\}$}}\end{equation} and
\begin{equation}\label{BL3-0}
E_{00}\cap\{x_1<0\}= \left\{ x_2< v'(0)\,x_1\right\}\cap\{x_1<0\},
\end{equation}
and~$E_{00}$ is a cone, namely~$tE_{00}=E_{00}$ for all~$t>0$.
\end{lemma}

To avoid technical complications, it is useful
to observe that the second blow-up cone can be obtained by
a direct blow-up of the original set, up to a subsequence:

\begin{lemma}\label{Sempli}
In the notation of Lemma~\ref{secondo blow up}, we have that
$$ {\mbox{$\chi_{E_{k}}\to \chi_{E_{00}}$ in~$L^1_{\rm loc}(\R^2)$,}}$$
up to a subsequence.
\end{lemma}

\section{Sliding methods}\label{SE:2}

In this section we prove that ``narrow $s$-minimal sets
are necessarily void'' (see below for a precise statement).
The proof is based
on the sliding method.

Given~$M$, $h>0$, we let
\begin{equation}\label{Qrr} {\mathcal{Q}}_{M,h}:=(0,M)\times(-h,h).\end{equation}
We also use the short notation
$$ \tilde\chi_A:=\chi_{A^c}-\chi_A.$$

\begin{proposition}\label{pallata}
Let~$\lambda>0$.
There exist~$M_0>1$ and~$\mu_0\in(0,1)$ such that if~$M\ge M_0$ and~$\mu\in(0,\mu_0]$ the following
claim holds true.

Let~$E\subset\R^2$ be~$s$-minimal in~${\mathcal{Q}}_{M,4}$ and such that
\begin{eqnarray} \label{soto}
&& E\cap \{ x_1\in(-M,0)\} \subseteq \{ x_2\le -\lambda x_1\},\\
\label{stretto}
{\mbox{and }}&& E\cap {\mathcal{Q}}_{M,M} \subseteq \{ x_1\in(0,\mu)\}
.\end{eqnarray}
Then, 
\begin{equation}\label{NOAU}
E\cap {\mathcal{Q}}_{\frac{M}2,1}=\varnothing.\end{equation}
\end{proposition}

\begin{figure}[h]
\centering
\includegraphics[width=13.5 cm]{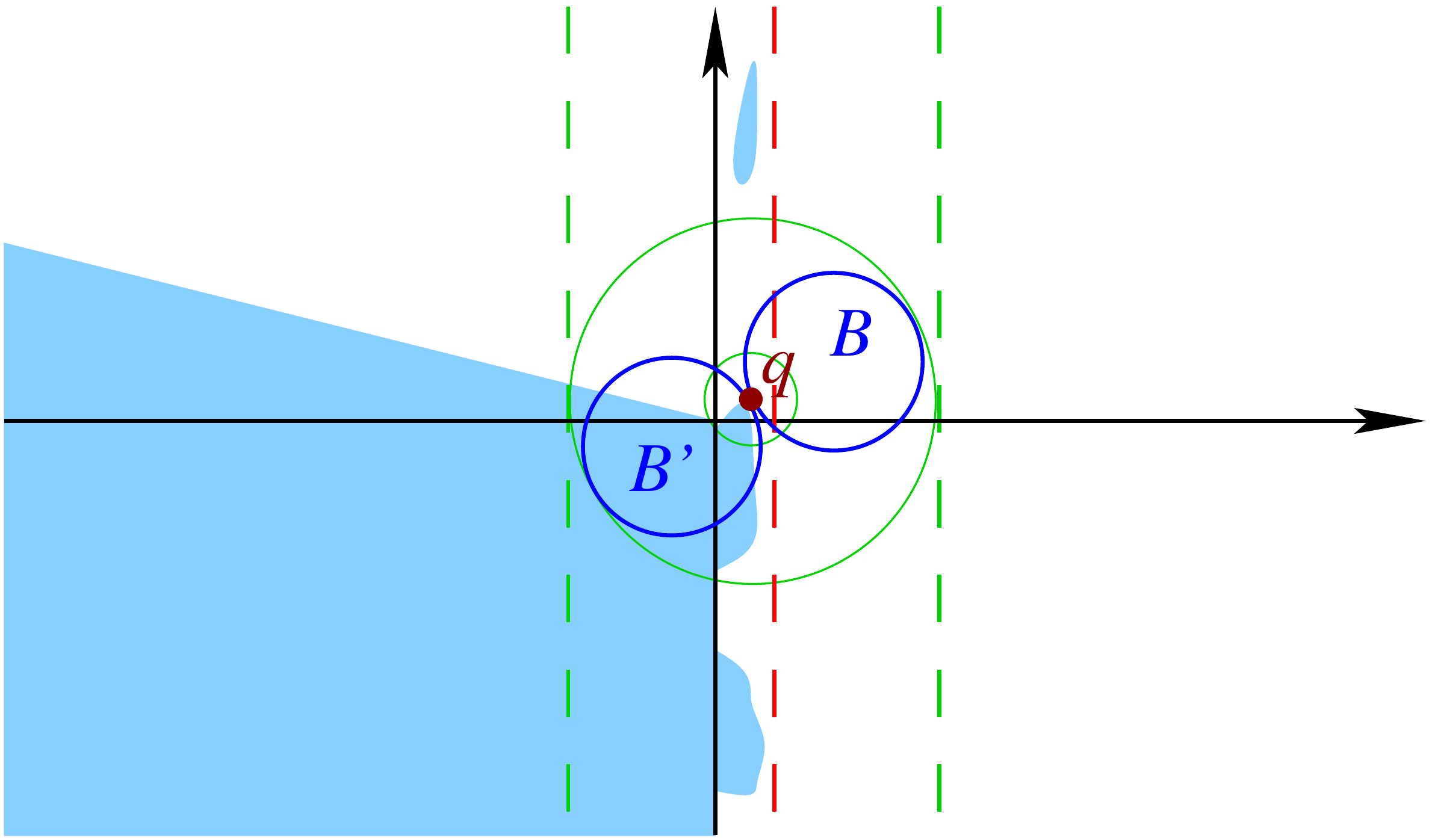}
\caption{\footnotesize\it The geometry involved in the proof of Proposition~\ref{pallata}.}
\label{FIGSCAL}
\end{figure}

\begin{proof} We let~$\vartheta\in(-1,1)$
and~$P_t=(t,\vartheta)\in\R^2$.
We observe that if~$t\ge 2$
then
\begin{eqnarray*} B_1(P_t)&\subseteq& (t-1,t+1)\times(-2,2)\\
&\subseteq &\{x_1>1\},
\end{eqnarray*}
and therefore, by~\eqref{stretto},
\begin{equation*}
B_1( P_t)\cap E\cap{\mathcal{Q}}_{M,M}\subseteq\{x_1>1\}\cap\{x_1\in(0,\mu)\}
=\varnothing.\end{equation*}
We now take
$$ t_\star:=\inf\left\{ t\in\left[1,\frac{M}{8}\right] {\mbox{ s.t. }}
B_1( P_t)\cap E\cap{\mathcal{Q}}_{M,M}
=\varnothing\right\}.$$
We claim that
\begin{equation}\label{SLIDE}
t_\star=1.
\end{equation}
To prove this, we argue by contradiction.
If not, there exist~$t_\star\in\left(1,\frac{M}{8}\right]$ and~$q\in \big(\partial B_1(P_{t_\star})\big)\cap(\partial E)$,
with~$B:=B_1(P_{t_\star})\subseteq E^c$.

We remark that the first coordinate of~$P_{t_\star}$
is equal to~$t_\star$, and therefore
\begin{equation}\label{q1do}
q_1\ge t_\star-1>0,
\end{equation}
and also
\begin{equation}\label{q1do2}
q_1\le t_\star+1\le1+\frac{M}{8}.
\end{equation}
Similarly, we see that
\begin{equation}\label{q2do}
|q_2|\le 2.
\end{equation}
We let~$\nu$ be the exterior normal to~$B$ at~$q$
and~$B':=B_1(P_{t_\star}+2\nu)$. We observe that~$B'$
is the symmetric ball to~$B$ with respect to the tangent plane through~$q$, see Figure~\ref{FIGSCAL}.
As a consequence, 
taking~$\rho:=\sqrt[3]{\mu}$,
since~$B_\rho(q)\cap E\subseteq B^c$,
\begin{equation}\label{9is734}
\int_{B_\rho(q)} \frac{\tilde\chi_E(x)}{|x-q|^{2+s}}dx\ge-
\int_{B_\rho(q)\setminus (B\cup B')} \frac{dx}{|x-q|^{2+s}}\ge- C\rho^{1-s}=- C\mu^{\frac{1-s}{3}},
\end{equation}
for some~$C>0$.

Now, we let
\begin{eqnarray*}&& A_1:= \big(B_1(q)\setminus B_\rho(q)\big)\cap\{ |x_1-q_1|\le 2\mu\}
\\{\mbox{and}}\qquad&&
A_2:= \big(B_1(q)\setminus B_\rho(q)\big)\cap\{ |x_1-q_1|> 2\mu\}.\end{eqnarray*}
We observe that
\begin{equation}\label{7362}
\int_{A_1} \frac{\tilde\chi_E(x)}{|x-q|^{2+s}}dx\ge
-\int_{ B_1(q)\cap \{ |x_1-q_1|\le 2\mu\} }\frac{dx}{\rho^{2+s}}\ge-\frac{C\mu}{\rho^{2+s}}=-C\mu^{\frac{1-s}{3}}.
\end{equation}
Now we let
$$ {\mathcal{Q}}^\star:=(q_1-1,q_1+1)\times \left( -\frac{M}{4},\frac{M}{4}\right),$$ and we
claim that
\begin{equation}\label{Pija}
E\cap \big({\mathcal{Q}}^\star\setminus B_\rho(q)\big)\cap\{ x_1-q_1> 2\mu\}=\varnothing.
\end{equation}
Suppose not and let~$p$ be in this set. Then, by~\eqref{q1do},
$$ p_1>2\mu+q_1>2\mu.$$
Similarly, by~\eqref{q1do2},
$$ p_1<q_1+1\le 2+\frac{M}{8}<M,$$
as long as~$M$ is large enough.
On the other hand, by~\eqref{q2do}, 
$$ |p_2|\le |p_2-q_2|+|q_2|\le \frac{M}{4}+2<M,$$
as long as~$M$ is large enough. These observations give that
$$ p\in E\cap {\mathcal{Q}}_{M,M}\cap \{x_1>2\mu\},$$
which is in contradiction with~\eqref{stretto}, and hence it proves~\eqref{Pija}.

Then, from~\eqref{Pija} we deduce that
$$ E\cap A_2\subseteq \{x_1-q_1<-2\mu\},$$
and consequently
\begin{equation*}
\int_{A_2} \frac{\tilde\chi_E(x)}{|x-q|^{2+s}}dx\ge0.\end{equation*}
Then, combining this inequality with~\eqref{7362}, we conclude that
\begin{equation*}
\int_{B_1(q)\setminus B_\rho(q)} \frac{\tilde\chi_E(x)}{|x-q|^{2+s}}dx\ge-C\mu^{\frac{1-s}{3}}.\end{equation*}
This and~\eqref{9is734} yield that
\begin{equation}\label{Xue}
\int_{B_1(q)} \frac{\tilde\chi_E(x)}{|x-q|^{2+s}}dx\ge-C\mu^{\frac{1-s}{3}},\end{equation}
up to renaming~$C$.

Furthermore, by~\eqref{Pija},
\begin{equation*}
\int_{{\mathcal{Q}}^\star\setminus B_1(q)} \frac{\tilde\chi_E(x)}{|x-q|^{2+s}}dx\ge
0,\end{equation*}
and therefore
\begin{equation}\label{7u1414373}
\int_{\{|x_1-q_1|<1\}\setminus B_1(q)} \frac{\tilde\chi_E(x)}{|x-q|^{2+s}}dx\ge-
\int_{\R^2\setminus B_{M/8}(q)} \frac{dx}{|x-q|^{2+s}}\ge -\frac{C}{M^s}.
\end{equation}
Now, we observe that
\begin{equation}\label{srta}
E\cap\left\{x_1-q_1>1\right\}\subseteq \R^2\setminus B_{M/8}(q).
\end{equation}
Indeed, if not, recalling~\eqref{q1do}, \eqref{q1do2} and~\eqref{q2do}
we see that there exists~$p=(p_1,p_2)\in E$ with
\begin{eqnarray*}&& p_1<q_1+\frac{M}{8}\le 1+\frac{M}{8}+\frac{M}{8}<M,\\
&& p_1>q_1+1>1>\mu,\\
&&|p_2|\le|p_2-q_2|+|q_2|<
\frac{M}{8}+1+\frac{M}{8}<M,
\end{eqnarray*}
and these inequalities produce a contradiction with~\eqref{stretto}, thus proving~\eqref{srta}.

Then, by~\eqref{soto} and~\eqref{srta},
$$ \int_{B_{M/8}(q)\cap\{|x_1-q_1|>1\}} \frac{\tilde\chi_E(x)}{|x-q|^{2+s}}dx\ge c,
$$
for some~$c>0$, and, as a consequence,
\begin{equation*}\begin{split}&
\int_{\{|x_1-q_1|>1\}\setminus B_1(q)} \frac{\tilde\chi_E(x)}{|x-q|^{2+s}}dx=
\int_{\{|x_1-q_1|>1\}} \frac{\tilde\chi_E(x)}{|x-q|^{2+s}}dx\\ &\qquad\qquad\ge
c-\int_{\R^2\setminus B_{M/8}(q)} \frac{dx}{|x-q|^{2+s}}\ge c-\frac{C}{M^s},\end{split}
\end{equation*}
for some~$C>0$.
This and~\eqref{7u1414373} give that
$$ \int_{\R^2\setminus B_1(q)} \frac{\tilde\chi_E(x)}{|x-q|^{2+s}}dx\ge c-\frac{C}{M^s},$$
up to renaming~$C$, which, combined with~\eqref{Xue}, yields that
$$ \int_{\R^2} \frac{\tilde\chi_E(x)}{|x-q|^{2+s}}dx\ge c-\frac{C}{M^s}-C\mu^{\frac{1-s}{3}}.$$
The latter quantity is strictly positive if~$M$ is sufficiently large and~$\mu$
is sufficiently small, which is in contradiction
with the $s$-minimality of~$E$.

This completes the proof of~\eqref{SLIDE}.
Then, by~\eqref{SLIDE}, we obtain that
$$ \varnothing=\bigcup_{t\ge1}
B_1( P_t)\cap E\cap{\mathcal{Q}}_{M,M}\supseteq (\R\times\{\vartheta\})
\cap E\cap{\mathcal{Q}}_{M,M},$$
from which~\eqref{NOAU} plainly follows since~$\vartheta\in (-1,1)$
is arbitrary.
\end{proof}

As a consequence of Proposition~\ref{pallata}, we have
that $s$-minimal sets which approach 
empty sets in the halfspace are necessarily locally empty:

\begin{corollary}\label{VUOCO}
There exists~$M_0>1$ such that the following claim holds true.
Let~$\e_k$ be an infinitesimal sequence. Let~$\lambda>0$
and~$E_k$ be a sequence of~$s$-minimal sets in~$
{\mathcal{Q}}_{\frac1{\e_k},\frac1{\e_k}}$
such that
$$ E_k\cap \{ x_1<0\} \subseteq \{ x_2\le -\lambda x_1\}.$$
Assume that, for every~$M\ge M_0$,
\begin{equation}\label{8qwdyugcvdggfg723trgf}{\mbox{$
\chi_{E_k}\to 0$ in~$L^1\left(\left(\frac1{2M},2M\right)\times(-2M,2M) \right)$.}}\end{equation} Then, for every~$M\ge M_0$
there exists~$k_M\in\N$ such that if~$k\ge k_M$ then
$$ E_k\cap {\mathcal{Q}}_{M,1}=\varnothing.$$
\end{corollary}

\begin{proof} Fix~$M>0$, to be taken sufficiently large in what follows. First of all, we prove that, if~$k$ is large enough, then
\begin{equation}\label{Ekin}
E_k\cap {\mathcal{Q}}_{M,M}\subseteq\left\{x_1\in\left(0,\frac1M\right)\right\}.\end{equation}
Suppose not. Then there are infinitely
many values of~$k$ for which there exists a point~$p_k=(p_{k1},p_{k2})\in E_k$
with~$p_{k1}\in\left[\frac1M,M\right]$ and~$|p_{k2}|\le M$.
We let~$r:=\frac{1}{10 \,M}$ and we claim that there exists~$q_k\in B_{2r}(p_k)$
and a constant~$c\in\left(0,\frac12\right)$ such that
\begin{equation}\label{palk}
B_{cr}(q_k)\subseteq E_k.
\end{equation}
Indeed, if~$B_{r/2}(p_k)\subseteq E_k$, then~\eqref{palk} holds true with~$q_k:=p_k$.
If instead~$B_{r/2}(p_k)\cap E_k\ne\varnothing$,
there exists a point~$z_k=(z_{k1},z_{k2})\in B_{r/2}(p_k)\cap (\partial E_k)$.
We notice that if~$\zeta=(\zeta_1,\zeta_2)\in B_r(z_k)$ then
\begin{eqnarray*}&&
\zeta_1\ge z_{k1}-r\ge p_{k1}-\frac{3r}{2}\ge\frac{1}{M}-\frac{3r}{2}=\frac{17}{20\,M}>\e_k
\\{\mbox{and }}&&|\zeta_i|\le|z_{ki}|+r\le|p_{ki}|+\frac{3r}{2}\le M+\frac{3}{20\,M}<
\frac1{\e_k}
\end{eqnarray*}
for all~$i\in\{1,2\}$, which implies that~$
B_r(z_k)\subseteq{\mathcal{Q}}_{\frac1{\e_k},\frac1{\e_k}}$,
and therefore we can use the clean ball condition
in~\cite{MR2675483} and conclude the proof of~\eqref{palk}.

Moreover, up to a subsequence, we can assume that~$p_k\to p$,
for some~$p\in \left[\frac1M,M\right]\times[- M,M]$,
and we observe that~$B_{4r}(p)\supseteq B_{cr}(q_k)$
if $k$ is sufficiently large.
Accordingly, by~\eqref{palk}, we find that
$$ \int_{B_{4r}(p)} \chi_{E_k}(x)\,dx\ge\int_{B_{cr}(q_k)} \chi_{E_k}(x)\,dx=
|B_{cr}(q_k)|=c'\,r^n,$$
for some~$c'>0$, which is in contradiction with~\eqref{8qwdyugcvdggfg723trgf},
and hence it proves~\eqref{Ekin}.

In light of~\eqref{Ekin}, we have that~$E_k$ satisfies the assumptions
of Proposition~\ref{pallata}. Consequently, the desired result follows from~\eqref{NOAU},
up to renaming~$M$.
\end{proof}

\section{Alternative on the second blow-up}\label{SALE}

In this section, we analyze the $s$-minimal cone arising from the second blow-up,
as described in Lemma~\ref{secondo blow up}, and we give a sharp
alternative for its behavior.
Roughly speaking, the alternative 
is that either a given $s$-minimal set
exhibits the stickiness phenomenon or its second blow-up
is a half-plane. 

To state this result precisely,
we say that 
a set~$A$ is {\em trivial} in~$B$ if either~$A\cap B=B$ or~$A\cap B=\varnothing$.
Then, we have:

\begin{theorem}\label{P:AL}
Let~$\Omega:=(0,1)\times\R$.
Let~$E$ be an $s$-minimal graph in~$\Omega$.
Suppose that the boundary of~$E$
meets the boundary of~$\Omega$ at the origin and it is $C^1$ near the origin from
outside~$\Omega$. 
%%% Assume also that the tangent line
%%% of~$\partial E$ from outside~$\Omega$ and the tangent line of~$\Omega$
%%% do not coincide, and that the latter is~$\{x_1=0\}$.

Let~$E_{00}$ be the second blow-up cone, as given in Lemma~\ref{secondo blow up}.
Then the following alternative holds true:
\begin{itemize}
\item either~$E_{00}$ is a halfplane,
\item or there exists~$r>0$
such that~$E$ is trivial in~$(0,r)\times(-r,r)$.
\end{itemize}
\end{theorem}

\begin{proof} {F}rom Lemma~\ref{secondo blow up}, it follows
that~$E_{00}$ has a graphical structure
in the sense that if~$p=(p_1,p_2)\in E_{00}$ then~$(p_1,\tau)\in E_{00}$ for all~$\tau\le p_2$
(actually, a similar statement also holds for~$E_0$, using instead Lemma~\ref{primo blow up},
but we focus here on~$E_{00}$).

In the notation of~\eqref{Qrr}, we claim that
\begin{equation}\label{8uJS-x}
{\mbox{if~$E_{00}$ is trivial in~$\{x_1>0\}$, then
there exists~$r>0$
such that~$E$ is trivial in~${\mathcal{Q}}_{r,r}$.}}
\end{equation}
To prove this, let us assume that
\begin{equation*}
E_{00}\cap\{x_1>0\}=\varnothing\end{equation*}
(the case~$E_{00}\cap\{x_1>0\}= \{x_1>0\}$ can be treated similarly).
By Lemma~\ref{Sempli}, we have that~$\chi_{ E_{k}}\to \chi_{ E_{00}}=0$
in~$L^1_{\rm loc}((0,+\infty)\times\R)$, up to a subsequence
(and, in fact, locally in the Hausdorff distance, thanks to Corollary~\ref{Sempli-APP}).
Hence, we can exploit Corollary~\ref{VUOCO} with~$M:=1$,
and find that
\begin{equation*} 
E_{k}\cap {\mathcal{Q}}_{1,1}=\varnothing,\end{equation*}
as long as~$k$ is sufficiently large.
We can thereby conclude that
\begin{equation*} 
E\cap {\mathcal{Q}}_{\frac{1}k,\frac1k}=\varnothing,
\end{equation*}
which completes the proof of~\eqref{8uJS-x}.

Now, to complete the proof of Theorem~\ref{P:AL},
we suppose that~$E$ is not trivial in~${\mathcal{Q}}_{r,r}$:
then the desired claim in Theorem~\ref{P:AL} is established
if we show that~$E_{00}$ is a half-plane. To this end,
we know from~\eqref{8uJS-x} that~$E_{00}$ cannot be trivial in~$\{x_1>0\}$.
Consequently,
since~$E_{00}$ is a cone, in light of Lemma~\ref{secondo blow up},
we can write that
$$ E_{00}\cap\{x_1<0\}=\{x_2<a x_1\}\qquad{\mbox{ and }}\qquad
E_{00}\cap\{x_1>0\}=\{x_2<b x_1\},$$
for some~$a$, $b\in\R$. We claim that
\begin{equation}\label{ab}
a=b.
\end{equation}
Indeed, suppose not.
Then, we compute the nonlocal mean curvature equation at the point~$p:=(1,b)\in\partial E_{00}$,
symmetrizing the integral with respect to the line~$\{x_2=b x_1\}$ and we obtain that
$$ \int_{\R^2} \frac{\chi_{\R^2\setminus E_{00}}(p)-\chi_{E_{00}}(p)}{|x-p|^{2+s}}\,dx\ne0.$$
This is a contradiction with the minimality of~$E_{00}$, and hence it proves~\eqref{ab},
thus completing the proof 
of Theorem~\ref{P:AL}.
\end{proof}

{F}rom the blow-up alternative in Theorem~\ref{P:AL}, we obtain
that an $s$-minimal graph cannot develop boundary corners,
unless it develops boundary stickiness (that is,
no ``nude'' Lipschitz singularity is possible, since
any corner would naturally produce boundary discontinuities). The precise
result that we have is the following:

\begin{corollary}\label{P:AL:IU}
Let~$\Omega:=(0,1)\times\R$.
Let~$E$ be a locally $s$-minimal graph in~$\Omega$.
Suppose that the boundary of~$E$
meets the boundary of~$\Omega$ at the origin and it is $C^1$
near the origin from
outside~$\Omega$. Assume also that the tangent line
of~$\partial E$ from outside~$\Omega$ is of the form~$x_2=\ell x_1$,
with~$\ell\in\R$.

Then:
\begin{itemize}
\item if
\begin{equation}\label{7l29} \{x\in\R^2 {\mbox{ s.t. }} x_1\in(0,\delta)
{\mbox{ and }}x_2<(\ell+\delta) x_1\}\subseteq E\end{equation}
for some~$\delta>0$, then there exists~$\delta'>0$ such that~$(0,\delta')^2\subseteq E$;
\item if
$$ \{x\in\R^2 {\mbox{ s.t. }} x_1\in(0,\delta)
{\mbox{ and }}x_2<(\ell-\delta) x_1\}\supseteq E$$
for some~$\delta>0$, then there exists~$\delta'>0$ such that~$
(0,\delta')\times(-\delta',0)\subseteq \R^2\setminus E$.
\end{itemize}
\end{corollary}

\begin{proof}
We prove the first claim of Corollary~\ref{P:AL:IU}, since the second
one would then follow by considering complementary sets.
If~\eqref{7l29} is satisfied, then the second blow-up~$E_{00}$ in Lemma~\ref{secondo blow up}
would satisfy
\begin{eqnarray*}&&
E_{00}\cap \{x_1<0\}=\{x_2<\ell x_1\}\cap\{x_1<0\}\\
{\mbox{and }}&&E_{00}\cap \{x_1>0\}\supseteq\{x_2<(\ell+\delta) x_1\}\cap\{x_1>0\}.
\end{eqnarray*}
In particular, $E_{00}$ is not a half-plane, and hence the claim
follows from Theorem~\ref{P:AL}.
\end{proof}

It is interesting to notice that the boundary behavior described in Corollary~\ref{P:AL:IU}
is significantly different not only with respect to the case of classical minimal surfaces,
but also with respect to the case of nonlocal capillarity problems.
Indeed, in the nonlocal capillarity theory the minimizers
satisfy uniform density estimates at the boundary, as proved in Theorem~1.7
of~\cite{MR3717439}: instead, the minimizers of the nonlocal perimeters
do not possess similar boundary density estimates
in the domain, exactly in view of the stickiness
phenomenon. In a sense, as we will also see in the forthcoming section
and in Appendix~\ref{T5PAO},
the role of this density estimates in our setting will be recovered by
the mass of the set and of its complement which arises from the exterior datum.

\section{Boundary Harnack Inequality}\label{HAR}

The main goal of this section
is to provide a suitable Harnack Inequality
for~$E$ ``up to the boundary''. The results obtained
will then complement those in Section~6.3
of~\cite{MR2675483}, where similar results where
obtained in the interior of the domain. Interestingly,
the boundary Harnack inequality
is stronger than the one in the interior,
since one is able to decrease the oscillation both from above and from below.
More precisely, the main result of this section is the following:

\begin{lemma}\label{HARNA}
Let~$\alpha\in(0,s)$, $\ell\in\R$, $u:\R\to\R$ and
$$ E:=\{ (x_1,x_2)\in\R^2 {\mbox{ s.t. }} x_2<u(x_1)\}.$$
Assume that~$E$ is locally $s$-minimal in~$(0,1)\times\R$.
Then, there exist 
$\eta\in(0,1)$, $\delta\in(0,1)$ and~$\e_0\in(0,\min\{\eta,\delta\})$,
depending only
on~$\alpha$, $\ell$, and~$s$,
such that if~$\e\in(0,\e_0)$ the following statement holds true.

Let
\begin{equation}\label{koche} k_0:=\left\lceil \frac{-\log_2\e}{\alpha}\right\rceil,\end{equation}
and suppose that
\begin{eqnarray}
\label{allinfi1}&&|u(x_1)-\ell x_1|\le\delta\e (2^k)^{1+\alpha} {\mbox{ for all $x_1\in (-2^{k},0)$,
for all $k\in\{0,\dots, k_0\}$,}}\\
\label{allinfi2}&&|u(x_1)-\ell x_1|\le \e (2^k)^{1+\alpha} {\mbox{ for all $x_1\in (0,2^{k})$,
for all $k\in\{0,\dots, k_0\}$.}}
\end{eqnarray}
Then, 
\begin{eqnarray}\label{sotto0}
&& \sup_{x_1\in(0,\eta)} u(x_1)-\ell x_1\le \e(1-\eta^2),\\
{\mbox{and }}&&
\label{sotto}
\inf_{x_1\in(0,\eta)} u(x_1)-\ell x_1\ge \e(\eta^2-1).
\end{eqnarray}
\end{lemma}

\begin{proof} We use the sliding method of Lemma~6.9
of~\cite{MR2675483}, by taking into account the following complications.
First of all, we do not impose a priori that the touching points
occur in the interior. Secondly, the sliding parabola
is adapted to take into consideration the linear perturbation provided
by the parameter~$\ell$. On the other hand,
dealing with these two complications
will produce an interesting byproduct since
the Harnack inequality that we obtain establishes both
the oscillation improvement from above in~\eqref{sotto0}
and the one from below in~\eqref{sotto} (while in the interior
case one can only prove either one or the other): in our
situation, this additional information will be a consequence
of the regularity of the data outside the domain,
which, in view of~\eqref{allinfi1} produces exterior mass of
both the set and of its complement.

The details go as follows. 
We prove that~\eqref{sotto} holds true
(the proof of~\eqref{sotto0} is similar).
We argue
by contradiction, supposing that~\eqref{sotto}
is violated, hence there exists~$t_\star\in(0,\eta)$
with
\begin{equation}\label{tstaso}
u(t_\star)-\ell t_\star< \e(\eta^2-1).
\end{equation}
We let
$${\mathcal{C}}:=\{ x=(x_1,x_2)\in\R^2 {\mbox{ s.t. }}
x_1\in (-\eta,\eta) {\mbox{ and }} |x_2-\ell x_1|<\e \}.$$
We prove that
\begin{equation}\label{sotto00}
|E\cap{\mathcal{C}}|\ge \frac{\e\eta}{2}.
\end{equation}
This will be established once we show that
\begin{equation}\label{87078870878-ha}
\left\{ x=(x_1,x_2)\in\R^2 {\mbox{ s.t. }}
x_1\in(-\eta,0) {\mbox{ and }}x_2\in \left(u(x_1)-\frac\e2,u(x_1)\right)\right\}
\subseteq  E\cap{\mathcal{C}}.\end{equation}
To prove this claim, let~$x$ be in the set
on the left hand side of~\eqref{87078870878-ha}.
Then, we have that~$x_2<u(x_1)$, and thus~$x\in E$. Moreover,
using~\eqref{allinfi1} with~$k:=0$,
$$ |x_2-\ell x_1|\le |x_2-u(x_1)|+|u(x_1)-\ell x_1|\le \frac\e2+\delta\e<\e,$$
as long as~$\delta<\frac12$, and consequently~$x\in{\mathcal{C}}$.
This proves~\eqref{87078870878-ha}, and hence~\eqref{sotto00}
readily follows.

Now, by varying the parameter~$\kappa\in\R$,
we slide by below a parabola~${\mathcal{P}}_\kappa$ of the form
$$ x_2=\ell x_1-\frac{\e (x_1-t_\star)^2}2-\kappa$$
till it touches~$u$ in~$\{|x_1|\le1\}$.

To formalize this idea, we first observe that, by~\eqref{allinfi1}
and~\eqref{allinfi2} (used here with~$k:=0$), if~$|x_1|\le1$,
\begin{eqnarray*}
\ell x_1-\frac{\e (x_1-t_\star)^2}2-\kappa-u(x_1)\le\e-\frac{\e (x_1-t_\star)^2}2-\kappa\le\e-\kappa
,\end{eqnarray*}
and so if~$\kappa\ge 1$, we have that~${\mathcal{P}}_\kappa$
lies below the graph of~$u$ in~$\{|x_1|\le1\}$.

Then, we can take~$\kappa$ as small as possible with this property.
We observe that, for this choice of~$\kappa$, 
\begin{equation}\label{INTps}
{\mbox{there must
be an interior touching point in $\{x_1\in(0,5\eta)\}$.}}\end{equation}
Indeed, by~\eqref{tstaso} and the fact that~$u$ lies above the parabola~${\mathcal{P}}_\kappa$,
$$ 0\ge \ell t_\star-\kappa-u(t_\star)>
-\kappa-
\e(\eta^2-1),
$$
which gives that
\begin{equation}\label{kasota}\kappa>\e (1- \eta^2 ).\end{equation}
Hence, if~$x_1\in(-1,0)$, by~\eqref{allinfi1} we have that
$$ \ell x_1-\frac{\e (x_1-t_\star)^2}2-\kappa-u(x_1)
<\delta\e-\frac{\e (x_1-t_\star)^2}2-\e (1-\eta^2)\le
-\e\left(1- \eta^2 -\delta\right)\le-\frac\e2,
$$
as long as~$\eta$ and~$\delta$ are sufficiently small,
which says that 
\begin{equation}\label{INTps2}
{\mbox{the parabola is well separated by~$u$
in~$\{x_1\in(-1,0)\}$.}}\end{equation}
Furthermore, 
if~$x_1\in [5\eta,1]$, we have that~$x_1-t_\star\ge 4\eta$
and then,
by~\eqref{allinfi2} and~\eqref{kasota},
\begin{eqnarray*}&&
\ell x_1-\frac{\e (x_1-t_\star)^2}2-\kappa-u(x_1)\le\e
-\frac{\e (x_1-t_\star)^2}2-\kappa<
\e-\frac{\e (x_1-t_\star)^2}2-\e(1-\eta^2)\\&&\qquad\le
\e-8\e \eta^2-\e(1-\eta^2)=-7\e\eta^2,
\end{eqnarray*}
hence
the parabola is well separated by~$u$
in~$\{x_1\in[5\eta,1]\}$.

{F}rom this and~\eqref{INTps2} we obtain that the parabola touches
the graph of~$u$ at some point in~$[0,5\eta)$.
To complete the proof of~\eqref{INTps}, we have to check that
\begin{equation}\label{Thsi}
{\mbox{this touching point cannot occur at~$x_1=0$.}}\end{equation}
We argue by contradiction, assuming that this is the case.
Then, recalling~\eqref{INTps2},
$$ -\frac{\e t_\star^2}2-\kappa=\liminf_{x_1\searrow0}u(x_1)<
\liminf_{x_1\nearrow0}u(x_1).$$
That is,
the $s$-minimal surface~$E$ would have a boundary
point~$P=(P_1,P_2):=\left(0,-\frac{\e t_\star^2}2-\kappa\right)$
such that~$B_\rho(P)\setminus\Omega\subset E$,
as long as~$\rho>0$ is small enough.

By this and Theorem~1.1 in~\cite{MR3532394},
it follows that, near~$P$, the boundary of~$E$ can be written
as a~$C^{1,\frac{1+s}2}$-graph in the horizontal direction,
namely
$$ \partial E\cap B_{\rho'}(P) =\{ x\in\R {\mbox{ s.t. }} x_1=\phi(x_2)\}\cap B_{\rho'}(P),$$
for some~$\rho'>0$ and~$\phi\in C^{1,\frac{1+s}2}(\R)$,
with~$\phi(P_2)=P_1=0$
%% $\phi(P_2-\tau)\ge0=\phi(P_2+\tau)$ if~$\tau\ge0$ is small enough,
and~$\phi'(P_2)=0$.
%%% $$ \limsup_{x_2\to P_2} \frac{|\phi(x_2)|}{|x_2-P_2|^{\frac{3+s}2}}<+\infty.$$
By construction, we have that~$u(\phi(x_2))=x_2$,
for~$x_2$ close to~$P_2$, and hence, the condition that the parabola~${\mathcal{P}}_\kappa$
lies below~$u$ (with equality at~$x_1=0$)
gives that
$$ x_2\ge \ell \phi(x_2)-\frac{\e (\phi(x_2)-t_\star)^2}2-\kappa,$$
as long as~$x_2$ is close enough to~$P_2$ (with equality at~$x_2=P_2$).
As a consequence,
\begin{eqnarray*}0&\ge&
\left(\ell\phi(x_2)-\frac{\e (\phi(x_2)-t_\star)^2}2-\kappa-x_2\right)-
\left(\ell\phi(P_2)-\frac{\e (\phi(P_2)-t_\star)^2}2-\kappa-P_2\right)\\&=&
\ell\phi(x_2)+\frac{\e }2\Big( t_\star^2-(\phi(x_2)-t_\star)^2\Big)+P_2 -x_2\\&=&
\ell\phi(x_2)+\frac{\e }2\Big( 2t_\star\phi(x_2)-\phi^2(x_2)\Big)+P_2 -x_2\\
&=& \left(\ell+\frac{\e }2\Big( 2t_\star-\phi(x_2)\Big)\right)\phi(x_2)+P_2 -x_2
\end{eqnarray*}
and thus
\begin{eqnarray*}
1\le \lim_{x_2\to P_2}
\left(\ell+\frac{\e }2\Big( 2t_\star-\phi(x_2)\Big)\right)\frac{\phi(x_2)}{x_2-P_2}
=\left(\ell+ \e  t_\star \right)\,\phi'(P_2)=0.
\end{eqnarray*}
This contradiction completes the proof of~\eqref{Thsi}, and hence of~\eqref{INTps},
as desired.

Then, by~\eqref{INTps}, there exists~$Q\in (\partial E)\cap\{ x_1\in(0,5\eta)\}$
which is touched from below by the parabola~${\mathcal{P}}_\kappa$.

We denote by~$\underline{\mathcal{P}}_\kappa$ the subgraph
of~${\mathcal{P}}_\kappa$, and we 
observe that most of~$\underline{\mathcal{P}}_\kappa$ lies outside~${\mathcal{C}}$.
More explicitly, we claim that
\begin{equation}\label{pkap}
|\underline{\mathcal{P}}_\kappa\cap {\mathcal{C}}|\le C(\ell)\,\e\eta^3.
\end{equation}
for some~$C(\ell)>0$. Indeed, if~$x$ belongs to~$\underline{\mathcal{P}}_\kappa\cap {\mathcal{C}}$
we have that~$|x_1|<\eta$, $|x_2-\ell x_1|<\e$ and
$$ x_2<\ell x_1-\frac{\e (x_1-t_\star)^2}2-\kappa.$$
Consequently, recalling~\eqref{kasota},
$$ x_2-\ell x_1<-\kappa
< -\e (1- \eta^2 )=-\e+\e \eta^2 .$$
These observations say that~$\underline{\mathcal{P}}_\kappa\cap {\mathcal{C}}$
is contained in the parallelogram
$$ \{ x\in\R^2 {\mbox{ s.t. }} x_1\in(-\eta,\eta) {\mbox{ and }}
x_2-\ell x_1\in(-\e,-\e+\e\eta^2)\},$$
which in turn implies~\eqref{pkap}.

Assuming~$\eta$ sufficiently small, we deduce from~\eqref{sotto00} and~\eqref{pkap} that
\begin{equation}\label{sotto0090}
\big|(E\cap{\mathcal{C}})\setminus
\underline{\mathcal{P}}_\kappa\big|\ge \frac{\e\eta}{4}.
\end{equation}
We also observe that
\begin{equation}\label{BAS}
{\mathcal{C}}\subseteq B_{C_0(\ell)\eta}(Q)
\end{equation}
for some~$C_0(\ell)>0$.
Indeed, if~$y\in{\mathcal{C}}$, we have that~$|y_1-Q_1|\le 6\eta$
and
$$|y_2-Q_2|\le |y_2-\ell y_1|+|\ell|\,|y_1-Q_1|+|\ell Q_1-Q_2|\le
\e+6|\ell|\,\eta+|\ell Q_1-u(Q_1)|\le
2\e+6|\ell|\,\eta,$$
where~\eqref{allinfi2} was used once again in the last step,
and accordingly~$|y-Q|\le C_0(\ell)\eta$, for some~$C_0(\ell)>0$,
as long as~$\e$ is small enough.

This proves~\eqref{BAS}. In particular, if~$\eta$ is sufficiently small,
we conclude that~${\mathcal{C}}\subseteq B_{1/2}(Q)$.
In addition, we observe that~$\underline{\mathcal{P}}_\kappa$ at~$Q$
has curvature~$-\frac{\e}{(1+\ell^2)^{\frac32}}+O(\e^2)$, and therefore
\begin{equation*}\begin{split}&
\int_{B_{1/2}(Q)}\frac{\chi_{\R^2\setminus E}(y)-\chi_{E}(y)}{|y-Q|^{2+s}}\,dy\\
=\;&
-\int_{B_{1/2}(Q)\cap\underline{\mathcal{P}}_\kappa}\frac{dy}{|y-Q|^{2+s}}
+
\int_{B_{1/2}(Q)\setminus \underline{\mathcal{P}}_\kappa}\frac{\chi_{\R^2\setminus E}(y)-\chi_{E}(y)}{|y-Q|^{2+s}}\,dy\\
=\;&
-\int_{B_{1/2}(Q)\cap\underline{\mathcal{P}}_\kappa}\frac{dy}{|y-Q|^{2+s}}
+
\int_{(E^c\cap B_{1/2}(Q))\setminus \underline{\mathcal{P}}_\kappa}
\frac{dy}{|y-Q|^{2+s}}
-\int_{(E\cap B_{1/2}(Q))\setminus \underline{\mathcal{P}}_\kappa}\frac{
dy}{|y-Q|^{2+s}}\\ \le\;&
C_1(\ell)\e
-\int_{(E\cap B_{1/2}(Q))\setminus \underline{\mathcal{P}}_\kappa}\frac{
dy}{|y-Q|^{2+s}}
\\ \le\;& C_1(\ell)\e
-\int_{(E\cap{\mathcal{C}})\setminus
\underline{\mathcal{P}}_\kappa}\frac{
dy}{|y-Q|^{2+s}},
\end{split}
\end{equation*}
for some~$C_1(\ell)>0$.
Therefore, using again~\eqref{BAS},
we conclude that
\begin{equation*}
\int_{B_{1/2}(Q)}\frac{\chi_{\R^2\setminus E}(y)-\chi_{E}(y)}{|y-Q|^{2+s}}\,dy\le
C_1(\ell)\e
-\frac{\big|(E\cap{\mathcal{C}})\setminus
\underline{\mathcal{P}}_\kappa\big|}{
(C_0(\ell)\eta)^{2+s}}.
\end{equation*}
This and~\eqref{sotto0090} give that
\begin{equation}\label{v4vvBAUksd844-2}
\int_{B_{1/2}(Q)}\frac{\chi_{\R^2\setminus E}(y)-\chi_{E}(y)}{|y-Q|^{2+s}}\,dy\le
-\frac{C_2(\ell)\,\e}{\eta^{1+s}}+
C_1(\ell)\e,
\end{equation}
for some~$C_2(\ell)>0$.

Furthermore, by~\eqref{allinfi1}
and~\eqref{allinfi2},
\begin{equation}\label{v4vvBAUksd844-3}
\int_{B_{2^{k_0-1}}(Q)\setminus
B_{1/2}(Q)}\frac{\chi_{\R^2\setminus E}(y)-\chi_{E}(y)}{|y-Q|^{2+s}}\,dy\le
C\e\int_{1/2}^{2^{k_0}} t^{\alpha-1-s}\,dt
\le \frac{C\e}{s-\alpha},
\end{equation}
for some~$C>0$.

In addition, by~\eqref{koche},
$$ \int_{\R^2\setminus B_{2^{k_0-1}}(Q)}\frac{\chi_{\R^2\setminus E}(y)
-\chi_{E}(y)}{|y-Q|^{2+s}}\,dy\le\frac{C'}{s\,(2^{k_0-1})^s}\le
C(\alpha,s)\,\e,
$$
for some~$C'$, $C(\alpha,s)>0$.
Combining this, \eqref{v4vvBAUksd844-2} and~\eqref{v4vvBAUksd844-3},
we find that
$$ \frac1\e\,\int_{\R^2}\frac{\chi_{\R^2\setminus E}(y)-\chi_{E}(y)}{|y-Q|^{2+s}}\,dy\le
-\frac{C_2(\ell)}{\eta^{1+s}}+
C_1(\ell)+\frac{C}{s-\alpha}+C(\alpha,s),
$$
which is strictly negative as long as~$\eta$ is sufficiently small.
This contradiction proves~\eqref{sotto}, as desired.
\end{proof}

As customary, Lemma~\ref{HARNA} can be iterated
(though a finite number of times for a fixed~$\e$).
In our setting, this iteration is somewhat more delicate
than in the interior case, since one has to treat
the data from outside differently from the solution from the inside
and obtain uniform estimates at the boundary.
The result that we obtain is therefore the following:

\begin{corollary}\label{HARICORSIVO}
Let~$\alpha\in(0,s)$, $\ell\in\R$, $u:\R\to\R$ and
$$ E:=\{ (x_1,x_2)\in\R^2 {\mbox{ s.t. }} x_2<u(x_1)\}.$$
Assume that~$E$ is locally $s$-minimal in~$(0,1)\times\R$.
Then, there exist $c_0$,
$\eta$, $\delta$ and~$\tilde\e_0\in(0,1)$,
depending only
on~$\alpha$, $\ell$, and~$s$,
such that if~$\e\in(0,\tilde\e_0)$ the following statement holds true.

Let
\begin{equation}\label{tilde know}
\tilde k_0:=\left\lceil \frac{-\log\e}{c_0}\right\rceil,\end{equation}
and suppose that
\begin{eqnarray}
\label{LASallinfi1}&&|u(x_1)-\ell x_1|\le \e^{\frac1{c_0}}\, |x_1|^{1+\alpha} {\mbox{ for all $x_1\in (-2^{\tilde k_0},0)$,}}
\\
\label{LASallinfi2}&&|u(x_1)-\ell x_1|\le \e (2^k)^{1+\alpha} {\mbox{ for all $x_1\in (0,2^{k})$,
for all $k\in\{0,\dots, \tilde k_0\}$.}}
\end{eqnarray}
Then, for every~$m\in\N$ such that
\begin{equation}\label{iwuur-alal}
1\le m\le c_0\,|\log\e|,
\end{equation}
we have that
\begin{eqnarray}\label{LASsotto0}
&& \sup_{x_1\in(-\eta^m,\eta^m)} |u(x_1)-\ell x_1|\le \e(1-\eta^2)^m.
\end{eqnarray}
In particular,
under assumptions~\eqref{LASallinfi1}
and~\eqref{LASallinfi2}, there exist~$\vartheta_1$, $\vartheta_2\in(0,1)$
such that, for any~$x_1\in(-\eta,\eta)$,
\begin{eqnarray}\label{11-LASsotto0}
|u(x_1)-\ell x_1|\le \e\max\{ x_1^{ \vartheta_1 },\,\e^{\vartheta_2}\}.
\end{eqnarray}
\end{corollary}

\begin{proof} We have that~\eqref{11-LASsotto0} readily
follows from~\eqref{iwuur-alal}
and~\eqref{LASsotto0}.
Hence, we focus on the proof of~\eqref{LASsotto0}.

To this end, we observe that, if~$x_1\in(-\eta^m,0)$ we can
exploit~\eqref{LASallinfi1} and see that
$$ |u(x_1)-\ell x_1|\le \e^{\frac1{c_0}}\, \eta^{(1+\alpha)m}
=\e^{\frac1{c_0}}(1-\eta^2)^m\, \left(\frac{\eta^{1+\alpha} }{ 1-\eta^2}\right)^m
\le\e^{\frac1{c_0}}(1-\eta^2)^m,$$
which implies~\eqref{LASsotto0} in this case.

Consequently, to prove~\eqref{LASsotto0}, it is enough to consider
the case~$x_1\in(0,\eta^m)$.

Hence, we prove~\eqref{LASsotto0} when~$x_1\in(0,\eta^m)$ by induction over~$m\in\N$.
We let~$\eta$, $\delta$ and~$\e_0$ be as in
Lemma~\ref{HARNA} and we iterate such a result till
\begin{equation}\label{iwuur-alal2}
m\le\frac{\log\frac{\e_0}{\e}}{\log\frac{1-\eta^2}\eta}.
\end{equation}
We will also take~$\tilde\e_0$ in the statement of
Corollary~\ref{HARICORSIVO} to be equal to~$\e_0^2$.
In this way, for every~$\e\in(0,\tilde\e_0)$,
we have that
$$\log\frac{\e_0}{\e}=\frac12\,\log\frac{\tilde\e_0}{\e^2}
\ge\frac12\,\log\frac{1}{\e},$$
and therefore~\eqref{iwuur-alal} implies~\eqref{iwuur-alal2}
for a suitable choice of~$c_0$. 

For future convenience, we also set
\begin{equation}\label{CETANO}
\tilde C(\eta):=-\frac{\log(1-\eta^2)}{\log\frac{1-\eta^2}\eta}>0.
\end{equation}
The induction argument goes
as follows. First of all, up to taking~$\eta$ in Lemma~\ref{HARNA}
smaller, we can assume that~$\eta=2^{-M_0}$, for some~$M_0\in\N$,
$M_0\ge2$.
When~$m=1$, we have that~\eqref{LASsotto0}
follows from Lemma~\ref{HARNA}. 

We now suppose that~\eqref{LASsotto0} holds true
for all~$m\in\{ 1,\dots,m_0\}$, with
\begin{equation}\label{iwuur-alal3}
m_0\le\frac{\log\frac{\e_0}{\e}}{\log\frac{1-\eta^2}\eta},\end{equation}
and we want to prove that~\eqref{LASsotto0} holds true
for~$m=m_0+1$.

To this end, we define
$$ \tilde u(x_1):=\frac1{\eta^{m_0}}\,u(\eta^{m_0} x_1)\qquad{\mbox{ and }}\qquad
\tilde\e:=\frac{\e(1-\eta^2)^{m_0}}{\eta^{m_0}}.$$
We also use the notation
\begin{equation}\label{NOTAZXCA}
x_1:=\eta^{m_0}\tilde x_1\qquad{\mbox{
and }}k:=\tilde k-m_0M_0.\end{equation}
Then, we claim that
\begin{equation}\label{0-sted1}
|\tilde u(\tilde x_1)-\ell \tilde x_1|\le\delta\tilde\e (2^{\tilde k})^{1+\alpha} {\mbox{ 
for all $\tilde x_1\in (-2^{\tilde k},0)$,
for all $\tilde k\in\{0,\dots, \tilde k_0\}$.}}
\end{equation}
For this, we exploit~\eqref{LASallinfi1}
and obtain that
\begin{equation}\label{980982948wiwegy}
|\tilde u(\tilde x_1)-\ell \tilde x_1|=\frac1{\eta^{m_0}}
|u(x_1)-\ell x_1|\le\frac{\e^{\frac1{c_0}}\, |x_1|^{1+\alpha}}{{\eta^{m_0}}}=
\frac{\e^{\frac1{c_0}-1}\, \tilde\e\,|x_1|^{1+\alpha}}{{(1-\eta^2)^{m_0}}}.\end{equation}
On the other hand, from~\eqref{CETANO} and~\eqref{iwuur-alal3},
\begin{eqnarray*}&&
\log\Big( (1-\eta^2)^{m_0}\Big)=m_0\,\log(1-\eta^2)\ge
\frac{\log\frac{\e_0}{\e}\;\log(1-\eta^2)}{\log\frac{1-\eta^2}\eta}\\&&
\qquad\qquad=-\tilde C(\eta)\,\log\frac{\e_0}{\e}=
\log\frac{\e^{\tilde C(\eta)}}{\e_0^{\tilde C(\eta)}}\ge
\log\big(\e^{\tilde C(\eta)}\big),
\end{eqnarray*}
and therefore~$(1-\eta^2)^{m_0}\ge \e^{\tilde C(\eta)}$.
Then, by plugging this information into~\eqref{980982948wiwegy},
we conclude that
$$ |\tilde u(\tilde x_1)-\ell \tilde x_1|\le
\e^{\frac1{c_0}-1-\tilde C(\eta)}\, \tilde\e\,|x_1|^{1+\alpha}\le\delta
\tilde\e\,|x_1|^{1+\alpha},$$
provided that~$c_0$ is chosen small enough,
which gives~\eqref{0-sted1}, as desired.

Now we claim that
\begin{equation}\label{0-sted2}
|\tilde u(\tilde x_1)-\ell \tilde x_1|\le\tilde\e (2^{\tilde k})^{1+\alpha} {\mbox{ 
for all $\tilde x_1\in (0,2^{\tilde k})$,
for all $\tilde k\in\{0,\dots, \tilde k_0\}$.}}
\end{equation}
To check this, we first consider the case~$\tilde k\in\{0,\dots,m_0M_0\}$,
and we let
$$ m_{\tilde k}:=\left\lfloor m_0-\frac{\tilde k}{M_0}\right\rfloor\in\{0,\dots,m_0\}.$$
We point out that
$$ x_1=2^{-m_0 M_0}\tilde x_1\in (0,2^{-m_0 M_0+\tilde k})\subseteq
(0,2^{-M_0 m_{\tilde k}})=(0,\eta^{m_{\tilde k}}).
$$
Then, since we know by inductive assumption
that~\eqref{LASsotto0} is satisfied when~$m\le m_0$, we find that
\begin{equation}\label{PROVV92}
|\tilde u(\tilde x_1)-\ell \tilde x_1|=\frac1{\eta^{m_0}}\,| u( x_1)-\ell x_1|
\le\frac{\e(1-\eta^2)^{ m_{\tilde k} }}{\eta^{m_0}}.
\end{equation}
Now, when~$\tilde k=0$ we deduce from~\eqref{PROVV92} that
$$ |\tilde u(\tilde x_1)-\ell \tilde x_1|
\le\frac{\e(1-\eta^2)^{ m_0}}{\eta^{m_0}}=\tilde\e,
$$
which is~\eqref{0-sted2} in this case.

If instead~$\tilde k\in\{1,\dots,m_0M_0\}$, we infer from~\eqref{PROVV92}
that
\begin{equation}\label{bRadifatuenna}
|\tilde u(\tilde x_1)-\ell \tilde x_1|
\le \tilde\e(1-\eta^2)^{ m_{\tilde k}-m_0 }
=\tilde\e(2^{\tilde k})^{1+\alpha}\times
\frac{ (1-\eta^2)^{ m_{\tilde k}-m_0 } }{ 2^{\tilde k(1+\alpha)} }.\end{equation}
Moreover,
\[ (1-\eta^2)^{\frac1{M_0}} \;2^{1+\alpha} =
(1-2^{-2M_0})^{\frac1{M_0}} \;2^{1+\alpha}\ge 2^\alpha,
\]
as long as~$\eta$ is small enough (hence~$M_0$ large enough), and thus
\begin{eqnarray*}
&& \frac{ (1-\eta^2)^{ m_{\tilde k}-m_0 } }{ 2^{\tilde k(1+\alpha)} }
\le \frac{ (1-\eta^2)^{ -\frac{\tilde k}{M_0} -1} }{ 2^{\tilde k(1+\alpha)} }=
\frac{1}{1-\eta^2}\;\left(
\frac{1}{ (1-\eta^2)^{ \frac{1}{M_0} }\; 2^{1+\alpha}} \right)^{\tilde k}\\
&&\qquad\qquad\le\frac{1}{(1-\eta^2)\;2^{\alpha\tilde k}}\le
\frac{1}{(1-\eta^2)\;2^{\alpha}}\le 1,
\end{eqnarray*}
provided that~$\eta$ is small enough.

{F}rom this and~\eqref{bRadifatuenna}, we conclude that
$$ |\tilde u(\tilde x_1)-\ell \tilde x_1|\le\tilde\e(2^{\tilde k})^{1+\alpha}.$$
The latter estimate gives~\eqref{0-sted2} when~$\tilde k\in\{1,\dots,m_0M_0\}$.

Now, we focus on the proof of~\eqref{0-sted2} when~$\tilde k\in\{m_0M_0+1,\dots, \tilde k_0\}$.
In this case, we exploit~\eqref{LASallinfi2} with the
notation in~\eqref{NOTAZXCA},
observing that~$k\ge0$ and~$x_1\in(0,\eta^{m_0}2^{\tilde k})=(0,2^k)$,
and we conclude that
\begin{eqnarray*}&&
|\tilde u(\tilde x_1)-\ell \tilde x_1|=\frac1{\eta^{m_0}}\,| u( x_1)-\ell x_1|
\le \frac{\e (2^k)^{1+\alpha} }{\eta^{m_0}}\\&&\qquad=\tilde\e(2^{\tilde k})^{1+\alpha} \times
\frac{(2^{-m_0M_0})^{1+\alpha} }{(1-\eta^2)^{m_0}}=
\tilde\e(2^{\tilde k})^{1+\alpha} \times\left(
\frac{\eta^{1+\alpha} }{(1-\eta^2)}\right)^{m_0}\leq\tilde\e(2^{\tilde k})^{1+\alpha}.
\end{eqnarray*}
This gives~\eqref{0-sted2} when~$\tilde k\in\{m_0M_0+1,\dots, \tilde k_0\}$.
The proof of~\eqref{0-sted2} is therefore
complete.

Now, by~\eqref{0-sted1} and~\eqref{0-sted2}, we obtain that~\eqref{allinfi1}
and~\eqref{allinfi2} are satisfied, with~$\tilde u$ and~$\tilde \e$
instead of~$u$ and~$\e$.

Also, in view of~\eqref{iwuur-alal3},
\begin{equation*}
\tilde\e=\e\times\left(
\frac{1-\eta^2}{\eta}\right)^{m_0}\le
\e\times\left(
\frac{1-\eta^2}{\eta}\right)^{
\frac{\log\frac{\e_0}{\e}}{\log\frac{1-\eta^2}\eta}}=
\e\,\exp\left( \log\frac{\e_0}{\e}\right)=\e_0.
\end{equation*}
Then, to be in the position of using
Lemma~\ref{HARNA} with~$\tilde u$ and~$\tilde \e$
instead of~$u$ and~$\e$, 
it remains to check that if~$k_0$ is defined as in~\eqref{koche}
with~$\e$ replaced by~$\tilde \e$, one has that
such a~$k_0$ is less than or equal to the~$\tilde k_0$
defined in~\eqref{tilde know}. This is indeed the case, since
\begin{eqnarray*}&& k_0-\tilde k_0=
\left\lceil \frac{-\log_2\tilde\e}{\alpha}\right\rceil-
\left\lceil \frac{-\log\e}{c_0}\right\rceil\le
\frac{-\log_2\e-m_0 \log_2
\frac{1-\eta^2}{\eta}}{\alpha}-\frac{-\log\e}{c_0}+1<0,
\end{eqnarray*}
as long as~$c_0$ is chosen sufficiently small.

With this, we have that all the assumptions of
Lemma~\ref{HARNA} are fulfilled with~$\tilde u$ and~$\tilde \e$
replacing~$u$ and~$\e$. Hence we
obtain that, for every~$\tilde x_1\in(0,\eta^{m_0+1})$,
\begin{eqnarray*}&&
|u(x_1)-\ell x_1|=\eta^{m_0} \left|\tilde u\left(\frac{
x_1}{\eta^{m_0}}\right)-
\frac{\ell x_1}{\eta^{m_0}}\right|\le\eta^{m_0}
\sup_{\tilde x_1\in(0,\eta)}\left|\tilde u(\tilde x_1)-\ell \tilde x_1\right|\\
&&\qquad\le\eta^{m_0}\tilde\e(1-\eta^2)=
\e(1-\eta^2)^{m_0+1},
\end{eqnarray*}
and this completes the inductive step, thus finishing the proof of~\eqref{LASsotto0}.
\end{proof}

Combining Corollary~\ref{HARICORSIVO} with
the interior Harnack estimates in Lemma~6.9 of~\cite{MR2675483}
we obtain the following global oscillation result:

\begin{corollary}\label{HARICORSIVO2}
Let~$\e\in(0,1)$, $\alpha\in(0,s)$, $\ell\in\R$, $u:\R\to\R$ and
$$ E:=\{ (x_1,x_2)\in\R^2 {\mbox{ s.t. }} x_2<u(x_1)\}.$$
Let also
\begin{equation*}
\tilde k_0:=\left\lceil \frac{-\log\e}{c_0}\right\rceil,\end{equation*}
Assume that~$E$ is locally $s$-minimal in~$(0,2^{\tilde k_0})\times\R$, and
\begin{eqnarray*}&&|u(x_1)-\ell x_1|\le \e^{\frac1{c_0}}\, |x_1|^{1+\alpha} {\mbox{ for all $x_1\in (-2^{\tilde k_0},0)$,}}
\\&&|u(x_1)-\ell x_1|\le \e (2^k)^{1+\alpha} {\mbox{ for all $x_1\in (0,2^{k})$,
for all $k\in\{0,\dots, \tilde k_0\}$.}}
\end{eqnarray*}
Then, there exist $\tilde\e_0$, $\vartheta\in(0,1)$,
depending only
on~$\alpha$, $\ell$, and~$s$,
and~$\phi_\e:[-2^{\tilde k_0-1},2^{\tilde k_0+1}]\to\R$ which is locally
H\"older continuous
with H\"older norm bounded uniformly in~$\e$ 
in any interval of the form~$[-2^{\tilde k_0-1},2^{\tilde k_0+1}]\setminus
(-\delta,\delta)$, for any given~$\delta>0$,
with~$|\phi_\e(x_1)|\le |x_1|^\vartheta$ for all~$x_1\in(-1,1)$,
and such that, if~$\e\in(0,\tilde\e_0)$,
$$ \e\big(\phi_\e(x_1)-\e^\vartheta\big)\le u(x_1)-\ell x_1\le \e\big(\phi_\e(x_1)+\e^\vartheta\big),$$
for all~$x_1\in[-2^{\tilde k_0-1},2^{\tilde k_0+1}]$.
\end{corollary}

\section{Vertical rescalings}\label{HAR2}

In this section, we will consider vertical rescaling of a nonlocal
minimal graph. In combination with the global Harnack result
in Corollary~\ref{HARICORSIVO2}, we will obtain suitable H\"older
estimates (up to a negligible scale) which allow us to deduce
convergence up to the boundary. Interestingly, the limit
function will be $\sigma$-harmonic in~$(0,+\infty)$,
with~$\sigma:=\frac{1+s}{2}$. The precise result
will be given in the forthcoming Lemma~\ref{-02390-23409}.
To this end, we first present an auxiliary result about
the limit equation.

\begin{lemma}\label{-02390-23409-0}
Let~$\ell\in\R$ and~$u:\R\to\R$. Let
$$ E:=\{ (x_1,x_2)\in\R^2 {\mbox{ s.t. }} x_2<u(x_1)\}$$
and, for any~$\e\in(0,1)$,
\begin{equation}\label{ueti} u_\e(x_1):= \frac{u(x_1)-\ell x_1}{\e}.\end{equation}
Assume that~$E$ is locally $s$-minimal in~$(0,1)\times\R$,
and that
\begin{equation}\label{CO:a}
|u_\e(x_1)|\le C\,(1+|x_1|^{1+\alpha}),\end{equation} for some~$C>0$
and~$\alpha\in(0,s)$, and
that~$u_\e$ converges to~$\bar u$ locally uniformly in~$\R$.
Then~$(-\Delta)^{\frac{1+s}{2}}\bar u=0$ in~$(0,+\infty)$.
\end{lemma}

\begin{proof} When~$\ell=0$, one can exploit the proof of Lemma~6.11
in~\cite{MR2675483}.
For the general case, we argue as follows. Let~$p\in(0,1)$
and~$\psi$ be a smooth function touching~$\bar u$
from below at~$p$. We can also suppose that~$\psi$
satisfies~\eqref{CO:a}.
Let also~$\gamma>0$, $\alpha'\in(\alpha,s)$, and
\begin{equation}\label{7GAdefoa}
\phi(x_1):=\psi(x_1)-\gamma\,(1+|x_1-p|^2)^{\frac{1+\alpha'}2},\end{equation}
with~$C$ as in~\eqref{CO:a}. We observe that
\begin{equation}\label{cooall}\begin{split}& \lim_{|x_1|\to+\infty} u_\e(x_1)-\phi(x_1)
=\lim_{|x_1|\to+\infty} u_\e(x_1)-\psi(x_1)+\gamma\,(1+|x_1-
p|^2)^{\frac{1+\alpha'}2}\\&\qquad\ge
\lim_{|x_1|\to+\infty} -2C|x_1|^{1+\alpha}+\gamma\,(1+|x_1-p|^2)^{\frac{1+\alpha'}2}=+\infty.\end{split}\end{equation}
Hence, we can take~$p_\e\in\R$ such that
$$ \inf_{x_1\in\R} u_\e(x_1)-\phi(x_1)=
u_\e(p_\e)-\phi(p_\e),$$
and thus a vertical translation of
the function~$\phi$ touches~$ u_\e$
from below at the point~$p_\e$.

As a consequence, by~\eqref{ueti}, the
function~$x_1\mapsto\e\phi(x_1)+
\ell x_1=:\zeta_\e(x_1)$ touches, up to a vertical translation,
the function~$ u$
from below at~$p_\e$. 

We observe that
\begin{equation}\label{pep}
\lim_{\e\searrow0} p_\e=p.
\end{equation}
Indeed, suppose not. Then, by~\eqref{cooall}, we know that~$p_\e$
is bounded and thus, up to a subsequence, we can suppose that~$p_\e\to \tilde p\ne p$
as~$\e\searrow0$. By construction,
\begin{eqnarray*}&& u_\e(p_\e)-\psi(p_\e)
+\gamma\,(1+|p_\e-p|^2)^{\frac{1+\alpha'}2}=
u_\e(p_\e)-\phi(p_\e)\\&&\qquad\le
u_\e(p)-\phi(p)=u_\e(p)-\psi(p)+\gamma.\end{eqnarray*}
Hence, taking the limit as~$\e\searrow0$,
$$ \bar u(\tilde p)-\psi(\tilde p)
+\gamma\,(1+|\tilde p-p|^2)^{\frac{1+\alpha'}2}\le
\bar u(p)-\psi(p)+\gamma=\gamma\le\bar u(\tilde p)-\psi(\tilde p)+\gamma.$$
This is a contradiction, that proves~\eqref{pep}.

In particular,
for small values of~$\e$, we may suppose that~$p_\e\in(0,1)$,
and then the minimality of~$E$ gives that
\begin{equation}\label{uJAM}
0\le
\int_{\R^2} \frac{\chi_{\R^n\setminus \Pi_\e}(y)-\chi_{\Pi_\e}(y)}{|y-P_\e|^{2+s}}\,dy,
\end{equation}
where~$P_\e:=(p_\e,u(p_\e))$ and~$\Pi_\e:=\{x_2<\zeta_\e(x_1)\}$.

Using formula~(49) in~\cite{MR3331523}, we can write~\eqref{uJAM}
in the form
\begin{equation}\label{8JA-9wsjAJ}
\begin{split}
0\,&\ge
\int_{\R} F\left(\frac{\zeta_\e(p_\e-t)-\zeta_\e(p_\e)}{|t|}
\right)\,\frac{dt}{|t|^{1+s}}\\&=
\int_{\R} F\left(\frac{
\e\phi(p_\e-t)-\e\phi(p_\e)-
\ell t}{|t|}
\right)\,\frac{dt}{|t|^{1+s}},
\end{split}
\end{equation}
with
\begin{equation}\label{DEF:F} F(r):=\int_0^r\frac{d\tau}{(1+\tau^2)^{\frac{2+s}2}}.\end{equation}
Since~$F$ is odd, we have that
$$ \int_{\R} F\left(\frac{-(\e\phi'(p_\e)+\ell )t}{|t|}
\right)\,\frac{dt}{|t|^{1+s}}=0.$$
Hence, setting
$$ G_\e(t):=\frac1\e\left[
F\left(\frac{
\e\phi(p_\e-t)-\e\phi(p_\e)-
\ell t}{|t|}
\right)-F\left(\frac{-(\e\phi'(p_\e)+\ell) t}{|t|}
\right)\right],$$
we obtain from~\eqref{8JA-9wsjAJ} that
\begin{equation}\label{71234567987654}
0\ge
\int_{\R} G_\e(t)\,\frac{dt}{|t|^{1+s}}.
\end{equation}
We also remark that
\begin{equation} \label{8iJA-x09w}
|G_\e(t)|\le\sup_{r\in\R}|F'(r)|\,
\frac{|\phi(p_\e-t)-\phi(p_\e)+\phi'(p_\e)t|}{|t|}.
\end{equation}
Hence, if~$d\in(0,\,|p|/2)$,
\begin{equation} \label{7uSKKA}\begin{split}&\sup_{|t|\le d/4}
|G_\e(t)|\le\sup_{r\in\R}|F'(r)|\,\sup_{|t|\le d/4}
\frac{|\phi(p_\e-t)-\phi(p_\e)+\phi'(p_\e)t|}{|t|}\\&\qquad\qquad\le\sup_{r\in\R}|F'(r)|
\sup_{q\in(p-d,p+d)}|\phi''(q)|\;|t|.\end{split}
\end{equation}
On the other hand,
$$ |\phi(p_\e-t)|\le \sup_{q\in(p-d,p+d)}|\phi(q)|+C(1+|t|^{1+\alpha'}),$$
for some~$C>0$.

As a consequence, recalling~\eqref{8iJA-x09w}, if~$|t|>d/4$,
\begin{eqnarray*}
|G_\e(t)|&\le&\sup_{r\in\R}|F'(r)|\,
\frac{1}{|t|}\left(
\sup_{q\in(p-d,p+d)}|\phi(q)|+C(1+|t|^{1+\alpha'})\right)\\&\le&
\sup_{r\in\R}|F'(r)|\,
\left( \frac4d\left(
\sup_{q\in(p-d,p+d)}|\phi(q)|+C\right)+C |t|^{\alpha'}\right).
\end{eqnarray*}
Using this and~\eqref{7uSKKA}, we can exploit the Dominated Convergence Theorem
in~\eqref{71234567987654}, and thus conclude that
\begin{eqnarray*}
0&\ge& \int_{\R}\lim_{\e\searrow0} G_\e(t)\,\frac{dt}{|t|^{1+s}}\\
&=& \int_{\R}\lim_{\e\searrow0} 
F'\left(\frac{
\e\phi(p_\e-t)-\e\phi(p_\e)-
\ell t}{|t|}
\right)\,
\frac{
\phi(p_\e-t)-\phi(p_\e)}{|t|}\,\frac{dt}{|t|^{1+s}}\\
&=& \int_{\R}
F'\left(\frac{-\ell t}{|t|}
\right)\,
\frac{
\phi(p-t)-\phi(p)}{|t|}\,\frac{dt}{|t|^{1+s}}\\
&=& \frac{1}{(1+\ell^2)^{\frac{n+s}2}}
\,\int_{\R}
\Big(\phi(p-t)-\phi(p)\Big)\,\frac{dt}{|t|^{2+s}}.
\end{eqnarray*}
Then, recalling~\eqref{7GAdefoa} and the fact that~$\alpha'<s$,
we can send~$\gamma\searrow0$ and find that
$$ 0\ge
\,\int_{\R}
\Big(\psi(p-t)-\psi(p)\Big)\,\frac{dt}{|t|^{2+s}}.$$
Since~$\psi$ is an arbitrary function touching from below,
we thus obtain that~$-(-\Delta)^{\frac{1+s}{2}}\bar u(p)\le0$.
The other inequality can be proved in a similar way.
\end{proof}

We now obtain a limit equation for the vertical rescaling,
together with its asymptotics,
as stated in the following result:

\begin{lemma}\label{-02390-23409}
Let~$\alpha\in(0,s)$, $\ell\in\R$ and~$u:\R\to\R$. Let
$$ E:=\{ (x_1,x_2)\in\R^2 {\mbox{ s.t. }} x_2<u(x_1)\}$$
and, for any~$\e\in(0,1)$,
\begin{equation}\label{XCAiqw} u_\e(x_1):= \frac{u(x_1)-\ell x_1}{\e}.\end{equation}
There exists~$c_0\in(0,1)$, only depending on~$\ell$, $\alpha$ and~$s$, such that the following statement
holds true. Assume that~$E$ is locally $s$-minimal in~$(0,2^{\tilde k_0})\times\R$,
with
\begin{equation*}
\tilde k_0:=\left\lceil \frac{-\log\e}{c_0}\right\rceil,\end{equation*}
and that
\begin{eqnarray*}&&|u(x_1)-\ell x_1|\le \e^{\frac1{c_0}}\, |x_1|^{1+\alpha} {\mbox{ for all $x_1\in (-2^{\tilde k_0},0)$,}}
\\&&|u(x_1)-\ell x_1|\le \e (2^k)^{1+\alpha} {\mbox{ for all $x_1\in (0,2^{k})$,
for all $k\in\{0,\dots, \tilde k_0\}$.}}
\end{eqnarray*}
Then, up to a subsequence, we have that~$u_\e$ converges locally uniformly in~$\R$
to a function~$\bar u:\R\to\R$ satisfying~$|\bar u(x_1)|\le C(1+|x_1|)^{1+\alpha}$
for all~$x_1\in\R$, and~$(-\Delta)^{\frac{1+s}{2}}\bar u=0$ in~$(0,+\infty)$.

Furthermore,
as~$x_1\searrow0$,
\begin{equation}\label{eqCAfddav} \bar u(x_1)= \bar a x_1^{\frac{1+s}{2}} +O(|x_1|^{{\frac{3+s}{2}}}),\end{equation}
for some~$\bar a\in\R$.
\end{lemma}

\begin{proof} The locally uniform
convergence of~$u_\e$ to some~$\bar u$ follows from
Corollary~\ref{HARICORSIVO2} and the Ascoli-Arzel\`a Theorem.
The corresponding limit equation is a consequence of Lemma~\ref{-02390-23409-0}.
Finally, the asymptotics in~\eqref{eqCAfddav} follows from
Lemma~\ref{CLAUDIA}.
\end{proof}

\section{A corner-like barrier}\label{HAR3}

Here, we construct a new subsolution of the fractional
mean curvature equation which exhibits a corner at the origin
and it is localized in space. Its existence plays a crucial role
in our setting, since it is capable to detect Lipschitz
singularities at the boundary. This, combined with the blow-up
analysis in Corollary~\ref{P:AL:IU}, is then sufficient to show discontinuity.
Hence, in some sense, the existence of this barrier
is the fundamental ingredient to show a result
like ``boundary corners imply discontinuities and stickiness,
but in absence of discontinuities the graph meets the boundary datum
in a differentiable way'' (as formalized in Corollary~\ref{CURV}).

The construction of this new barrier is as follows:

\begin{lemma}\label{P:AL:IUBB}
Let~$\tilde\ell>0$ and~$\bar\ell\in[-\tilde\ell,\tilde\ell]$. Let~$\lambda>0$ and~$L\in(\lambda,+\infty)$. Let~$
a$, $b$, $c>0$, $\alpha\in(0,s)$.

Let~$\e\in(0,1)$ and assume that
\begin{equation}\label{78-0123e}
L\ge\frac{c}{\e^{1/s}}.
\end{equation}
Let
$$ \beta(x_1):=\begin{cases}
\bar\ell x_1 & {\mbox{ if }}x_1\in(-L,0),\\
\bar\ell x_1+\e a x_1& {\mbox{ if }}x_1\in[0,\lambda],\\
\bar\ell x_1-\e bx_1^{1+\alpha} & {\mbox{ if }}x_1\in(\lambda,L).
\end{cases}$$
Let also
$$ F:=\{(x_1,x_2)\in\R^2 {\mbox{ s.t. }} x_1\in(-L,L) {\mbox{ and }}
x_2<\beta(x_1)\}.$$
Then there exists~$\mu\in \left(0,\frac\lambda8\right)$ depending only on~$s$,
$\tilde\ell$, $\lambda$, $a$, $b$, $c$ and~$\alpha$
(but independent of~$\e$ and~$L$), such that
\begin{equation}\label{SPkedyria} \int_{\R^2} \frac{\chi_{\R^2\setminus F}(y)-\chi_{F}(y)}{|x-y|^{2+s}}\,dy\le0\end{equation}
for every~$x=(x_1,x_2)\in\partial F$ with~$x_1\in(0,\mu)$.
\end{lemma}

\begin{figure}[h]
\centering
\includegraphics[width=15.5 cm]{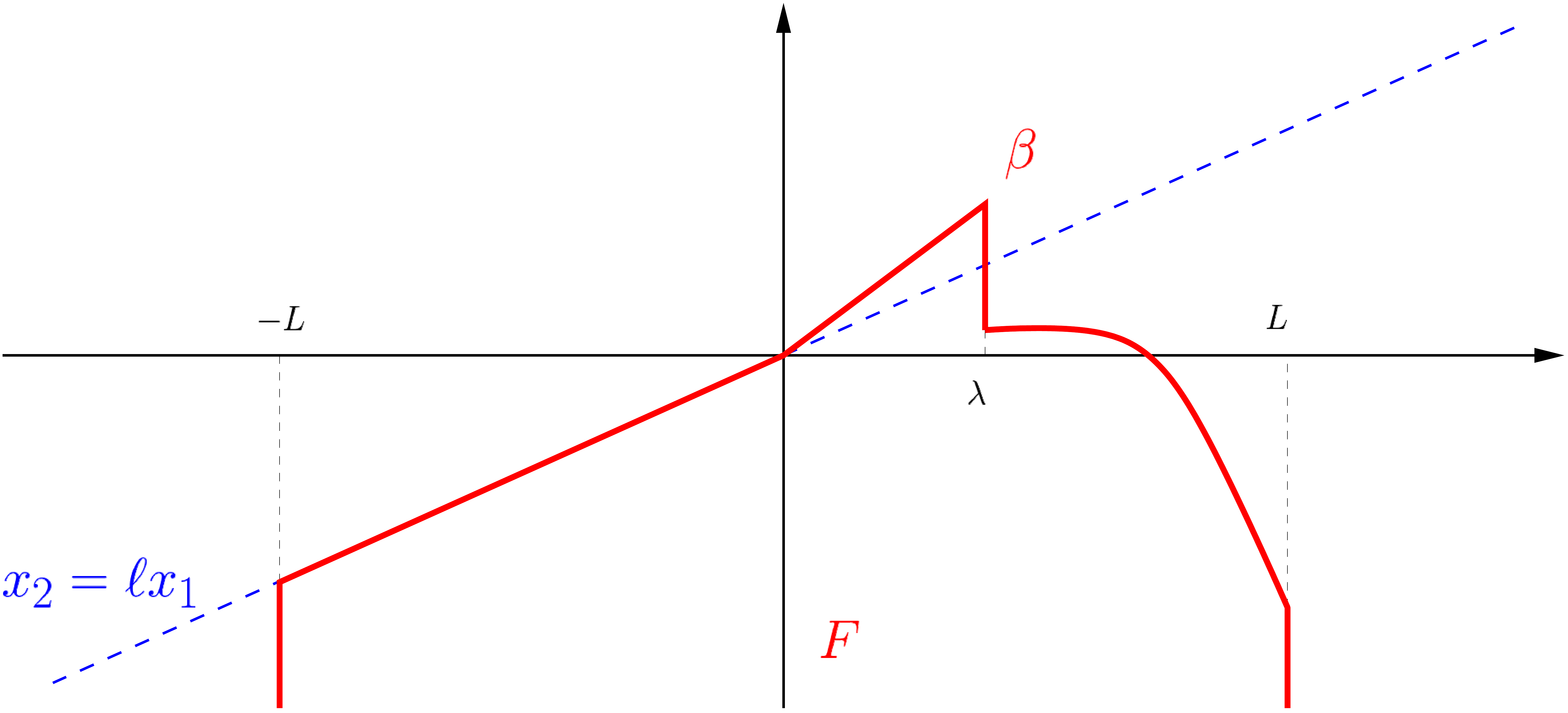}
\caption{\footnotesize\it The barrier in Lemma~\ref{P:AL:IUBB}.}
\label{FIGBARRIO}
\end{figure}

\begin{proof} The geometric situation under consideration
is depicted in Figure~\ref{FIGBARRIO}.
Fixed~$P=(P_1,P_2)\in\partial F$ with~$P_1\in(0,\mu)$,
as long as~$\mu<\frac\lambda2$, we have that~$P_2=\beta(P_1)=
(\bar\ell +\e a)\,P_1$, and
we consider the tangent half-plane of~$F$ through~$P$.
Namely, we define
$$ \Pi_P:=\{(x_1,x_2)\in\R^2 {\mbox{ s.t. }} x_2< (\bar\ell+\e a) (x_1-P_1)+P_2 \}.$$
Given~$r>0$ and~$R\in(r,+\infty]$, we also define
$$ H_P(r,R):=
\int_{B_R(P)\setminus B_r(P)} \frac{\big(\chi_{\R^2\setminus F}(y)-\chi_{F}(y)\big)-
\big(\chi_{\R^2\setminus \Pi_P}(y)-\chi_{\Pi_P}(y)\big)}{|P-y|^{2+s}}\,dy,$$
with the obvious notation that the ball of radius~$+\infty$ is simply the whole~$\R^2$.
Then, the desired claim is established once we prove that
\begin{equation}\label{TBOpAOjoi}
H_P(0,+\infty)\le0.
\end{equation}
The strategy for proving this is that the main positive contribution
comes from the corner at the origin, and the rest of the barrier
will provide negligible terms. To formalize this, we first observe that
\begin{equation}\label{0-A9238eg475-F}
\big(\chi_{\R^2\setminus F}(y)-\chi_{F}(y)\big)-
\big(\chi_{\R^2\setminus \Pi_P}(y)-\chi_{\Pi_P}(y)\big)\le0\qquad{\mbox{for all }}
y\in B_{\lambda/2}(P).
\end{equation}
To check this, let~$y\in B_{\lambda/2}(P)\cap \Pi_P$.
Then, since~$y\in\Pi_P$,
\begin{equation}\label{0-A9238eg475}
y_2< (\bar\ell+\e a) (y_1-P_1)+P_2=
(\bar\ell+\e a) (y_1-P_1)+(\bar\ell+\e a) P_1=(\bar\ell+\e a) y_1.
\end{equation}
Also, since~$y\in B_{\lambda/2}(P)$, we have that
$$ |y_1|\le |y_1-P_1|+|P_1|\le |y-P|+\mu<\frac\lambda2+\mu\le\lambda,$$
and thus
$$ \beta(y_1)=\bar\ell y_1+\e a y_1\,\chi_{[0,\lambda]}(y_1)\ge
(\bar\ell +\e a )y_1
.$$
This and~\eqref{0-A9238eg475} give that~$y_2<\beta(y_1)$,
and thus~$y\in F$. We have thereby established that~$
B_{\lambda/2}(P)\cap \Pi_P\subseteq F$, which
proves~\eqref{0-A9238eg475-F}.

Now, let
\begin{equation}\label{DEPP} {\mathcal{P}}:=
\left\{(y_1,y_2)\in\R^2 {\mbox{ s.t. }} y_1\in\left(-P_1,-\frac{P_1}2\right)
{\mbox{ and }}y_2-\bar\ell y_1\in\left( -\frac{\e a P_1}{4},0\right)\right\}.\end{equation}
We observe that
if~$y\in{\mathcal{P}}$ we have that
\begin{equation*}
\begin{split}
&|y_1-P_1|\le|y_1|+|P_1|< 2P_1
\\{\mbox{ and }}&
|y_2-P_2|\le |y_2-\bar\ell y_1|+|\bar\ell|\,|y_1|+|P_2|
<\frac{\e aP_1}{4}+|\bar\ell|\,P_1+|(\bar\ell+\e a)P_1|\\
&\qquad\qquad\le\left(
\frac{a}{4}+|\bar\ell|+(|\bar\ell|+a)\right)\,P_1
\le\left(
\frac{a}{4}+\tilde\ell+(\tilde\ell+a)\right)\,P_1.
\end{split}\end{equation*}
Accordingly,
\begin{equation}\label{PinFn-00}
{\mathcal{P}}\subseteq B_{C(\tilde\ell,a)\, P_1}(P),
\end{equation}
for some~$C(\tilde\ell,a)>0$.

We also claim that
\begin{equation}\label{PinFn}
{\mathcal{P}}\subseteq F\setminus \Pi_P.\end{equation}
Indeed, if~$y\in {\mathcal{P}}$, we know from~\eqref{DEPP}
that~$|y_1|<\lambda$, and thus~$\beta(y_1)=
\bar\ell y_1+\e a y_1\,\chi_{[0,\lambda]}(y_1)$.
Consequently,
$$ y_2-\beta(y_1)=y_2-\bar\ell y_1-\e a y_1\,\chi_{[0,\lambda]}(y_1)\le
y_2-\bar\ell y_1<0,
$$
which says that~$y\in F$. In addition,
$$ y_2-
(\bar\ell+\e a) (y_1-P_1)-P_2=(
y_2-
\bar\ell y_1)-\e a y_1>-\frac{\e aP_1}{4}+\frac{\e a P_1}2>0,
$$
which says that~$y\in\R^2\setminus\Pi_P$.
These observations give~\eqref{PinFn}.

Now, in light of~\eqref{PinFn-00} and~\eqref{PinFn},
noticing that~$C(\tilde\ell,a)P_1\le C(\tilde\ell,a)\mu<\lambda/2$
as long as~$\mu$ is
sufficiently small (possibly in dependence of~$\tilde\ell$ and~$a$),
we can write that
\begin{equation}\label{0-x-249293}
{\mathcal{P}}\subseteq B_{\lambda/2}(P)\cap (F\setminus \Pi_P).
\end{equation}
Also, since~${\mathcal{P}}$ is a parallelogram,
its area~$|{\mathcal{P}}|$ is equal to~$\frac{\e a P_1^2}8$.
{F}rom this, \eqref{0-A9238eg475-F}
and~\eqref{0-x-249293}, and using again~\eqref{PinFn-00},
we obtain that
\begin{equation}\label{9091.2394}
-H_P(0,\lambda/2)\ge
2\int_{{\mathcal{P}}} \frac{dy}{|P-y|^{2+s}}\ge
\frac{|{\mathcal{P}}|}{\big( C(\tilde\ell,a)P_1\big)^{2+s}}=
\frac{\e a P_1^2}{8\,\big( C(\tilde\ell,a)P_1\big)^{2+s}}=
\frac{C'(\tilde\ell,a)\,\e}{P_1^s},
\end{equation}
for some~$C'(\tilde\ell,a)>0$.

Now we take into account the contributions in~$B_{L/2}(P)\setminus B_{\lambda/2}(P)$.
To this end, let~$y\in B_{L/2}(P)\setminus B_{\lambda/2}(P)$.
If~$y\in\Pi_P\setminus F$, we see that
$$ \bar\ell y_1-\e a|y_1|-\e b|y_1|^{1+\alpha}\le\beta(y_1)\le y_2\le (\bar\ell+\e a) (y_1-P_1)+P_2=(\bar\ell+\e a) y_1.$$
Consequently,
\begin{equation}\label{-1x+s}
\begin{split}
H_P(\lambda/2,L/2)\,&\le 2
\int_{\big(
B_{L/2}(P)\setminus B_{\lambda/2}(P)\big)\cap(\Pi_P\setminus F)
} \frac{dy}{|P-y|^{2+s}}\\
&\le 2
\int_{\big(
B_{L/2}(P)\setminus B_{\lambda/2}(P)\big)\cap(\Pi_P\setminus F)
} \frac{dy}{|P_1-y_1|^{2+s}}\\
&\le 2
\int_{\{|y_1|\in(\lambda/2,L/2)\}}\left[
\int_{\bar\ell y_1-\e a|y_1|-\e b|y_1|^{1+\alpha}}^{(\bar\ell+\e a) y_1}
\frac{dy_2}{|P_1-y_1|^{2+s}}\right]\,dy_1\\
&\le 2
\int_{\{|y_1|>\lambda/2\}
} \frac{2\e a|y_1|+\e b|y_1|^{1+\alpha}}{|P_1-y_1|^{2+s}}\,dy_1\\
&\le 2
\int_{\{|t|>\lambda/4\}
} \frac{2\e a(|t|+\lambda)+\e b(|t|+\lambda)^{1+\alpha}}{|t|^{2+s}}\,dt\\
&\le C(s,a,b,\lambda,\alpha)\,\e,
\end{split}
\end{equation}
for some~$C(s,a,b,\lambda,\alpha)>0$.

Finally, we see that
$$ H_P(L/2,+\infty)\le
2\int_{\R^2\setminus B_{L/2}(P)} 
\frac{dy}{|P-y|^{2+s}}=\frac{C_0}{L^s},$$
for some constant~$C_0>0$.

Using this, \eqref{9091.2394} and~\eqref{-1x+s},
and recalling~\eqref{78-0123e}, we conclude that
\begin{eqnarray*}
H_P(0,+\infty)&=&
H_P(0,\lambda/2)+
H_P(\lambda/2,L/2)+
H_P(L/2,+\infty)\\&\le& -\frac{C'(\tilde\ell,a)\,\e}{P_1^s}+
C(s,a,b,\lambda,\alpha)\,\e+
\frac{C_0}{L^s}\\&\le&
\e\,
\left( -\frac{C'(\tilde\ell,a)}{\mu^s}+
C(s,a,b,\lambda,\alpha)+
\frac{C_0}{c^s}\right)\\
&\le&0,
\end{eqnarray*}
as long as
$$ \mu<\left(\frac{C'(\tilde\ell,a)}{C(s,a,b,\lambda,\alpha)+
\frac{C_0}{c^s}}\right)^{\frac1s},$$
and this completes the proof of~\eqref{TBOpAOjoi}, as desired.
\end{proof} 

\section{Boundary improvement of flatness for non-sticky points}\label{HAR4}

A classical method to obtain regularity results is based on the iteration
of a flatness principle, according to which graphs contained in a suitably narrow
cylinder are necessarily contained in the interior in an even narrower cylinder, with estimates.
In our situation, a result of this type cannot in general hold at the boundary,
due to the stickiness phenomenon. Nevertheless,
we will manage to produce a suitable version of an improvement of flatness scheme
in all the cases that do not exhibit boundary discontinuities. The technical details
of this scheme go as follows.

\begin{theorem}\label{TH6.8}
Let~$\alpha\in(0,s)$, $\ell\in\R$, $u:\R\to\R$, with~$u\in C^1([-h,0])$
for some~$h>0$, and
$$ E:=\{ (x_1,x_2)\in\R^2 {\mbox{ s.t. }} x_2<u(x_1)\}.$$
Let~$\e\in(0,1)$, and assume that~$E$ is locally $s$-minimal in~$(0,2^{\tilde k_0})\times\R$,
where
\begin{equation}\label{tilde k-0-0}
\tilde k_0:=\left\lceil \frac{-\log\e}{c_0}\right\rceil.\end{equation}
Suppose also that
\begin{equation}\label{NOSITIC} \lim_{x_1\searrow0}u(x_1)=0.\end{equation}
Then there exist~$\tilde\e_0$, $c_0\in(0,1)$ such that if~$\e\in(0,\tilde\e_0]$ the
following statement holds true.

If
\begin{eqnarray}
\label{DAQUIS}&&|u(x_1)-\ell x_1|\le \e^{\frac1{c_0}}\, |x_1|^{1+\alpha} {\mbox{ for all $x_1\in (-2^{\tilde k_0},0)$,}}\\
\label{UAdAFVD}&&|u(x_1)-\ell x_1|\le \e (2^k)^{1+\alpha} {\mbox{ for all $x_1\in (0,2^{k})$,
for all $k\in\{0,\dots, \tilde k_0\}$,}}
\end{eqnarray}
then, for all~$j\in\N$,
\begin{eqnarray}\label{8UJSJ-asod-1-cok}
&&|u(x_1)-\ell x_1|\le  \frac\e{(2^j)^{1+\alpha} }\qquad
{\mbox{ for all \;$x_1\in \left(0,\displaystyle\frac1{2^j}\right)$.}}
\end{eqnarray}
Moreover,
\begin{eqnarray}\label{8UJSJ-asod-1-cok2}
&&|u(x_1)-\ell x_1|\le  4\e\, x_1^{1+\alpha }\qquad
{\mbox{ for all \;$x_1\in \left(0,\displaystyle\frac1{2}\right)$.}}
\end{eqnarray}
\end{theorem}

\begin{proof} We observe that~\eqref{8UJSJ-asod-1-cok2}
follows from~\eqref{8UJSJ-asod-1-cok}, hence we focus on the proof
of~\eqref{8UJSJ-asod-1-cok}. Moreover, we can obtain~\eqref{8UJSJ-asod-1-cok}
iteratively, that is, we prove it for~$j=1$, then we scale and we apply
it recursively.

Hence, the claims in Theorem~\ref{TH6.8} are established once we prove that
\begin{eqnarray}\label{8UJSJ-asod-1}
&&|u(x_1)-\ell x_1|\le  \frac\e{2^{1+\alpha} }\qquad
{\mbox{ for all \;$x_1\in \left(0,\displaystyle\frac12\right)$.}}
\end{eqnarray}
To prove~\eqref{8UJSJ-asod-1}, we first exploit Lemma~\ref{-02390-23409}.
In particular, by~\eqref{XCAiqw}  and~\eqref{eqCAfddav},
fixed any~$\delta\in(0,1)$, to be taken suitably small in what follows,
we can write that
\begin{equation}\label{ExavxciECAALQ}
| u(x_1)-\ell x_1-\e\bar a x_1^{\frac{1+s}{2}} |\le C\e\,(
x_1^{\frac{3+s}{2}} +\delta),
\end{equation}
for all~$x_1\in(0,3)$, for some~$C>0$.

We claim that
\begin{equation}\label{abarra0}
\bar a=0.
\end{equation}
To prove this, we argue by contradiction and we suppose that~$\bar a>0$
(the case~$\bar a<0$ can be treated similarly).

Without loss of generality, we can assume that~$c_0\in(0,1)$ is so small that
\begin{eqnarray}\label{56wh92}
&&\frac{\log2}{c_0}\ge \frac1s,\\
\label{recahd}{\mbox{and }}&& 1-\log 2-c_0>0.
\end{eqnarray}
We consider the barrier~$\beta$ in Lemma~\ref{P:AL:IUBB},
with~$\bar\ell:=\ell+\e\delta$, $\lambda:=\min\left\{1,\frac{\bar a}{4C}\right\}$, $L:=\frac{1}{2\,\e^{\frac{\log2}{c_0}}}$, $a:=\frac{\bar a}8$, $b:=2^{2+\alpha}+\frac{3}{\lambda^{1+\alpha}}$, $c:=\frac12$.
In this way, condition~\eqref{78-0123e}
is fulfilled, thanks to~\eqref{56wh92},
and therefore Lemma~\ref{P:AL:IUBB} produces a suitable~$\mu\in(0,1)$
such that~\eqref{SPkedyria} holds true (and, in our setting, we are taking~$\e\ll\delta\ll\mu$).
We slide the barrier of Lemma~\ref{P:AL:IUBB} from below till it touches~$u$.
To this end, we point out that the original barrier~$\beta$ in
Lemma~\ref{P:AL:IUBB} lies below~$u$ in~$(-\infty,0)\cup[\mu,+\infty)$.

Indeed, if~$x_1\le-L$ or~$x_1\ge L$ the claim is obvious.
If instead~$x_1\in(-L,0)$, 
we observe that, by~\eqref{tilde k-0-0},
\begin{equation}\label{QUELLA} 2^{\tilde k_0}\ge2^{-1-\frac{\log \e}{c_0}}=\frac1{2\e^{\frac{\log2}{c_0}}}=L.
\end{equation}
Therefore, $x_1\in(-2^{\tilde k_0},0)$, and hence
we use~\eqref{DAQUIS} and see that
\begin{eqnarray*}&& \beta(x_1)-u(x_1)=\bar\ell x_1-u(x_1)=\ell x_1-u(x_1)-\e \delta|x_1|\\&&\qquad\le
\e^{\frac1{c_0}}\, |x_1|^{1+\alpha}-\e\delta|x_1|\le
\e\,|x_1|(L^\alpha\e^{\frac{1-c_0}{c_0}}-\delta)
\\&&\qquad=
\e\,|x_1|(\e^{\frac{1-c_0}{c_0}-\frac{\alpha\log2}{c_0}}-\delta)
<0,\end{eqnarray*}
as long as~$\e$ is small enough (recall that~$\frac{1-c_0}{c_0}-\frac{\alpha\log2}{c_0}>0$,
thanks to~\eqref{recahd}).

If instead~$x_1\in[\mu,\lambda]$, we exploit~\eqref{ExavxciECAALQ}
and we write that
\begin{eqnarray*}&& \beta(x_1)-u(x_1)=(\bar\ell +\e a)x_1-u(x_1)
= \left(\ell+\e\delta +\frac{\e \bar a}8\right)x_1-u(x_1)
\\&&\qquad\le
\left( \e\delta +\frac{\e \bar a}8\right)x_1
-\e\bar a x_1^{\frac{1+s}{2}} +C\e\,(x_1^{\frac{3+s}{2}} +\delta)
\\&&\qquad\le\e\,x_1^{\frac{1+s}{2}}\,\left(
\left( \delta +\frac{ \bar a}8\right)x_1^{\frac{1-s}{2}}
-\bar a  +C\,\left(x_1 +\frac\delta{x_1^{\frac{1+s}{2}} }\right)\right)
\\&&\qquad\le\e\,x_1^{\frac{1+s}{2}}\,\left(2
\left( \delta +\frac{ \bar a}8\right)
-\bar a  +C\,\left(\frac{\bar a}{4C} +\frac\delta{\mu^{\frac{1+s}{2}} }\right)\right)\\&&\qquad
=\e\,x_1^{\frac{1+s}{2}}\,\left(
-\frac{\bar a}{2}  +2\delta+\frac{C\delta}{\mu^{\frac{1+s}{2}} }\right)<0.
\end{eqnarray*}
Then, if~$x_1\in[\lambda,1]$,
we use~\eqref{UAdAFVD} with~$k:=0$ to write that~$|u(x_1)-\ell x_1|\le\e$
and thus
\begin{eqnarray*}
\beta(x_1)-u(x_1)=(\ell+\e\delta) x_1-\e b x_1^{1+\alpha}-u(x_1)\le
2\e-\e b\lambda^{1+\alpha}<0.
\end{eqnarray*}
Finally, if~$x_1\in[1,L]$, recalling~\eqref{QUELLA}, we have that~$x_1\in
[1, 2^{\tilde k_0}]$. Therefore,
we can take~$k\in\{0,\dots,\tilde k_0\}$
such that~$2^{k-1}\le x_1\le 2^k$. Then,
we recall~\eqref{UAdAFVD}
and we conclude that
$$ |u(x_1)-\ell x_1|\le 2^{1+\alpha}\e x_1^{1+\alpha}.$$
Accordingly,
\begin{eqnarray*}&& \beta(x_1)-u(x_1)=(\ell+\e\delta) x_1-\e b x_1^{1+\alpha}-u(x_1)\le
2^{1+\alpha}\e x_1^{1+\alpha}+\e\delta x_1-\e b x_1^{1+\alpha}
\\&&\qquad\le \e x_1^{1+\alpha}\,\left(
2^{1+\alpha} +\frac\delta{x_1^{\alpha}}- b \right)\le
\e x_1^{1+\alpha}\,\left(
2^{1+\alpha} +1- b \right)<0.
\end{eqnarray*}
With this, we can slide the barrier from below and deduce
from~\eqref{SPkedyria} that~$\beta\le u$ in the whole~$\R$.

Consequently, if we consider
the second blow-up~$E_{00}$ given in Lemma~\ref{secondo blow up},
we see that
\begin{equation}\label{0-2ejdjdj}
\begin{split}
E_{00}\;&\supseteq
\{ x=(x_1,x_2) {\mbox{ s.t. }} x_1>0 {\mbox{ and }} x_2<(\bar\ell+\e a)x_1\}
\\&=\{ x=(x_1,x_2) {\mbox{ s.t. }} x_1>0 {\mbox{ and }} x_2<(\ell+\e(\delta+ a))x_1\}.
\end{split}\end{equation}
On the other hand, by~\eqref{BL3-0} and~\eqref{DAQUIS},
$$ E_{00}\cap\{x_1<0\}=
\{ x=(x_1,x_2) {\mbox{ s.t. }} x_1<0 {\mbox{ and }} x_2< \ell x_1\}.$$
Using this, \eqref{0-2ejdjdj} and Theorem~\ref{P:AL}, we infer that
\begin{equation}\label{LPSIQA02}
(0,\delta')^2\subseteq E,\end{equation} for some~$\delta'>0$.
Indeed, Theorem~\ref{P:AL} would give that either~\eqref{LPSIQA02}
is satisfied, or~$(0,\delta')^2\subseteq E^c$, but the latter
possibility would give that~$E_{00}\cap\{x_1>0\}=\varnothing$,
which is in contradiction with~\eqref{0-2ejdjdj}.

As a consequence of~\eqref{LPSIQA02}, we have that
$$ \lim_{x_1\searrow0} u(x_1)\ge\delta',$$
which is in contradiction with~\eqref{NOSITIC} and thus
it completes the proof of~\eqref{abarra0}.

{F}rom this and~\eqref{ExavxciECAALQ}, we conclude that
\begin{equation}\label{8UJSJ-asod}
{\mbox{for all }}x_1\in(0,2^{-k_\star}),\qquad
| u(x_1)-\ell x_1 |\le C\e\,(x_1^{\frac{3+s}{2}} +\delta)\le\frac\e{2^{k_\star(1+\alpha)}},
\end{equation}
for a sufficiently large~$k_\star\in\N$, as long as~$\delta>0$ is chosen small enough. 

We remark that~\eqref{8UJSJ-asod}
is ``almost'' the desired result in~\eqref{8UJSJ-asod-1},
up to the ``delay'' produced by this~$k_\star\in\N$
(which now plays the role of a universal constant).
To remove such a delay, we take the freedom of further reducing the constant~$c_0$ in~\eqref{tilde k-0-0},
and we consider a rescaling by the factor~$2^{1-k_\star}$,
that is we consider the set~$E_\star:=2^{1-k_\star}E$
and the function~$u_\star(x_1):=2^{1-k_\star}u(2^{k_\star-1}x_1)$.

In this way, we have that~\eqref{DAQUIS}
and~\eqref{UAdAFVD} are fulfilled with~$\e$ replaced by~$\e_\star:=2^{\alpha(k_\star-1)}\e$.
To check this, we observe that, 
changing~$\e$ to~$\e_\star$ only affects the constant~$c_0$
in~\eqref{tilde k-0-0}, and therefore
if~$x_1\in(0,2^k)$, then~$2^{k_\star-1}x_1\in
(0,2^{k+k_\star-1})$ and so, by~\eqref{UAdAFVD},
$$ |u_\star(x_1)-\ell x_1|=
2^{1-k_\star}|u(2^{k_\star-1}x_1)-\ell(2^{k_\star-1}x_1)|\le
2^{1-k_\star}\e\,(2^{k+k_\star-1})^{1+\alpha}
= \e_\star\, (2^k)^{1+\alpha}, $$
which is~\eqref{UAdAFVD} with~$u$ and~$\e$ replaced by~$u_\star$
and~$\e_\star$, respectively.

Similarly, if~$x_1\in (-2^{\tilde k_0},0)$, we write~$2^{k_\star-1}
x_1\in (-2^{\tilde k_0+k_\star-1},0)$ and exploit~\eqref{DAQUIS}
to find that
\begin{eqnarray*}&& |u_\star(x_1)-\ell x_1|=
2^{1-k_\star}|u(2^{k_\star-1}x_1)-\ell(2^{k_\star-1}x_1)|\le
2^{1-k_\star}\, \e^{\frac1{c_0}}\, |2^{k_\star-1}x_1|^{1+\alpha} 
\\&&\qquad\qquad=
2^{\alpha(k_\star-1)}\, 2^{\frac{\alpha(1-k_\star)}{c_0}}\,\e_\star^{\frac1{c_0}}\, |x_1|^{1+\alpha}
\le \e_\star^{\frac1{c_0}}\, |x_1|^{1+\alpha} .
\end{eqnarray*}
This says that~\eqref{DAQUIS}
is satisfied
with~$u$ and~$\e$ replaced by~$u_\star$
and~$\e_\star$, respectively.

Then, we are in the position of
using~\eqref{8UJSJ-asod} with~$u$ and~$\e$
replaced by~$u_\star$ and~$\e_\star$, and accordingly we see that
\begin{eqnarray*}&&
\sup_{y_1\in(0,1/2)}| u(y_1)-\ell y_1 |=2^{k_\star-1}
\sup_{y_1\in(0,1/2)}|u_\star(2^{1-k_\star}y_1)-\ell(2^{1-k_\star}y_1)|\\
&&\qquad=2^{k_\star-1}
\sup_{x_1\in(0,2^{-k_\star})}|u_\star(x_1)-\ell x_1|\le2^{k_\star-1}\times
\frac{\e_\star}{2^{k_\star(1+\alpha)}}=\frac\e{2^{1+\alpha}}.
\end{eqnarray*}
This completes the proof of~\eqref{8UJSJ-asod-1}.
\end{proof}

{F}rom this, we obtain the following global regularity result:

\begin{theorem}\label{PROVV:Areg}
Let~$\alpha\in(0,s)$, $u:\R\to\R$, with
\begin{equation}\label{DAFUO}
u\in C^{1,\alpha}([-h,0])\end{equation}
for some~$h>0$, and
$$ E:=\{ (x_1,x_2)\in\R^2 {\mbox{ s.t. }} x_2<u(x_1)\}.$$
Assume that~$E$ is locally $s$-minimal in~$(0,1)\times\R$.
Suppose also that
\begin{equation}\label{765763jdd}\lim_{x_1\searrow0}u(x_1)=\lim_{x_1\nearrow0}u(x_1).\end{equation}
Then, $u\in C^{1,\alpha}(\left[-h,\frac12\right])$.
\end{theorem}

\begin{proof} Up to a vertical translation, we can suppose that
\begin{equation}\label{LIM73}
{\mbox{the
limit in~\eqref{765763jdd} is equal to~$0$,}}\end{equation} thus reducing to the setting in~\eqref{NOSITIC}.
Hence, the regularity in the interior is warranted by
the results in~\cites{MR2675483},
and the boundary regularity is a consequence of
Theorem~\ref{TH6.8}. Combining these two results, we obtain
the desired $C^{1,\alpha}$-regularity up to the boundary.
More precisely, one proves uniform pointwise
$C^{1,\alpha}$-regularity at any given point
by first exploiting the boundary improvement of flatness
at the origin given by
Theorem~\ref{TH6.8} up to  a suitable scale, then switching
the center of the improvement of flatness
and applying there the interior improvement of flatness given by Theorem~6.8
of~\cite{MR2675483}. 

The technical details of the proof go as follows.
First of all, we consider the double blow-up~$E_{00}$ in Lemma~\ref{secondo blow up}.
By~\eqref{LIM73} and Theorem~\ref{P:AL}, we obtain that
\begin{equation}\label{00pe}
{\mbox{$E_{00}$ is a halfplane, say~$E_{00}=\{ x_2=\ell x_1\}$,}}\end{equation}
for some~$\ell\in\R$. As a matter of fact, by~\eqref{DAFUO},
we have that~$\ell=u'(0^-)$.

Now we define~$\e_\star$ to be the minimum between the~$\tilde\e_0$
in Theorem~\ref{TH6.8} here and the corresponding quantity in
Theorem~6.8 of~\cite{MR2675483}. Similarly, we take~$\tilde k_0$
as in Theorem~\ref{TH6.8} (with~$\e$ there taken to be equal to~$\tilde\e_0$). 
Then, we take~$\tilde k_\star$ to be the maximum between such a~$\tilde k_0$
and the corresponding quantity in Theorem~6.8 of~\cite{MR2675483}.
{F}rom now on, $\e_\star$
and~$\tilde k_\star$ will be fixed quantities. Then, we use~\eqref{00pe} and
Corollary~\ref{Sempli-APP}, to find~$k_\star$, only depending on
the fixed~$\e_\star$ and~$\tilde k_\star$, such that~$\partial E_{k_\star}$
lies locally so close to~$\partial E_{00}$ that
\begin{equation}\label{MEG}
| u_\star(x_1)-\ell x_1|\le \frac{\e_\star}{ 2^{\tilde k_\star(1+\alpha)+10}}\qquad{\mbox{ for all }}x_1\in (-2^{\tilde k_\star+1},2^{\tilde k_\star+1}),
\end{equation}
where
\begin{equation}\label{POBIEY2} u_\star(x_1):= k_\star u\left(\frac{x_1}{k_\star}\right).\end{equation}
For further use, we remark that, in view of~\eqref{MEG}, for all~$k\in\{0,\dots,\tilde k_\star\}$,
\begin{equation}\label{BAT}
| u_\star(x_1)-\ell x_1|\le \frac{\e_\star }{{ 2^{\tilde k_\star(1+\alpha)+10}}}\le \frac{
\e_\star}{ { 2^{\tilde k_\star(1+\alpha)+10}} } (2^k)^{1+\alpha}
\qquad{\mbox{ for all }}x_1\in (0,2^{k}).
\end{equation}
Then, Theorem~\ref{PROVV:Areg} is proven once we show that,
for all~$\bar x_1$, $\bar y_1\in\left[0,\frac12\right]$, there exists~$\ell_{\bar y_1}\in\R$ such that
\begin{equation}\label{POBIEY}
|u_\star(\bar x_1)-u_\star(\bar y_1)-\ell_{\bar y_1} (\bar x_1-\bar y_1)|\le C\,|\bar x_1-\bar y_1|^{1+\alpha},
\end{equation}
for a suitable~$C>0$, since a similar estimate for~$u$, up to changing~$C$,
would follow directly from~\eqref{POBIEY} and~\eqref{POBIEY2}.

To prove this, we let~$d:=\frac{ |\bar x_1-\bar y_1| }{2}$
and~$z:=\frac{\bar x_1+\bar y_1}2$.
We distinguish two cases: either~$z\in[0, 2d]$, or~$z\in\left(2d, \frac12\right]$.
To start with, let us suppose that~$z\in[0, 2d]$.
In this case, we deduce from~\eqref{BAT} that
\begin{equation*}
| u_\star(x_1)-\ell x_1|\le \tilde\e_0 (2^k)^{1+\alpha}
\qquad{\mbox{ for all }}x_1\in (0,2^{k}).
\end{equation*}
Consequently, the assumption in~\eqref{UAdAFVD} is satisfied by~$u_\star$.
Furthermore, using~\eqref{DAFUO}, we have that, if~$x_1\in(-h,0)$,
$$ |u(x_1)-\ell x_1|=|u(x_1)-u(0)-u'(0^-) x|\le C |x_1|^{1+\alpha},$$
for some~$C>0$, and accordingly, if~$x_1\in (-h k_\star,0)$
$$ |u_\star(x_1)-\ell x_1|\le \frac{C}{k_\star^\alpha}\, |x_1|^{1+\alpha}.$$
This says that, possibly enlarging~$k_\star$ by a fixed amount,
also the condition in~\eqref{DAQUIS} is satisfied by~$u_\star$.

Hence, we are in the position of using Theorem~\ref{TH6.8} on~$u_\star$,
thus deducing from~\eqref{8UJSJ-asod-1-cok2} that, for all~$x_1\in\left(0,\displaystyle\frac1{2}\right)$,
$$ |u_\star(x_1)-\ell x_1|\le  4\tilde\e_0\, x_1^{1+\alpha }.$$
As a consequence,
$$ |u_\star(\bar x_1)-u_\star(\bar y_1)-\ell (\bar x_1-\bar y_1)|\le
|u_\star(\bar x_1)-\ell \bar x_1|+
|u_\star(\bar y_1)-\ell \bar y_1|\le 4\tilde\e_0\, \big(\bar x_1^{1+\alpha }+\bar y_1^{1+\alpha }\big).$$
Hence, since
$$ \frac{\bar x_1}2,\,\frac{\bar  y_1}2\le \frac{\bar x_1+\bar y_1}{2}=z\le 2d=|\bar x_1-\bar y_1|,$$
we conclude that
$$ |u_\star(\bar x_1)-u_\star(\bar y_1)-\ell (\bar x_1-\bar y_1)|\le
32 \tilde\e_0\, |\bar x_1-\bar y_1|^{1+\alpha },$$
and this establishes~\eqref{POBIEY} with~$\ell_{\bar y_1}:=\ell$
when~$z\in[0, 2d]$.

Hence, we suppose now that
\begin{equation}\label{CAS8123}
z\in\left(2d, \frac12\right]\end{equation}
and we prove~\eqref{POBIEY} also in this case.
To this end, we set
\begin{equation}\label{UTIL} \tilde u(x_1):=\frac{u_\star(z(1+x_1))}{z}.\end{equation}
Then, $\tilde u$ describes a nonlocal minimal graph in~$\left(-1,\frac1z-1\right)$,
which contains~$(-1,1)$. Moreover,
\begin{equation}\label{BAT2}\begin{split}& |\tilde u(x_1)-\tilde u(0)-\ell x_1|=\frac1z\,\big| u_\star(z(1+x_1))-z\ell x_1- u_\star(z)
\big|\\ &\qquad\le\frac1z\,\Big(\big| u_\star(z(1+x_1))-\ell (z(1+x_1))\big|+| u_\star(z)-\ell z|
\Big).\end{split}\end{equation}
On the other hand, by~\eqref{BAT}, and recalling also~\eqref{8UJSJ-asod-1-cok2},
for all~$x_1\in(0,2^{\tilde k_\star})$,
$$ |u_\star(x_1)-\ell x_1|\le \frac{\e_\star}{
{ 2^{\tilde k_\star(1+\alpha)+8}}} \,x_1^{1+\alpha}.$$
Hence, for all~$x_1\in(-1,1)$,
$$ \big| u_\star(z(1+x_1))-\ell (z(1+x_1))\big|\le \frac{\e_\star }{{ 2^{\tilde k_\star(1+\alpha)+6}}}\,z^{1+\alpha}(1+x_1)^{1+\alpha}
\le \frac{\e_\star}{{ 2^{\tilde k_\star(1+\alpha)+4}}} z^{1+\alpha}.
$$
This and~\eqref{BAT2} give that, if~$k\in \{0,\dots,\tilde k_\star\}$
and~$x_1\in(-2^{-k},2^{-k})$,
$$ |\tilde u(x_1)-\tilde u(0)-\ell x_1|\le \frac{\e_\star \,z^\alpha}{ 2^{\tilde k_\star(1+\alpha)}}\le
\frac{\e_\star\,z^\alpha}{ 2^{k(1+\alpha)}}.$$
Hence, by Theorem~6.8
of~\cite{MR2675483}, for all~$x_1$, $y_1\in\left(-\frac12,\frac12\right)$,
\begin{equation}\label{Po0A23}
|\tilde u(x_1)-\tilde u(y_1)-\tilde u'(y_1)(x_1-y_1)|\le C\e_\star\,z^\alpha\,|x_1-y_1|^{1+\alpha},\end{equation}
for some~$C>0$. 

Now, we suppose that~$\bar x_1\le \bar y_1$ (the case~$\bar x_1> \bar y_1$
can be treated in a similar way). Then, we have that~$\bar x_1=z-d$ and~$\bar y_1=z+d$.
Hence, since~$\frac{d}{z}<\frac12$ thanks to~\eqref{CAS8123},
we are in the position of utilizing~\eqref{UTIL} and~\eqref{Po0A23}
(taking~$\ell_{\bar y_1}:=\tilde u'(y_1)$,
with~$x_1:=-\frac{d}{z}$ and~$y_1:=\frac{d}{z}$)
and conclude that
\begin{eqnarray*}&&
|u_\star(\bar x_1)-u_\star(\bar y_1)-\ell_{\bar y_1} (\bar x_1-\bar y_1)|=
|u_\star(z-d)-u_\star(z+d)+2\ell_{\bar y_1} d|\\&&\qquad
= \left|z\tilde u\left(-\frac{d}{z}\right)-z\tilde u\left(\frac{d}{z}\right)+2\ell_{\bar y_1} d\right|\le
C\e_\star\,z^{1+\alpha}\,\left(\frac{2d}z\right)^{1+\alpha}\le  4C\e_\star\,d^{1+\alpha}
.\end{eqnarray*}
Up to renaming~$C$,
this completes the proof of~\eqref{POBIEY}, as desired.
\end{proof}

\section{Boundary regularity, and proof of Theorem~\ref{BR}}\label{HAR5}

While Theorem~\ref{PROVV:Areg} has its own interest, since
it says that nonlocal minimal graphs which are not boundary discontinuous
are necessarily $C^{1,\alpha}$ up to the boundary,
for concrete applications of such a result it is essential
to improve the value of~$\alpha$ (provided that the exterior datum
is regular enough): in particular,
the H\"older exponent needed to pass the equation
to the limit at boundary points needs to be higher than~$s$,
while the one coming from Theorem~\ref{PROVV:Areg} happens
to be less than~$s$. {F}rom the technical point of view, this
is due to the possible growth of the solution at infinity and the long range
interactions of the problem.

Our objective is now to modify and extend some linearization methods
in~\cite{MR3331523}, combined with
some precise boundary asymptotics of fractional Laplace equations,
taking into account the additional boundary effects,
to improve the H\"older exponent in Theorem~\ref{PROVV:Areg}
and finally obtain Theorem~\ref{BR}. 

\begin{proof}[Proof of Theorem~\ref{BR}] 
We take~$\zeta\in C^\infty_0\left( \left[
-\frac{h}2,\frac{h}2\right],\,[0,1]\right)$ with~$\zeta=1$
in~$\left( -\frac{h}4,\frac{h}4\right)$,
and we define
$$ v(x_1):=\zeta(x_1)\,u(x_1) .$$
Using formula~(49) in~\cite{MR3331523}, and recalling the notation in~\eqref{DEF:F},
we have that, for all~$x_1\in(0,1)$,
\begin{equation}\label{EULAHD0} \int_{\R} F\left(\frac{u(x_1+ t)-u(x_1)}{|t|}
\right)\,\frac{dt}{|t|^{1+s}}=0,\end{equation}
in the principal value sense.

We also recall that~$u'(x_1)$ is well defined,
thanks to the interior regularity in~\cites{MR3090533},
and therefore, by odd symmetry, and using that~$F$ is bounded,
for any $J\subseteq\R$ which is symmetric with respect to the origin,
we obtain that
\begin{equation}\label{KS-OAKS92348}
\begin{split}&
\int_{J} F\left(\frac{u(x_1+ t)-u(x_1)}{|t|}
\right)\,\frac{dt}{|t|^{1+s}}\\
=\;& \int_{J} \left[ F\left(\frac{u(x_1+ t)-u(x_1)}{|t|}
\right) -F\left(\frac{u'(x_1)\,t}{|t|}
\right)\right]\,\frac{dt}{|t|^{1+s}}
\\=\;&
\int_{J} \Big(u(x_1+ t)-u(x_1)- u'(x_1)\,t\Big)
\,\frac{a(x_1,t)\,dt}{|t|^{2+s}}\\=\;&
\int_{J} \delta(x_1,t)\,\frac{a(x_1,t)\,dt}{|t|^{2+s}}
,\end{split}\end{equation}
where
\begin{eqnarray*}&&a(x_1,t):=\int_0^1
F'\left(\frac{\tau \big(u(x_1+ t)-u(x_1)\big)+ (1-\tau)u'(x_1)\,t}{|t|}
\right)\,d\tau\\{\mbox{and }}&&
\delta(x_1,t):=u(x_1+ t)-u(x_1)- u'(x_1)\,t.\end{eqnarray*}
In particular, choosing~$J:=\R$, we deduce from~\eqref{EULAHD0} and~\eqref{KS-OAKS92348} that
$$ \int_{\R} \delta(x_1,t)\,\frac{a(x_1,t)\,dt}{|t|^{2+s}}=0.$$
As a consequence, if~$\sigma:=\frac{1+s}{2}$,
$I:=\left(-\frac{h}4,\frac{h}4\right)$, and~$x_1\in\left(0,\frac{h}{4}\right)$,
\begin{equation}\label{LEQUASIGA}
\begin{split}&
-F'\big( u'(x_1)\big)\,
(-\Delta)^\sigma v(x_1)\\=\;&F'\big( u'(x_1)\big)\,
\int_{\R} \Big(v(x_1+ t)-v(x_1)-v'(x_1)\,t\Big)
\,\frac{ dt}{|t|^{2+s}}\\
=\;&F'\big( u'(x_1)\big)\,\int_{I} \Big(u(x_1+ t)-u(x_1)-u'(x_1)\,t\Big)
\,\frac{ dt}{|t|^{2+s}}\\&\qquad\qquad+F'\big( u'(x_1)\big)\,
\int_{\R\setminus I} \Big(v(x_1+ t)-u(x_1)-u'(x_1)\,t\Big)
\,\frac{ dt}{|t|^{2+s}}\\
=\;&\int_{I} \delta(x_1,t)\,\frac{a(x_1,t)\,dt}{|t|^{2+s}}\,dt
+
\int_{I} 
\delta(x_1,t)\,\frac{F'\big( u'(x_1)\big)-a(x_1,t)}{|t|^{2+s}}\,dt
\\&\qquad\qquad+F'\big( u'(x_1)\big)\,
\int_{\R\setminus I} \Big(v(x_1+ t)-u(x_1)-u'(x_1)\,t\Big)
\,\frac{ dt}{|t|^{2+s}}\\
=\;&f(x_1),
\end{split}
\end{equation}
where
\begin{equation}\label{CHA823r44rfr}
\begin{split}&
f_1(x_1):=-\int_{\R\setminus I} \delta(x_1,t)\,\frac{a(x_1,t)\,dt}{|t|^{2+s}},\\
& f_2(x_1):=\int_{I} 
\delta(x_1,t)\,\frac{F'\big( u'(x_1)\big)-a(x_1,t)}{|t|^{2+s}}\,dt\\
& f_3(x_1):=F'\big( u'(x_1)\big)\,
\int_{\R\setminus I} \Big(v(x_1+ t)-u(x_1)-u'(x_1)\,t\Big)
\,\frac{ dt}{|t|^{2+s}},\\
{\mbox{and }}\;&f(x_1):=f_1(x_1)+f_2(x_1)+f_3(x_1).
\end{split}\end{equation} 
Now we use the following notation: given~$r\in\R$ we denote by~$\underline{r}$
a real number strictly smaller than~$r$ which can be taken as close as we want to~$r$.
For instance, the thesis in Theorem~\ref{PROVV:Areg}
can be written as~$u\in C^{1,\underline{s}}([0,1/2])$,
with continuity of the derivative
of~$u$ at the origin with respect to the external datum.
We claim that if~$u\in L^\infty\left([0,h]\right)\cap
C^{1,\alphyy}\left(\left[0,\frac{h}{2}\right]\right)$ for some
\begin{equation}\label{RESTRIZ}
\alphyy\in\left(\frac{s}2,1\right)\setminus\{s\},\end{equation}
then
\begin{equation}\label{BSC}\begin{split}
&f\in C^{\kappa(\alphyy)}\left(\left[0,\frac{h}{4}\right]\right), {\mbox{ with }}
\| f\|_{ C^{\kappa(\alphyy)}\left(\left[0,\frac{h}{4}\right]\right) }\le C_{\alphyy,s},\\
{\mbox{where }}\;&\kappa(\alphyy):=\min\{ {\alphyy},
2\alphyy-s\},
\end{split}
\end{equation}
for some~$C_{\alphyy,s}>0$ depending only on~$s$, $\alphyy$,
$h$, and~$\|u\|_{C^{1,\alphyy}\left(\left[0,\frac{h}{2}\right]\right)}$.

The proof of~\eqref{BSC} is rather complicated and it is based
on a long and delicate computation. Not to interrupt the flow
of ideas, we postpone this proof to Appendix~\ref{TED}.

We now define
$$ f_\star(x_1):=-\frac{f(x_1)}{F'(u(x_1))},$$
and we deduce from~\eqref{BSC} that
\begin{equation}\label{SERstar} f_\star\in C^{\kappa(\alphyy)}\left(\left[0,\frac{h}{4}\right]\right).\end{equation}
Now we recall that~$v$ belongs to~$C^{1,\betxxalf}\left( [-h,0]\right)$,
and also to~$C^\infty\left( \left( \frac{h}{4},2h\right)\right)$,
thanks to the regularity of~$u$ established in~\cite{MR3090533, MR3331523}.
Therefore,
we can
take a function
\begin{equation}\label{DEFINQUOVTAI}
\tilde v\in C^{1,\betxxalf}\left( \left[-\frac{h}2,\frac{h}2\right]\right)\end{equation}
such that~$\tilde v= v$ outside~$\left(0,\frac{h}4\right)$.

We observe that, by construction, $\tilde v(0)=v(0)=u(0)=0$
and~$\tilde v'(0)=v'(0)=u'(0)$. Hence, as~$|x_1|\to0$,
\begin{equation}\label{MAndfiqwgfu834}
\tilde v(x_1)=u'(0)x_1+O(|x_1|^{1+\betxxalf}).
\end{equation}
Furthermore, by~\eqref{DEFINQUOVTAI}, see e.g. Proposition~2.1.8
in~\cite{MR2707618}, we have that~$g_\star:=(-\Delta)^\sigma \tilde v\in
C^{1+\betxxalf-2\sigma}\left(\left[0,\frac{h}4\right]\right)$.
Hence, setting~$h_\star:=f_\star-g_\star$, we deduce from~\eqref{SERstar}
that
\begin{equation}\label{SERstar-ira} h_\star\in C^{\vartheta}
\left(\left[0,\frac{h}{4}\right]\right),\qquad{\mbox{
with }}\vartheta:=\min\{\sigma,\,\kappa(\alphyy),\,1+\betxxalf-2\sigma\}.\end{equation}
Notice that~$\vartheta\in(0,1)$, since~$\betxxalf>s$.

Let also~$W:= v-\tilde v$ and~$\tilde W(x_1):=x_1^{-\sigma} W(x_1)$.
By~\eqref{LEQUASIGA}, we see that
\begin{equation*} \begin{cases}
(-\Delta)^\sigma W=h_\star & {\mbox{ in }}\left(0,\displaystyle\frac{h}4\right),\\
W=0 & {\mbox{ in }}\R\setminus\left(0,\displaystyle\frac{h}4\right).
\end{cases}\end{equation*}
Hence, using~\eqref{SERstar-ira} and Lemma~\ref{CLAUDIA2},
we deduce that
\begin{equation}\label{TROP} v(x_1)-\tilde v(x_1)=W(x_1)=x_1^{\sigma}\tilde W(x_1)=
x_1^{\sigma} \Big(a_0+ O\big(x_1^{\mu}\big)\Big)
\end{equation}
as~$x_1\searrow0$,
for a suitable~$a_0\in\R$, with
$$ \mu:=\min\{1,\,{\underline\vartheta+\sigma}\}=
\min\{1,\,\underline{2\sigma},\,\kappa(\alphyy)+\underline{\sigma}
,\,1+\underline{\betxxalf}-\sigma\}.$$
On the other hand, by \eqref{MAndfiqwgfu834} and
Theorem~\ref{PROVV:Areg}, we know that
\begin{equation}\label{Xcsv73-923948-duqu} v(x_1)-\tilde v(x_1)=u(x_1)-u'(0)x_1+O(x_1^{1+\betxxalf}) =
O(x_1^{1+\alphyy}),\end{equation}
as~$x_1\searrow0$.
Combining this and~\eqref{TROP}, we conclude that
\begin{eqnarray*}a_0=
\lim_{x_1\searrow0} a_0+ O\big(x_1^{\mu}\big)=
\lim_{x_1\searrow0}
x_1^{-\sigma}\Big(v(x_1)-\tilde v(x_1)\Big)=
\lim_{x_1\searrow0} O(x_1^{1+\alphyy-\sigma})=0.
\end{eqnarray*}
Therefore, we can write~\eqref{TROP} as~$W(x_1)=O\big(x_1^{\mu+\sigma}\big)$,
and then exploit Lemma~\ref{REGONIOA} to deduce that~$W\in C^{\mu+\sigma}\left(\left[0,\frac{h}{8}\right]\right)$.
As a consequence, recalling~\eqref{DEFINQUOVTAI},
we find that~$v=W+\tilde v\in
C^{\xi(\alphyy)}\left(\left[0,\frac{h}{8}\right]\right)$,
with
$$ \xi(\alphyy):=\min\{ \mu+\sigma, \,1+\betxxalf\}=
\min\left\{\frac{3+s}{2},\,
1+\underline\betxxalf,\,\kappa(\alphyy)+1+\underline{s}\right\}\in(1+s,2).
$$
This gives that
\begin{equation}\label{BBOTAL}
u\in C^{1,\xi(\alphyy)-1}\left(\left[0,\frac{h}{2}\right]\right).
\end{equation}
We can now bootstrap this procedure till we reach an almost optimal
H\"older exponent for~$u'$. Namely, in view of Theorem~\ref{PROVV:Areg},
we know that~$u\in C^{1,\alpha_1}\left(\left[0,\frac{h}{2}\right]\right)$,
with~$\alpha_1:=\underline{s}$. Then, by~\eqref{BBOTAL}, we obtain that~$u\in C^{1,\alpha_2}\left(\left[0,\frac{h}{2}\right]\right)$,
with~$\alpha_2:=\min\left\{\frac{1+s}{2},\,
\betxxalf,\,\kappa(\alpha_1)+s\right\}$.

Iterating this, for all~$k\in\N$ with~$k\ge2$, we conclude that
\begin{equation}\label{56572394}
u\in C^{1,\alpha_k}\left(\left[0,\frac{h}{2}\right]\right),\qquad{\mbox{
with }}\qquad \alpha_k:=\min\left\{\frac{1+s}{2},\,
\betxxalf,\,\kappa(\alpha_{k-1})+s\right\}.\end{equation}
We claim that there exists~$k\in\N$, $k\le 4+\frac1s$, such that
\begin{equation}\label{SOtgahsu33-2}
\kappa(\alpha_{k-1})+s\ge \min\left\{\frac{1+s}{2},\,
\betxxalf\right\}
.\end{equation}
Indeed, if not, we have that
\begin{equation}\label{SOtgahsu33}
\kappa(\alpha_{k-1})+s< \min\left\{\frac{1+s}{2},\,
\betxxalf\right\}
,\end{equation}
for all~$k\le4+\frac1s$.
Accordingly,
we obtain that~$\alpha_k=\kappa(\alpha_{k-1})+s$,
for all~$k\le 4+\frac1s$.
In addition, by~\eqref{56572394}, we have that~$\alpha_k> s$
for all~$k\ge2$, and therefore, by~\eqref{BSC}, if~$k\ge3$,
$$ \kappa(\alpha_{k-1})=\min\{ {\alpha_{k-1}},
2\alpha_{k-1}-s\}= {\alpha_{k-1}}.$$
Hence, for all $k\in \N\cap[3,4+\frac1s]$,
and all~$j\in\N $, with~$j\le k-2$,
$$ \alpha_k= {\alpha_{k-1}}+s\ge \alpha_{k-1}+\frac{s}{2}\ge\dots\ge
\alpha_{k-j}+\frac{sj}{2}.$$
In particular, taking~$k\in \left( 3+\frac1s,4+\frac1s\right]$
$$ \alpha_k\ge
\alpha_{2}+\frac{s(k-2)}{2}\ge\frac{s(k-2)}{2}\ge\frac{1+s}{2},$$
which is in contradiction with~\eqref{SOtgahsu33}, and so it proves~\eqref{SOtgahsu33-2}.

Hence, by~\eqref{56572394} and~\eqref{SOtgahsu33-2},
\begin{equation*}
u\in C^{1,\gamma}\left(\left[0,\frac{h}{2}\right]\right),\qquad{\mbox{
with }}\qquad \gamma:=\min\left\{\frac{1+s}{2},\,
\betxxalf\right\}.\end{equation*}
This completes the proof of Theorem~\ref{BR}.
\end{proof}

\section{Boundary validity of the nonlocal curvature equation,
and proof of Theorem~\ref{ELBOU}}\label{SDRA6}

With the previous work, we are now ready to obtain that the Euler-Lagrange
equation is satisfied pointwise at all points which are accessible from the interior
(being evidently false elsewhere).

\begin{proof}[Proof of Theorem~\ref{ELBOU}]
For interior points, that is when~$x\in(\partial E)\cap((0,1)\times\R)$,
in view of~\cite{MR2675483},
we know that~\eqref{USEF} is satisfied in the viscosity sense.
But since in this case~$(\partial E)\cap((0,1)\times\R)$
is locally a $C^\infty$ set, thanks to~\cites{MR3090533, MR3331523},
we conclude that~\eqref{USEF} is also satisfied in the
pointwise sense at every point of~$x\in(\partial E)\cap((0,1)\times\R)$.

Hence, to complete the proof of Theorem~\ref{ELBOU},
we only have to take into account the boundary points,
i.e., the points which lie on the boundary of the slab and can be reached
by points lying in~$(\partial E)\cap((0,1)\times\R)$.

For this, we distinguish two cases. If~$u$ is discontinuous at the boundary point, the result follows
by suitably applying the results in~\cite{MR3532394},
see e.g. Theorem~B.9 in~\cite{CLAUDIALUCA} for a detailed statement.

If instead~$u$ is continuous at the boundary point, we first write~\eqref{USEF}
at interior points. That is, using the notation in~\eqref{DEF:F}
and recalling
formula~(49) in~\cite{MR3331523}, we have that
\begin{equation}\label{EL}\begin{split}&\int_{\R}
\left[ F\left(\frac{u(x_1+w)-u(x_1)}{|w|} \right)
-F\left(\frac{u'(x_1)\,w}{|w|} \right)
\right]\,\frac{dw}{|w|^{1+s}}
\\&\qquad=
\int_{\R} F\left(\frac{u(x_1+w)-u(x_1)}{|w|} \right)\,\frac{dw}{|w|^{1+s}}=0,
\end{split}\end{equation}
for all~$x_1\in (0,1)$.

We also observe that, for small~$x_1>0$,
\begin{equation}\label{DCS} \begin{split}&
\left[ F\left(\frac{u(x_1+w)-u(x_1)}{|w|} \right)
-F\left(\frac{u'(x_1)\,w}{|w|} \right)
\right]\,\frac{1}{|w|^{1+s}}\\&\qquad\qquad\le C\,
\min\left\{\frac{1}{|w|^{1+s}}
,\,\frac{1}{|w|^{1+s-\min\left\{ \betxxalf,\frac{1+s}2\right\} }}\right\}
,\end{split}\end{equation}
for some~$C>0$,
thanks to Theorem~\ref{BR}, and the function on the right hand side of~\eqref{DCS}
belongs to~$L^1(\R)$.

As a consequence, 
we obtain the desired claim in~\eqref{USEF}
taking the limit in~\eqref{EL},
by the Dominated
Convergence Theorem.
\end{proof}

\section{Genericity of the stickiness phenomenon,
and proof of Theorem~\ref{GENER}}\label{SDRA7}

Now we are ready to prove our main result concerning the genericity
of boundary discontinuities for nonlocal minimal graphs in the plane.

\begin{proof}[Proof of Theorem~\ref{GENER}] First of all, we observe that, for all~$t\ge0$ and~$x_1\in\R$,
\begin{equation}\label{E1}
u(x_1,t)\ge u(x_1,0).
\end{equation}
To prove this, we recall
Lemma~3.3 in~\cite{MR3516886}, according to which
$$ (0,1)\times(-\infty, -M_t)\subseteq
E_t\cap\Omega\subseteq (0,1)\times(-\infty, M_t),$$
for some~$M_t>0$. As a consequence,
if~$\tau>M_t+M_0$, we have that
$$ E_t+\tau e_2 \supset E_0,$$
where, as customary, we set~$e_2:=(0,1)$.

Thus, we reduce $\tau$ till the first contact between the two sets.
Since $\varphi$ is nonnegative and supported away from~$(-d,1+d)$, we have that
this contact~$\tau$ is necessarily nonnegative.
We show that~$\tau=0$, which in turn proves~\eqref{E1}. Indeed, if~$\tau>0$,
the touching point~$p$
between~$E_t+\tau e_2$ and~$E_0$ must necessarily occur in~$\overline\Omega$.
As a consequence
\begin{eqnarray*}&& 0=H^s_{E_t+\tau e_2}(p)=\int_{\R^2}\frac{
\chi_{\R^2\setminus(E_t+\tau e_2)}(y)-\chi_{E_t+\tau e_2}(y)}{|p-y|^{2+s}}\,dy\\&&\qquad
\le \int_{\R^2}\frac{
\chi_{\R^2\setminus E_0}(y)-\chi_{E_0}(y)}{|p-y|^{2+s}}\,dy=H^s_{E_0}(p)=0.\end{eqnarray*}
In particular, this says that~$E_t+\tau e_2$ and~$E_0$ must coincide,
which is impossible since~$\varphi$ does not vanish identically.
This proves~\eqref{E1}. 

Now, we focus on the proof of~\eqref{TB}.
To prove~\eqref{TB}, we assume the converse inequality and we show that~$\varphi$ vanishes
identically, which is against our assumption.

More precisely, if~\eqref{TB} were not true, then there would be~$T>0$
such that for all~$t\in[0,T]$
$$ \limsup_{{x_1\searrow0}} u(x_1,t)\le v(0).$$
Hence, by~\eqref{TB0} and~\eqref{E1},
$$ \limsup_{{x_1\searrow0}} u(x_1,t)=v(0).$$
Then, by Theorem~\ref{ELBOU}, we would find that
\begin{equation*}
\int_{\R} F\left(\frac{u(w,t)-v(0)}{|w|} \right)\,\frac{dw}{|w|^{1+s}}=0,
\end{equation*}
for all~$t\in[0,T]$, and therefore
\begin{equation}\label{DC1} \int_{\R} \left[ F\left(\frac{u(w,t)-v(0)}{|w|} \right)-
F\left(\frac{u(w,0)-v(0)}{|w|} \right)\right]\,\frac{dw}{|w|^{1+s}}=0.\end{equation}
On the other hand,
recalling~\eqref{DEF:F}, we notice that~$F$ is strictly increasing.
Consequently, by~\eqref{E1}, we would find that
\begin{equation}\label{DC2} F\left(\frac{u(w,t)-v(0)}{|w|} \right)-
F\left(\frac{u(w,0)-v(0)}{|w|} \right)\ge0,\end{equation}
with strict inequality unless
\begin{equation}\label{DC3}
{\mbox{$u(w,t)=u(w,0)$ for all~$w\in\R$.
}}\end{equation}
In particular, since the strict inequality in~\eqref{DC2} is excluded by~\eqref{DC1},
we obtain that~\eqref{DC3} is necessarily satisfied, for all~$t\in[0,T]$.

Taking~$w$ outside~$(0,1)$, we would thus conclude that
$$ v(w)+t\varphi(w)=u(w,t)=u(w,0)=v(w),$$
for all~$t\in[0,T]$,
which would give that~$\varphi$ vanishes identically, against our assumption.
\end{proof}

\begin{appendix}

\section{Blow-up methods and proofs of the technical statements in Section~\ref{SE:1}}\label{TECH}

This appendix contains the proofs of the auxiliary results stated in Section~\ref{SE:1}.
The details of this proofs are not easily accessible in the literature,
since they deal with the newly explored setting of boundary blow-up methods
for nonlocal minimal surfaces, but we postponed these technical
arguments not to interrupt the main line of reasoning.

\begin{proof}[Proof of Lemma~\ref{primo blow up}] We observe that~$E\setminus\Omega$
is a Lipschitz set in~$B_{r_0}$, therefore, for all~$r\in(0,r_0)$,
\begin{equation}\label{GHAs1}
\begin{split}
&\iint_{(\Omega\cap B_r)\times(\Omega^c\cap B_r)}\frac{dx\,dy}{|x-y|^{2+s}}\le Cr^{2-s}
\\{\mbox{and}}\qquad&
\iint_{(E\cap\Omega^c\cap B_r)\times(E^c\cap\Omega^c\cap B_r)}\frac{dx\,dy}{|x-y|^{2+s}}\le Cr^{2-s}
,\end{split}\end{equation}
for some~$C>0$. 

Also, for all~$r\in(0,r_0)$ we have that the set~$G_r:=E\setminus(B_r\cap\Omega)$
coincides with~$E$ outside~$\Omega$ and therefore~${\rm Per}_s(E,\Omega)\le
{\rm Per}_s(G_r,\Omega)$, which gives that
\begin{equation*}
\iint_{(E\cap B_r\cap\Omega)\times E^c}\frac{dx\,dy}{|x-y|^{2+s}}\le 
\iint_{(E\cap B_r\cap\Omega)\times G_r}\frac{dx\,dy}{|x-y|^{2+s}}.
\end{equation*}
As a consequence,
\begin{equation}\label{GHAs2}
\begin{split}&
\iint_{(E\cap B_r\cap\Omega)\times E^c}\frac{dx\,dy}{|x-y|^{2+s}}\\
\le \;&
\iint_{(E\cap B_r\cap\Omega)\times (\Omega^c\cap B_r)}\frac{dx\,dy}{|x-y|^{2+s}}
+\iint_{(E\cap B_r\cap\Omega)\times B_r^c}\frac{dx\,dy}{|x-y|^{2+s}}\\
\le \;&
\iint_{(\Omega\cap B_r)\times (\Omega^c\cap B_r)}\frac{dx\,dy}{|x-y|^{2+s}}
+\iint_{B_r\times B_r^c}\frac{dx\,dy}{|x-y|^{2+s}}\\ \le\;&Cr^{2-s},
\end{split}
\end{equation}
where~\eqref{GHAs1} has been used in the last step, and~$C>0$ has been renamed.

In addition,
\begin{eqnarray*}
\psi(r)&:=&\iint_{(E\cap B_r)\times (E^c\cap B_r)}\frac{dx\,dy}{|x-y|^{2+s}}\\
&=&
\iint_{(E\cap B_r\cap\Omega)\times (E^c\cap B_r\cap\Omega)}\frac{dx\,dy}{|x-y|^{2+s}}
+\iint_{(E\cap B_r\cap\Omega)\times (E^c\cap B_r\cap\Omega^c)}\frac{dx\,dy}{|x-y|^{2+s}}
\\&&+\iint_{(E\cap B_r\cap\Omega^c)\times (E^c\cap B_r\cap\Omega)}\frac{dx\,dy}{|x-y|^{2+s}}
+\iint_{(E\cap B_r\cap\Omega^c)\times (E^c\cap B_r\cap\Omega^c)}\frac{dx\,dy}{|x-y|^{2+s}}\\
&\le&
\iint_{(E\cap B_r\cap\Omega)\times E^c}\frac{dx\,dy}{|x-y|^{2+s}}
+2\iint_{(\Omega\cap B_r)\times (\Omega^c\cap B_r)}\frac{dx\,dy}{|x-y|^{2+s}}
\\&&+\iint_{(E\cap B_r\cap\Omega^c)\times (E^c\cap B_r\cap\Omega^c)}\frac{dx\,dy}{|x-y|^{2+s}}.
\end{eqnarray*}
Combining this inequality with~\eqref{GHAs1} and~\eqref{GHAs2}, we conclude that
\begin{equation}\label{GHAs3}
\psi(r)\le Cr^{2-s},
\end{equation}
up to renaming~$C>0$.

Now we take~$R>0$ and we observe that
\begin{eqnarray*}&&
\frac12\iint_{B_R\times B_R}\frac{|\chi_{E_k}(x)-\chi_{E_k}(y)|^2}{|x-y|^{2+s}}\,dx\,dy
=\iint_{(E_k\cap B_R)\times(E_k^c\cap B_R)}\frac{dx\,dy}{|x-y|^{2+s}}\\&&\qquad
=k^{2-s}\iint_{(E\cap B_{R/k})\times(E^c\cap B_{R/k})}\frac{dx\,dy}{|x-y|^{2+s}}=
k^{2-s}\,\psi(R/k).
\end{eqnarray*}
Consequently, for $k$ sufficiently large, recalling~\eqref{GHAs3} we obtain that
$$ \frac12\iint_{B_R\times B_R}\frac{|\chi_{E_k}(x)-\chi_{E_k}(y)|^2}{|x-y|^{2+s}}\,dx\,dy\le
CR^{2-s}.$$
Hence, by fractional Sobolev embeddings, we conclude that, up to a subsequence,
$\chi_{E_k}$ converges to some function~$g$ in~$L^1_{\rm loc}(\R^2)$
and also a.e.; in this way, since~$\chi_{E_k}(x)\in\{0,1\}$ for all~$x\in\R^2$,
at each point~$p$ for which
$$ g(p)=\lim_{k\to+\infty}\chi_{E_k}(p),$$
we obtain that~$g(p)\in\{0,1\}$. In particular, $g$ takes value in~$\{0,1\}$, up to
a negligible set, and therefore we can define~$E_0:=\{g=1\}$ and
obtain~\eqref{BL1}, as desired.

Now we prove~\eqref{BL2}. For this, we consider a ball~$B_r(p)\Subset\{x_1>0\}$.
We observe that, if~$k$ is large enough,
\begin{equation*}
B_r(p)\subseteq\Omega_k:=k\Omega=(0,k)\times\R.
\end{equation*}
and hence the set~$E_k$
is $s$-minimal in~$B_r(p)$.
Then, from \eqref{BL1} and
Theorem~3.3 in~\cite{MR2675483} we obtain
that~$E_0$ is $s$-minimal in~$B_r(p)$, which establishes~\eqref{BL2}.

Now we prove~\eqref{BL3}. For this, we take~$Q=(Q_1,Q_2)\in E_0\cap\{x_1<0\}$
(the case~$Q\in E_0^c\cap\{x_1<0\}$ can be treated similarly), and 
we show that
\begin{equation}\label{PoA}
Q_2\le v'(0)\,Q_1.
\end{equation}
We have that~$\frac{Q}{k}\in\Omega^c$. Also, by~\eqref{BL1},
up to negligible sets, we can assume that
$$ 1=\chi_{E_0}(Q)=\lim_{k\to+\infty}\chi_{E_k}(Q),$$
hence~$\chi_{E_k}(Q)=1$ if~$k$ is large enough, and accordingly~$\frac{Q}{k}\in E$.
This gives that
$$ \frac{Q_2}{k}<v\left(\frac{Q_1}k\right)=
v\left(\frac{Q_1}k\right)-v(0).$$
Multiplying this inequality by~$k$ and taking the limit as~$k\to+\infty$,
we obtain~\eqref{PoA}.

By~\eqref{PoA} we have obtained that
$$ E_0\cap\{x_1<0\}\subseteq\left\{ x_2\le v'(0)\,x_1\right\},$$
and, similarly, that
$$ E_0\cap\{x_1<0\}\supseteq\left\{ x_2\ge v'(0)\,x_1\right\},$$
and these two inclusions give~\eqref{BL3}
(up to negligible sets).
\end{proof}

\begin{proof}[Proof of Lemma~\ref{secondo blow up}] We can apply Lemma~\ref{primo blow up}
to the set~$E_0$. In this case, the function~$v(x_1)$ can be replaced
by the function~$v'(0)\,x_1$.
Then, the claims in~\eqref{BL1-0}, \eqref{BL2-0} and~\eqref{BL3-0}
plainly follow, respectively, from~\eqref{BL1}, \eqref{BL2} and~\eqref{BL3}.

It remains to show that~$E_{00}$ is a cone. For this,
we observe that the set~$E_0\cap\{x_1<0\}$ is preserved under dilations.
As a consequence, the monotonicity formula of Theorem~8.1 in~\cite{MR2675483}
can be applied to~$E_0$ (even if the center
lies in this case on the boundary of the domain,
since the same proof in~\cite{MR2675483}
would work once the data outside the domain are invariant under dilations).
As a consequence, $E_{00}$ is necessarily a cone, in view of
Theorem~9.2 in~\cite{MR2675483}.
\end{proof}

\begin{proof}[Proof of Lemma~\ref{Sempli}] Let~$R>0$ and~$j\in\N$. By~\eqref{BL1-0}, there exists~$k'_j\in\N$
such that~$k'_j\ge j$ and
$$ \int_{B_R}|\chi_{E_{0k'_j}}(x)- \chi_{E_{00}}(x)|\,dx\le\frac1{j}.$$
Then, by~\eqref{BL1}, we have that
there exists~$k''_j\in\N$
such that~$k''_j\ge j$ and
$$ \int_{B_{R/k'_j}} |\chi_{E_{k''_j}}(y)-\chi_{E_0}(y)|\,dy\le \frac{1}{(k'_j)^2\;j}.$$
Consequently, letting~$k_j:=k'_j k''_j$ and
using the change of variable~$y:=x/k'_j$,
\begin{eqnarray*}&&
\int_{B_R}|\chi_{E_{k_j}}(x)- \chi_{E_{00}}(x)|\,dx \\&\le&
\int_{B_R}|\chi_{ k'_j E_{k''_j} }(x)- \chi_{ k'_jE_{0} }(x)|\,dx+
\int_{B_R}|\chi_{k'_jE_0}(x)- \chi_{E_{00}}(x)|\,dx\\&=& (k'_j)^2
\int_{B_{R/k'_j}}|\chi_{ E_{k''_j} }(y)- \chi_{ E_{0} }(y)|\,dy+
\int_{B_R}|\chi_{E_{0k'_j}}(x)- \chi_{E_{00}}(x)|\,dx\\&\le&\frac2j,
\end{eqnarray*}
which gives the desired result.
\end{proof}

\section{Some remarks about linear fractional equations}\label{LIN}

In this appendix we collect a number of auxiliary
results for the solutions of linear fractional equations.
They are probably not completely new in the literature,
since some of them may follow from more general results in~\cites{MR3168912,MR3276603,MR3831283}.
For completeness, we give here precise statements and
self-contained elementary proofs of the results needed in our specialized framework.

To start with, we recall a simple property of~$\sigma$-harmonic
functions with respect to their Dirichlet data. In our framework,
this property will be useful to detect the possible behavior of
nonlocal minimal surfaces at the boundary and distinguish between
sticky and non-sticky points, since we will reduce this alternative
to the analysis of the first nontrivial term of the vertical rescaling
of the solution. The technical and basic result that we exploit is
the following:

\begin{lemma}\label{CLAUDIA}
Let~$\sigma\in(0,1)$, $\alpha\in[0,2\sigma)$ and~$g:\R\to\R$
be such that~$|g(x_1)|\le |x_1|^\alpha$. Let~$u:\R\to\R$ be a solution of
$$ \begin{cases}
(-\Delta)^\sigma u=0 & {\mbox{ in }}(0,1),\\
u=0& {\mbox{ in }}(-\infty,0],\\
u=g&{\mbox{ in }}[1,+\infty).
\end{cases}$$
Then, as~$x_1\searrow0$,
$$ u(x_1)= \bar a x_1^\sigma +O(x_1^{1+\sigma}),$$
for some~$\bar a\in\R$.
\end{lemma}

\begin{proof} We give a quick and self-contained proof
(more general arguments can be found in~\cites{MR3168912, MR3276603},
see also Theorem~6 in~\cite{MR3831283}).
By the  Poisson kernel representation (see e.g.
Theorem~2.10 in~\cite{MR3461641}), we can write, up to normalizing constants,
$$ u(x_1)=\int_1^{+\infty}\left( \frac{1-|x_1-1|^2}{|y|^2-1}\right)^\sigma
\frac{g(y)}{|(x_1-1)-y|}\,dy= x_1^\sigma\,
\int_1^{+\infty}\left( \frac{2-x_1}{|y|^2-1}\right)^\sigma
\frac{g(y)}{|1+y-x_1|}\,dy,$$
which gives the desired result.
\end{proof}

In the proof of Theorem~\ref{BR}, we also need
a sharp boundary regularity result for the fractional Laplacian,
that we state as follows:

\begin{lemma}\label{CLAUDIA2}
Let~$\sigma\in\left(\frac12,1\right)$, $\vartheta\in(0,\sigma]$, $\theta\in(0,\vartheta)$ and~$h\in C^\vartheta([0,2])$.
Let~$u:\R\to\R$ be a solution of
\begin{equation}\label{GREENA} \begin{cases}
(-\Delta)^\sigma u=h & {\mbox{ in }}(0,2),\\
u=0& {\mbox{ in }}\R\setminus(0,2).
\end{cases}\end{equation}
Then, as~$x_1\searrow0$,
\begin{equation}\label{7232340sne} u(x_1)= \bar a x_1^\sigma +O(x_1^{\mu+\sigma}),\end{equation}
for some~$\bar a\in\R$, with~$\mu:=\min\{1,\sigma+\theta\}$.
\end{lemma}

\begin{proof} See formula~(6) in Theorem~4
of~\cite{MR3276603},
%%% Theorem~6 in~\cite{MR3831283},
or formula~(2.13) in Theorem~2.2 of~\cite{MR3293447}.\end{proof}

Another auxiliary result that we need in the proof of Theorem~\ref{BR}
deals with the improved boundary regularity for linear equations:

\begin{lemma} \label{REGONIOA}
Let~$\sigma\in\left(\frac12,1\right)$, $\vartheta\in(0,1)$
and~$\mu:=\min\{1,\vartheta+\sigma\}$.
Assume that
$$ \begin{cases}
(-\Delta)^\sigma u=h & {\mbox{ in }}(0,1),\\
u=0 & {\mbox{ in }}(-\infty,0).\end{cases}$$
Suppose that, for all~$x_1\in\R$,
\begin{equation}\label{I82:AKSvvJJS}
|u(x_1)|\le C\, |x_1|^{\mu+\sigma}\end{equation}
for some~$C>0$
and that~$h\in C^\vartheta([0,1])$.
Then~$u\in C^{\mu+\sigma}\left(\left[0,\frac1{10}\right]\right)$.
\end{lemma}

\begin{proof} We give a direct proof of this result.
Interestingly, we will exploit a Schauder estimate
that we have recently obtained in a fractional Laplace setting
for functions with polynomial growth, which allows us
to treat this case without using any sophisticate method
involving blow-up limits of condition~\eqref{I82:AKSvvJJS}.

We remark that~$\mu+\sigma=
\min\{1+\sigma,\vartheta+2\sigma\}>1$. Consequently,
fixed~$p$, $q\in\left[0,\frac1{10}\right]$, our goal is to show that
\begin{equation}\label{SA:spoe}
|u(p)-u(q)-\ell_q\,(p-q)|\le C|p-q|^{\mu+\sigma},
\end{equation}
for some~$C>0$ and~$\ell_q\in\R$.
To this end, we set~$z:=\frac{p+q}2$ and~$d:=\frac{|p-q|}{2}$.
We observe that~$z\in\left[ 0,\frac1{10}\right]$.

We distinguish two cases: either~$z\in[0,2d]$ or~$z>2d$.
Let us first suppose that~$z\in[0,2d]$.
In this case, $p+q=2z\le4d$ and thus, by~\eqref{I82:AKSvvJJS},
$$  |u(p)-u(q)|\le |u(p)|+|u(q)|\le C(q^{\mu+\sigma}+p^{\mu+\sigma})\le
Cd^{\mu+\sigma}=C|p-q|^{\mu+\sigma},
$$
up to changing~$C>0$ at any step of the computation.
This establishes~\eqref{SA:spoe} in this case (with~$\ell_q:=0$).

Let us now suppose that
\begin{equation}\label{dsc2dz}
z>2d
\end{equation}
and let
$$ v(x_1):= u(dx_1+z)
\qquad{\mbox{and}}\qquad
g(x_1):=d^{2\sigma}h(dx_1+z).$$
Notice that if~$x_1\in (-2,2)$, we have that~$dx_1+z\in
(-2d+z,2d+z)\subseteq\left(0, \frac1{5}\right)$, thanks to~\eqref{dsc2dz},
and accordingly
\begin{equation}\label{BAFinfir-1}
(-\Delta)^\sigma v(x_1)=d^{2\sigma}\,(-\Delta)^\sigma
u(dx_1+z)=d^{2\sigma}\,h(dx_1+z)=g(x_1)
\end{equation}
for every~$ x_1\in (-2,2)$.

Moreover, for any~$x_1$, $y_1\in(-2,2)$,
\begin{equation} \label{BAFinfir-2}
\begin{split}&|g(x_1)-g(y_1)|=
d^{2\sigma}|h(dx_1+z)-h(dy_1+z)|\le 
d^{2\sigma}
[h]_{C^\vartheta([0,1])}\,|dx_1-dy_1|^\vartheta
\\&\qquad\qquad=d^{2\sigma+\vartheta}
[h]_{C^\vartheta([0,1])}\,|x_1-y_1|^\vartheta\le [h]_{C^\vartheta([0,1])}\,|x_1-y_1|^\vartheta
.\end{split}\end{equation}
Furthermore, by~\eqref{I82:AKSvvJJS},
$$ |v(x_1)|\le  C\, |dx_1+z|^{\mu+\sigma} 
\le C\,\left( d^{\mu+\sigma} |x_1|^{\mu+\sigma}+1 \right).
$$
Consequently, since~$2+\sigma-\mu\ge2-\vartheta>1$,
\begin{equation} \label{BAFinfir-3}
\int_\R \frac{|v(x_1)|}{1+|x_1|^{2+2\sigma}}\,dx_1\le
C .
\end{equation}
In view of \eqref{BAFinfir-1}, \eqref{BAFinfir-2} and~\eqref{BAFinfir-3},
we can use the Schauder estimate in Theorem~1.3 of~\cite{REVI}
(exploited here with~$n:=1$ and~$k:=1$). Accordingly, we find that,
for every~$\zeta\in(1,\vartheta+2\sigma)$,
$$ \| v\|_{C^\zeta([-1,1])}\le C.$$
As a consequence,
\begin{eqnarray*}
|u(p)-u(q)-u'(q)\,(p-q)|&=&
d^{\mu+\sigma}\left| v\left( \frac{p-z}{d}\right)-
v\left( \frac{q-z}{d}\right)-v\left( \frac{p-z}{d}\right)\frac{p-q}{d}\right|\\
&\le & d^{\mu+\sigma}\,\| v\|_{C^\zeta([-1,1])}\,\left|\frac{p-q}{d}\right|^{\zeta}\\
&\le& C d^{\mu+\sigma-\zeta}\, |p-q|^\zeta,
\end{eqnarray*}
which proves~\eqref{SA:spoe} also in this case, by choosing~$\zeta:=\mu+\sigma$.
\end{proof}

\section{Uniform convergence to hyperplanes}\label{T5PAO}

In this section, we discuss some general density estimates
and their relation with uniform convergence of blow-up limits
in the Hausdorff distance. The pivotal result in this setting is the following:

\begin{lemma} 
Let~$E$ be $s$-minimal in~$\Omega$ and~$r>0$.
Assume that~$x\in(\partial E)\cap\Omega$. Suppose also that
there exists~$c_o>0$ such that
for all~$y\in E\setminus\Omega$
with~${\rm dist}(y,\partial\Omega)\le r/2$ we have that
\begin{equation}\label{MUjn09-02}
|B_{r/2}(y)\cap E|\ge c_o\,r^n.\end{equation}
Then, there exists~$c>0$, possibly depending on~$c_o$,
such that
$$ |B_r(x)\cap E| \ge c\,r^n.$$
\end{lemma}

\begin{proof} We recall that a similar result was proved in Theorem~4.1 of~\cite{MR2675483}
when condition~\eqref{MUjn09-02} is replaced
by the stronger condition that~$B_r(x)\cap E\subset\Omega$.
Then, in our setting, we distinguish two cases.
First, if~$B_{r/2}(x)\cap E\subset\Omega$,
we can apply Theorem~4.1 of~\cite{MR2675483}
(with~$r/2$ instead of~$r$) and conclude that~$|B_{r/2}(x)\cap E|
\ge c\,(r/2)^n$, from which we plainly obtain the desired result,
up to renaming~$c>0$.

If instead~$B_{r/2}(x)\cap E\not\subset\Omega$,
we take~$y\in \big(B_{r/2}(x)\cap E\big)\setminus\Omega$.
We are in the position of using~\eqref{MUjn09-02} and thus obtain that
$$
|B_{r}(x)\cap E|\ge
|B_{r/2}(y)\cap E|\ge c_o\,r^n,$$
which gives the desired result, up to renaming constants.
\end{proof}

{F}rom this and Lemma~\ref{Sempli}, we easily conclude that:

\begin{corollary}\label{Sempli-APP}
In the notation of Lemma~\ref{Sempli}, we have that,
up to a subsequence,
$E_{k}$ converges to~$E_{00}$ locally in the Hausdorff distance.
\end{corollary}

\section{Proof of formula~\eqref{BSC}}\label{TED}

This section is devoted to the proof of a technical statement needed in Theorem~\ref{BR}.

\begin{proof}[Proof of~\eqref{BSC}]
We consider separately~$f_1$, $f_2$ and~$f_3$, as given by~\eqref{CHA823r44rfr}.
To estimate~$f_1$, we observe that, using~\eqref{KS-OAKS92348}
with~$J:=\R\setminus I$,
\begin{equation}\label{8UJS9idjveiiri} f_1(x_1)=-
\int_{\R\setminus I} F\left(\frac{u(x_1+ t)-u(x_1)}{|t|}
\right)\,\frac{dt}{|t|^{1+s}}.
\end{equation}
Consequently,
\begin{equation}\label{F1est1lin}
\| f_1\|_{L^\infty\left(\left[0,\frac{h}{4}\right]\right)}\le
\| F\|_{L^\infty(\R)}\,\int_{\R\setminus I} \frac{dt}{|t|^{1+s}}\le C\,\| F\|_{L^\infty(\R)},
\end{equation}
with~$C>0$ depending only on~$s$ and~$h$.

Now, recalling~\eqref{DEF:F}, we define
$$ \Phi(r):=rF'(r)=\frac{r}{(1+r^2)^{\frac{2+s}2}} .$$
%%%	We observe that
%%%	$$ \|\Phi'\|_{L^\infty(\R)}=\sup_{r\in\R}\left|\frac{
%%%	1 - (1 + s)r^2}{ (1 + r^2)^{\frac{4+s}2}}
%%%	\right|<+\infty.$$
Then, changing variable in~\eqref{8UJS9idjveiiri}, we see that
\begin{eqnarray*}
f_1'(x_1)&=&-\frac{d}{dx_1}
\int_{\vartheta\in(-\infty,x_1-\frac{h}4)\cup(x_1+\frac{h}4,+\infty)} F\left(\frac{u(\vartheta)-u(x_1)}{|x_1-\vartheta|}
\right)\,\frac{d\vartheta}{|x_1-\vartheta|^{1+s}}\\
&=&
\int_{\{|x_1-\vartheta|>h/4\}} F'\left(\frac{u(\vartheta)-u(x_1)}{|x_1-\vartheta|}
\right)\,
\left( \frac{u'(x_1)}{|x_1-\vartheta|^{2+s}}+\frac{(u(\vartheta)-u(x_1))(x_1-\vartheta)
}{|x_1-\vartheta|^{3+s}}\right)
\,d\vartheta
\\&&\qquad+(1+s)
\int_{\{|x_1-\vartheta|>h/4\}} F\left(\frac{u(\vartheta)-u(x_1)}{|x_1-\vartheta|}
\right)\,\frac{(x_1-\vartheta)\,d\vartheta}{|x_1-\vartheta|^{3+s}}\\&&\qquad
+\left[
F\left(\frac{4\,\left(u\left(x_1+\frac{h}4\right)-u(x_1)\right)}{h}
\right)
-F\left(\frac{4\,\left( u\left(x_1-\frac{h}4\right)-u(x_1)\right)}{h}\right)
\right]\frac{4^{1+s}}{h^{1+s}}\\
&=&u'(x_1)\,
\int_{\{|x_1-\vartheta|>h/4\}} F'\left(\frac{u(\vartheta)-u(x_1)}{|x_1-\vartheta|}
\right)\,
\frac{d\vartheta}{|x_1-\vartheta|^{2+s}}
\\
&&\qquad+
\int_{\{|x_1-\vartheta|>h/4\}} \Phi\left(\frac{u(\vartheta)-u(x_1)}{|x_1-\vartheta|}
\right)\,
\frac{(x_1-\vartheta)\,d\vartheta
}{|x_1-\vartheta|^{2+s}}
\\&&\qquad+(1+s)
\int_{\{|x_1-\vartheta|>h/4\}} F\left(\frac{u(\vartheta)-u(x_1)}{|x_1-\vartheta|}
\right)\,\frac{(x_1-\vartheta)\,d\vartheta}{|x_1-\vartheta|^{3+s}}\\&&\qquad+\left[
F\left(\frac{4\,\left(u\left(x_1+\frac{h}4\right)-u(x_1)\right)}{h}
\right)
-F\left(\frac{4\,\left( u\left(x_1-\frac{h}4\right)-u(x_1)\right)}{h}\right)
\right]\frac{4^{1+s}}{h^{1+s}}.
\end{eqnarray*}
Consequently,
\begin{equation}\label{Completaf1}\begin{split}
|f_1'(x_1)|\;&\le\|u'\|_{L^\infty\left(\left[0,\frac{h}{2}\right]\right)}
\|F'\|_{L^\infty(\R)}
\int_{\{|x_1-\vartheta|>h/4\}} 
\frac{d\vartheta}{|x_1-\vartheta|^{2+s}}
\\
&\qquad+\|\Phi\|_{L^\infty(\R)}
\int_{\{|x_1-\vartheta|>h/4\}} 
\frac{d\vartheta
}{|x_1-\vartheta|^{1+s}}
\\&\qquad+(1+s)\,\|F\|_{L^\infty(\R)}
\int_{\{|x_1-\vartheta|>h/4\}}
\frac{d\vartheta}{|x_1-\vartheta|^{2+s}}\\
&\qquad+\frac{4^{2+s}\,\|F\|_{L^\infty(\R)}}{h^{1+s}}
\\
&\le C\,\big(\|u'\|_{L^\infty\left(\left[0,\frac{h}{2}\right]\right)}+1\big),
\end{split}\end{equation}
with~$C>0$ depending only on~$s$ and~$h$.

Having completed the desired estimate for~$f_1$, we now focus on~$f_2$.
For this, we remark that, since~$F'$ is an even function,
if~$x_1$, $y_1\in\left[0,\frac{h}{4}\right]$
and~$t\in I$,
\begin{eqnarray*}&&
\big| a(x_1,t)-F'\big( u'(x_1)\big)\big|\\
&\le&
\int_0^1\left|
F'\left(\frac{\tau \big(u(x_1+ t)-u(x_1)\big)+ (1-\tau)u'(x_1)\,t}{|t|}
\right)-F'\big( u'(x_1)\big)\right|\,d\tau\\
&=&
\int_0^1\left|
F'\left(\frac{\tau \big(u(x_1+ t)-u(x_1)-u'(x_1)t\big)+u'(x_1)\,t}{|t|}
\right)-F'\left(\frac{u'(x_1)\,t}{|t|}
\right)\right|\,d\tau\\
&\le& \| F''\|_{L^\infty(\R)}\;
\frac{ \big|u(x_1+ t)-u(x_1)-u'(x_1)t\big|}{|t|}\\&\le&
\| F''\|_{L^\infty(\R)}\;\|u\|_{C^{1,\alphyy}\left(\left[0,\frac{h}{2}\right]\right)}\,|t|^{\alphyy}.
\end{eqnarray*}
Therefore, setting
\begin{equation}\label{9.16bis} {\mathcal{K}}(x_1,t):=
\frac{F'\big( u'(x_1)\big)-a(x_1,t)}{|t|^{2+s}},\end{equation}
we conclude that
\begin{equation}\label{EXAUDJC239}
\big| {\mathcal{K}}(x_1,t)\big|\le \| F''\|_{L^\infty(\R)}\;\|u\|_{C^{1,\alphyy}\left(\left[0,\frac{h}{2}\right]\right)}\,|t|^{\alphyy-s-2}.
\end{equation}
We also observe that
$$ \big|\delta(x_1,t)\big|\le\|u\|_{C^{1,\alphyy}\left(\left[0,\frac{h}{2}\right]\right)}\,|t|^{1+\alphyy}.$$
{F}rom this and~\eqref{EXAUDJC239}, we conclude that
\begin{equation}\label{2qasduquuasd}
\begin{split}&
\| f_2\|_{L^\infty\left(\left[0,\frac{h}{4}\right]\right)}\le
\sup_{x_1\in\left[0,\frac{h}{4}\right]}
\int_{I} \big|
\delta(x_1,t)\big|\,\big| {\mathcal{K}}(x_1,t)\big|\,dt\\&\qquad\le
\| F''\|_{L^\infty(\R)}\;\|u\|_{C^{1,\alphyy}\left(\left[0,\frac{h}{2}\right]\right)}^2
\int_I\,|t|^{2\alphyy-s-1}\le C,\end{split}
\end{equation}
thanks to~\eqref{RESTRIZ}, with~$C>0$ depending only
on~$s$, $\alphyy$, $h$ and~$
\|u\|_{C^{1,\alphyy}\left(\left[0,\frac{h}{2}\right]\right)}$.

We also remark that, for all~$x_1$, $y_1\in\left[0,\frac{h}4\right]$ and~$t\in I$,
\begin{equation}\label{Dearujerfba}
\big|\delta(x_1,t)-\delta(y_1,t)\big|\le5\,\|u\|_{C^{1,\alphyy}\left(\left[0,\frac{h}{2}\right]\right)}\,|t|\,
\min\{ |t|^{\alphyy} , |x_1-y_1|^{\alphyy} \}.
\end{equation}
To check this, we write
\begin{equation}\label{Dearujerfba2}
\big|\delta(x_1,t)-\delta(y_1,t)\big|=\big| u(x_1+t)-u(x_1)-u'(x_1)\,t-
u(y_1+t)+u(y_1)+u'(y_1)\,t\big|,
\end{equation}
and we distinguish two cases. When~$|t|\le|x_1-y_1|$, we observe that
\begin{eqnarray*}
\big| u(x_1+t)-u(x_1)-u'(x_1)\,t\big|\le
\|u\|_{C^{1,\alphyy}\left(\left[0,\frac{h}{2}\right]\right)}\,|t|^{1+\alphyy},
\end{eqnarray*}
and a similar estimate holds true with~$y_1$ replacing~$x_1$.
Hence, we deduce from~\eqref{Dearujerfba2}
that
\begin{eqnarray*}
\big|\delta(x_1,t)-\delta(y_1,t)\big|
\le2\,\|u\|_{C^{1,\alphyy}\left(\left[0,\frac{h}{2}\right]\right)}\,|t|^{1+\alphyy}
,\end{eqnarray*}
which gives~\eqref{Dearujerfba} in this case.

If instead~$|t|>|x_1-y_1|$, we assume without loss of generality that~$x_1\le y_1$,
and
we exploit~\eqref{Dearujerfba2} to write that
\begin{eqnarray*}&&
\big|\delta(x_1,t)-\delta(y_1,t)\big|\\
&=&
\big| u(x_1+t)-u(y_1+t)+u(y_1)-u(x_1)+u'(y_1)\,t-u'(x_1)\,t\big|
\\ &\le&
\big| u(x_1+t)-u(y_1+t)-u'(y_1+t)(x_1-y_1)\big|
\\&&\qquad+\big|u(y_1)-u(x_1)-u'(x_1)(y_1-x_1)\big|\\&&\qquad+|t|\,\big|
u'(y_1)-u'(x_1)\big|+\big|u'(y_1+t)-u'(x_1)\big|\,|x_1-y_1|\\
&\le& 2\,\|u\|_{C^{1,\alphyy}\left(\left[0,\frac{h}{2}\right]\right)}|x_1-y_1|^{1+\alphyy}
+\|u\|_{C^{1,\alphyy}\left(\left[0,\frac{h}{2}\right]\right)}\,|t|\,|x_1-y_1|^{\alphyy}\\
&&\qquad+
\big|u'(y_1+t)-u'(y_1)\big|\,|x_1-y_1|
+
\big|u'(y_1)-u'(x_1)\big|\,|x_1-y_1|\\
&\le& 3\,\|u\|_{C^{1,\alphyy}\left(\left[0,\frac{h}{2}\right]\right)}|x_1-y_1|^{1+\alphyy}
+\|u\|_{C^{1,\alphyy}\left(\left[0,\frac{h}{2}\right]\right)}\,|t|\,|x_1-y_1|^{\alphyy}+\|u\|_{C^{1,\alphyy}\left(\left[0,\frac{h}{2}\right]\right)}\,|t|^{\alphyy}\,|x_1-y_1|
\\&\le&5\,\|u\|_{C^{1,\alphyy}\left(\left[0,\frac{h}{2}\right]\right)}\,|t|\,|x_1-y_1|^{\alphyy}\end{eqnarray*}
which establishes~\eqref{Dearujerfba} also in this case.

Consequently, exploiting~\eqref{EXAUDJC239} and~\eqref{Dearujerfba},
and recalling~\eqref{RESTRIZ},
\begin{equation}\label{6zwe73737HS}
\begin{split}&
\left| \int_I \big( \delta(x_1,t)-\delta(y_1,t)\big) {\mathcal{K}}(x_1,t)\right|
\\ \le\;&5\,\| F''\|_{L^\infty(\R)}\;\|u\|_{C^{1,\alphyy}\left(\left[0,\frac{h}{2}\right]\right)}^2\,\int_I
|t|^{\alphyy-s-1}
\min\{ |t|^{\alphyy} , |x_1-y_1|^{\alphyy} \}\,dt\\
\le\; & C_{\alphyy,s}\,|x_1-y_1|^{\min\{\alphyy,\,2\alphyy-s\}},
\end{split}
\end{equation}
for some~$C_{\alphyy,s}>0$ depending only on~$s$, $\alphyy$, $h$, and~$\|u\|_{C^{1,\alphyy}\left(\left[0,\frac{h}{2}\right]\right)}$.

Now, if~$x_1$, $y_1\in\left[0,\frac{h}{4}\right]$,
$\tau\in(0,1)$, $t\in I$, we set
\begin{eqnarray*}
G(t,\tau,x_1,y_1)&:=&
F'\left( u'(x_1)\right)-
F'\left(\frac{\tau \big(u(x_1+ t)-u(x_1)-u'(x_1)\,t\big)+u'(x_1)\,t}{|t|}\right)\\&&\qquad
-F'\left(u'(y_1)\right)+
F'\left(\frac{\tau \big(u(y_1+ t)-u(y_1)-u'(y_1)\,t\big)+u'(y_1)\,t}{|t|}
\right).\end{eqnarray*}
To ease the notation, we write~$G(t):=G(t,\tau,x_1,y_1)$.
Using that~$F'$ is even, we see that
$$ G(0):=\lim_{t\to0} G(t)=0,$$
and, if~$t>0$,
\begin{eqnarray*} G(t)&=&
F'\left( u'(x_1)\right)-
F'\left(\frac{\tau \big(u(x_1+ t)-u(x_1)-u'(x_1)\,t\big)}{t}+u'(x_1)\right)\\&&\qquad
-F'\left(u'(y_1)\right)+
F'\left(\frac{\tau \big(u(y_1+ t)-u(y_1)-u'(y_1)\,t\big)}{t}
+u'(y_1)\right),\end{eqnarray*}
which gives that
\begin{equation}\label{G22}\begin{split} t^2 G'(t)\;&=
-\tau F''\left(\frac{\tau \big(u(x_1+ t)-u(x_1)-u'(x_1)\,t\big)}{t}+u'(x_1)\right)\\&\qquad\times
\Big(u'(x_1+ t)\,t-u(x_1+ t)+u(x_1)\Big)
\\&+\tau
F''\left(\frac{\tau \big(u(y_1+ t)-u(y_1)-u'(y_1)\,t\big)}{t}
+u'(y_1)\right)\\&\qquad\times
\Big(u'(y_1+ t)\,t-u(y_1+ t)+u(y_1)\Big).\end{split}\end{equation}
Now, we claim that, if~$t\in\left(0,\frac{h}4\right)$,
\begin{equation}\label{2cas}
|G'(t)|\le \frac{ C_{\alphyy,s} }{t}\,\min\{ t^{\alphyy}, |x_1-y_1|^{\alphyy}\},
\end{equation}
for some~$C_{\alphyy,s}>0$ depending only on~$s$, $\alphyy$ and~$\|u\|_{C^{1,\alphyy}\left(\left[0,\frac{h}{2}\right]\right)}$.

For this, we observe that
\begin{equation}\label{DESRIMAL}
\begin{split}&
\Big|u(x_1+ t)-u(x_1)-u'(x_1+ t)\,t\Big|\\ \le\;&
\Big|u(x_1+ t)-u(x_1)-u'(x_1)\,t\Big|+t\,
\Big|u'(x_1)-u'(x_1+ t)\Big|\\ \le\;&
2\,\|u\|_{C^{1,\alphyy}\left(\left[0,\frac{h}{2}\right]\right)}\,t^{1+\alphyy},
\end{split}\end{equation}
and a similar estimate holds true for~$x_1$ replaced by~$y_1$.

Now, we distinguish two cases. If~$t\le|x_1-y_1|$, we exploit~\eqref{DESRIMAL}
and we obtain that
\begin{equation}
t^2 |G'(t)|\le 4\,\|F''\|_{L^\infty(\R)}\,\|u\|_{C^{1,\alphyy}\left(\left[0,\frac{h}{2}\right]\right)}\,t^{1+\alphyy}
,\end{equation}
which proves~\eqref{2cas} in this case.

If instead~$t>|x_1-y_1|$, we observe that, for all~$A$, $B$, $C$, $D\in\R$,
\begin{equation}\label{PInavvot}
\begin{split}&
\big| F''(A)B-F''(C)D\big|\le
\|F''\|_{L^\infty(\R)}\,|B-D|+\big| F''(A)-F''(C)\big|\,|D|\\&\qquad\qquad\le
\|F''\|_{C^1(\R)}\,\big(|B-D|+|D|\,|A-C|\big).
\end{split}\end{equation}
We apply this estimate with
\begin{eqnarray*}&& A:=\frac{\tau \big(u(x_1+ t)-u(x_1)-u'(x_1)\,t\big)}{t}+u'(x_1)
,\qquad B:=u'(x_1+ t)\,t-u(x_1+ t)+u(x_1),
\\&& C:=\frac{\tau \big(u(y_1+ t)-u(y_1)-u'(y_1)\,t\big)}{t}
+u'(y_1),\qquad D:= u'(y_1+ t)\,t-u(y_1+ t)+u(y_1).\end{eqnarray*}
In this case, assuming, without loss
of generality, that~$x_1\le y_1$, we observe that
\begin{eqnarray*}&&
t\,|A-C|\\&=&
\Big|
\tau \big(u(x_1+ t)-u(y_1+t)-u(x_1)+u(y_1)-u'(x_1)\,t+
u'(y_1)\,t\big)+u'(x_1)\,t-u'(y_1)\,t
\Big|\\&\le&
\Big| u(x_1+ t)-u(y_1+t)-u(x_1)+u(y_1)-u'(x_1)\,t+
u'(y_1)\,t\Big|+\Big|u'(x_1)-u'(y_1)\Big|\,t
\\ &\le&
\left| \int_0^t \Big(u'(x_1+ \theta)-u'(y_1+\theta)\Big)\,d\theta
-u'(x_1)\,t+
u'(y_1)\,t\right|+\|u\|_{C^{1,\alphyy}\left(\left[0,\frac{h}{2}\right]\right)}\,|x_1-y_1|^{\alphyy}\,t\\ &\le&
\int_0^t \Big|u'(x_1+ \theta)-u'(y_1+\theta)\Big|\,d\theta
+2\,\|u\|_{C^{1,\alphyy}\left(\left[0,\frac{h}{2}\right]\right)}\,|x_1-y_1|^{\alphyy}\,t\\
&\le&3\,\|u\|_{C^{1,\alphyy}\left(\left[0,\frac{h}{2}\right]\right)}\,|x_1-y_1|^{\alphyy}\,t.
\end{eqnarray*}
Similarly,
\begin{eqnarray*}
|B-D|&=&\Big|
u'(x_1+ t)\,t-u'(y_1+t)\,t-u(x_1+ t)+u(y_1+t)+u(x_1)-u(y_1)\Big|\\
&=&\left|
u'(x_1+ t)\,t-u'(y_1+t)\,t+\int_0^t\Big( u'(y_1+\theta)-u'(x_1+\theta)\Big)\,d\theta\right|\\&\le&
2\,\|u\|_{C^{1,\alphyy}\left(\left[0,\frac{h}{2}\right]\right)}\,|x_1-y_1|^{\alphyy}\,t,
\end{eqnarray*}
and~$|D|\le 2\,\|u\|_{C^{1,{\alphyy}}\left(\left[0,\frac{h}{2}\right]\right)}\,t^{1+{\alphyy}}$, thanks to~\eqref{DESRIMAL}.

In view of these observations, we thus find that
\begin{eqnarray*} |B-D|+|D|\,|A-C|
&\le& 2\,\|u\|_{C^{1,{\alphyy}}\left(\left[0,\frac{h}{2}\right]\right)}\,|x_1-y_1|^{\alphyy}\,t+
6\,\|u\|_{C^{1,{\alphyy}}\left(\left[0,\frac{h}{2}\right]\right)}^2\,|x_1-y_1|^{\alphyy}\,t^{1+{\alphyy}}\\&
\le &C_{\alphyy}\,|x_1-y_1|^{\alphyy}\,t,\end{eqnarray*}
with~$C_{\alphyy}>0$ only depending on~$\|u\|_{C^{1,{\alphyy}}\left(\left[0,\frac{h}{2}\right]\right)}$.

Using this, \eqref{G22} and~\eqref{PInavvot}, we conclude that
\begin{eqnarray*}
t^2 |G'(t)| \le
\|F''\|_{C^1(\R)}\,\Big(|B-D|+|D|\,|A-C|\Big)\le
C_{{\alphyy},s}\,|x_1-y_1|^{\alphyy}\,t,
\end{eqnarray*}
for some~$C_{{\alphyy},s}>0$,
and this completes the proof of~\eqref{2cas}.

Therefore, if~$t>0$,
\begin{eqnarray*}&&
|G(t)|=|G(t)-G(0)|\le \int_{0}^{t} |G'(\theta)|\,d\theta
\le\int_{0}^{t}
\frac{ C_{{\alphyy},s} }{\theta}\,\min\{ \theta^{\alphyy}, |x_1-y_1|^{\alphyy}\}
\,d\theta\\&&\qquad\qquad\qquad\le C_{{\alphyy},s}\,
\min\{ t^{\alphyy}, |x_1-y_1|^{\alphyy}\}\,\left(1+
\left(\log\frac{t}{|x_1-y_1|}\right)_+
\right),
\end{eqnarray*}
up to renaming~$C_{{\alphyy},s}$. Arguing for~$t<0$ in a similar
way, we obtain an estimate valid for all~$t\in I$ (in which~$t$
on the right hand side is replaced by~$|t|$).

With this, recalling~\eqref{9.16bis}, we conclude that
\begin{eqnarray*}
&& \big|{\mathcal{K}}(x_1,t)-{\mathcal{K}}(y_1,t)\big|\\&=&\left|\frac{F'\big( u'(x_1)\big)}{|t|^{2+s}}-
\frac{ a(x_1,t) }{|t|^{2+s}}-
\frac{F'\big( u'(y_1)\big)}{|t|^{2+s}}+
\frac{ a(y_1,t) }{|t|^{2+s}}\right|\\&=&
\frac1{|t|^{2+s}}\,\left|
F'\big( u'(x_1)\big)-
\int_0^1 F'\left(\frac{\tau \big(u(x_1+ t)-u(x_1)\big)+ (1-\tau)u'(x_1)\,t}{|t|}
\right)\,d\tau\right.\\&&\left.\qquad-F'\big( u'(y_1)\big)+
\int_0^1 F'\left(\frac{\tau \big(u(y_1+ t)-u(y_1)\big)+ (1-\tau)u'(y_1)\,t}{|t|}
\right)\,d\tau\right|\\&=&
\frac1{|t|^{2+s}}\,\left|\int_0^1 G(t,\tau,x_1,y_1)\,d\tau\right|\\&\le&
\frac{C_{{\alphyy},s}}{|t|^{2+s}}\,
\min\{ |t|^{\alphyy}, |x_1-y_1|^{\alphyy}\}\,\left(1+
\left(\log\frac{|t|}{|x_1-y_1|}\right)_+
\right).
\end{eqnarray*}
As a consequence, up to renaming~$C_{{\alphyy},s}$,
\begin{equation}\label{LOGIN0}
\begin{split}&
\left|\int_I \delta u(x_1,t)\,\big({\mathcal{K}}(x_1,t)-{\mathcal{K}}(y_1,t)\big)\,dt\right|
\\ \le\;& C_{{\alphyy},s}\,
\int_I |t|^{{\alphyy}-1-s}\,\min\{ |t|^{\alphyy}, |x_1-y_1|^{\alphyy}\}\,\left(1+
\left(\log\frac{|t|}{|x_1-y_1|}\right)_+
\right)\,dt\\
\le\;& C_{{\alphyy},s}\,\left(
|x_1-y_1|^{\min\{{\alphyy},\,2{\alphyy}-s\}}+
|x_1-y_1|^{{\alphyy}}
\int_{t\in(|x_1-y_1|,h/4)} t^{{\alphyy}-1-s}\log\frac{t}{|x_1-y_1|}\,dt\right).
\end{split}\end{equation}
Now we claim that
\begin{equation}\label{LOGIN}
|x_1-y_1|^{{\alphyy}}
\int_{t\in(|x_1-y_1|,h/4)} t^{{\alphyy}-1-s}\log\frac{t}{|x_1-y_1|}\,dt\le
C\,|x_1-y_1|^{\kappa({\alphyy})},\end{equation}
for some~$C>0$ depending on~${\alphyy}$, $s$ and~$h$.
Indeed, if~${\alphyy}<s$ we
integrate the logarithm by parts and we obtain that
\begin{eqnarray*}&&
\int_{t\in(|x_1-y_1|,h/4)} t^{{\alphyy}-1-s}\log\frac{t}{|x_1-y_1|}\,dt=
\frac1{{\alphyy}-s}
\int_{t\in(|x_1-y_1|,h/4)} \frac{d}{dt} \Big(
t^{{\alphyy}-s}\Big)\,\log\frac{t}{|x_1-y_1|}\,dt\\&&\qquad=
\frac1{{\alphyy}-s}\left(
\left(\frac{h}4\right)^{{\alphyy}-s}\log\frac{h}{4\,|x_1-y_1|}
-\int_{t\in(|x_1-y_1|,h/4)}
t^{{\alphyy}-s-1}\,dt\right)\\
&&\qquad=\frac1{{\alphyy}-s}\left(
\left(\frac{h}4\right)^{{\alphyy}-s}\log\frac{h}{4\,|x_1-y_1|}
-\frac{1}{{\alphyy}-s}\left( \left( \frac{h}4\right)^{{\alphyy}-s}-
|x_1-y_1|^{{\alphyy}-s}\right)\right)\\
&&\qquad\le\frac{|x_1-y_1|^{{\alphyy}-s}}{({\alphyy}-s)^2},
\end{eqnarray*}
and this proves~\eqref{LOGIN} in this case.

If instead~${\alphyy}\ge s$, we have from~\eqref{RESTRIZ} that~${\alphyy}>s$.
Hence, we see that, in this case,
\begin{eqnarray*}&&
\int_{t\in(|x_1-y_1|,h/4)} t^{{\alphyy}-1-s}\log\frac{t}{|x_1-y_1|}\,dt=
|x_1-y_1|^{{\alphyy}-s}
\int_{1}^{\frac{h}{4 |x_1-y_1|} }T^{{\alphyy}-1-s}\log T\,dT\\&&\qquad\le C\,
|x_1-y_1|^{{\alphyy}-s}
\int_{1}^{\frac{h}{4 |x_1-y_1|} }T^{\underline{\alphyy}-1-s}\,dT\le
\frac{C\,|x_1-y_1|^{{\alphyy}-s}}{\underline{\alphyy}-s} \left(
{\frac{h}{4 |x_1-y_1|} }
\right)^{\underline{\alphyy}-s} \\&&\qquad=
\frac{C \,h^{\underline{\alphyy}-s}}{\underline{\alphyy}-s}\,
|x_1-y_1|^{{\alphyy}-\underline{{\alphyy}}},
\end{eqnarray*}
with~$C>0$,
which, together with the definition of~$\kappa({\alphyy})$
in~\eqref{BSC}, completes the proof of~\eqref{LOGIN}.

Then, inserting~\eqref{LOGIN} into~\eqref{LOGIN0}, we conclude that
$$ \left|\int_I \delta u(x_1,t)\,\big({\mathcal{K}}(x_1,t)-{\mathcal{K}}(y_1,t)\big)\,dt\right|
\le C_{{\alphyy},s}\,
|x_1-y_1|^{\kappa({\alphyy})},$$
up to renaming~$C_{{\alphyy},s}$.

Then, the latter estimate and~\eqref{6zwe73737HS} give that
\begin{equation}\label{LASTIMASUf2}
\begin{split}&
|f_2(x_1)-f_2(y_1)|\\ =\;&\left|
\int_I \Big(\delta u(x_1,t)\,{\mathcal{K}}(x_1,t)
-\delta u(y_1,t)\,{\mathcal{K}}(y_1,t)\Big)\,dt
\right|\\ \le\;&
\left|
\int_I \big(\delta u(x_1,t)-\delta u(y_1,t)\big)\,{\mathcal{K}}(x_1,t)\Big)\,dt
\right|
+\left|
\int_I  \delta u(y_1,t)\,\big( {\mathcal{K}}(x_1,t)
-{\mathcal{K}}(y_1,t)\big)\,dt
\right|\\ \le\;&C_{{\alphyy},s}\,
|x_1-y_1|^{\kappa({\alphyy})},
\end{split}
\end{equation}
up to renaming~$C_{{\alphyy},s}$.

Now we estimate~$f_3$. To this end, we observe that
\begin{equation}\label{GIOAchdfuwjs}
\|f_3\|_{L^\infty\left(\left[0,\frac{h}{4}\right]\right)}\le
\|F'\|_{L^\infty(\R)}\,
\int_{\R\setminus I} \Big(2\|u\|_{L^\infty\left([0,h]\right)}+
\|u'\|_{L^\infty\left(\left[0,\frac{h}{2}\right]\right)}\,|t|\Big)
\,\frac{ dt}{|t|^{2+s}}\le C_{{\alphyy},s},
\end{equation}
up to renaming~$C_{{\alphyy},s}$.

We also point out that~$\|v\|_{C^1(\R)}\le
\|u\|_{C^1\left(\left[0,\frac{h}{2}\right]\right)}$.
Therefore, setting
$$ V(x_1):=
\int_{\R\setminus I} \Big(v(x_1+ t)-u(x_1)-u'(x_1)\,t\Big)
\,\frac{ dt}{|t|^{2+s}},$$
we see that
\begin{equation}\label{89-29-1}
\begin{split}&
|V(x_1)-V(y_1)|\\ \le\;&
\int_{\R\setminus I} \Big(\|v\|_{C^1(\R)}|x_1-y_1|+\|u\|_{C^1\left(\left[0,\frac{h}{2}\right]\right)}
|x_1-y_1|+
\|u\|_{C^{1,{\alphyy}}\left(\left[0,\frac{h}{2}\right]\right)}
|x_1-y_1|^{\alphyy}\,|t|\Big)
\,\frac{ dt}{|t|^{2+s}}\\
\le\;&C_{{\alphyy},s}\, |x_1-y_1|^{\alphyy}
\end{split}
\end{equation}
up to renaming~$C_{{\alphyy},s}$.

Also, as in~\eqref{GIOAchdfuwjs}, one can write that
\begin{equation}\label{89-29-2}
\|V\|_{L^\infty\left(\left[0,\frac{h}{4}\right]\right)}\le C_{{\alphyy},s}.
\end{equation}
In addition,
\begin{eqnarray*}&&
\big| F'\big( u'(x_1)\big)-F'\big( u'(y_1)\big)\big|\le
\| F''\|_{L^\infty(\R)}\,\left| u'(x_1)-u'(y_1)\right|\\&&\qquad
\le
\| F''\|_{L^\infty(\R)}\,
\|u\|_{C^1\left(\left[0,\frac{h}{2}\right]\right)}\,|x_1-y_1|^{\alphyy}
.\end{eqnarray*}
{F}rom this, \eqref{89-29-1} and~\eqref{89-29-2}, we see that
\begin{eqnarray*}
|f_3(x_1)-f_3(y_1)|&=&\big|F'\big( u'(x_1)\big)\,V(x_1)-F'\big( u'(y_1)\big)\,V(y_1)\big|\\
&\le& \big|F'\big( u'(x_1)\big)-F'\big( u'(y_1)\big)\big|\,|V(x_1)|+
\big|F'\big( u'(y_1)\big)\big|\,\big|V(x_1)-V(y_1)\big|\\&\le& C_{{\alphyy},s}\,|x_1-y_1|^{{\alphyy}},
\end{eqnarray*}
up to renaming~$C_{{\alphyy},s}$.

Using this, \eqref{F1est1lin}, \eqref{Completaf1},
\eqref{2qasduquuasd}, \eqref{LASTIMASUf2} and~\eqref{GIOAchdfuwjs},
we obtain~\eqref{BSC}, as desired.
\end{proof}

\end{appendix}

\section*{Acknowledgments}

The first and third authors are member of INdAM and AustMS, and they
are supported by the Australian Research Council
Discovery Project DP170104880 NEW ``Nonlocal Equations at Work''.
The first author's visit to Columbia has been partially funded by
the Fulbright Foundation
and the Australian Research Council DECRA DE180100957
``PDEs, free boundaries and applications''. The second author is supported
by the National Science Foundation grant DMS-1500438.

\end{document}